\documentclass[10pt,letterpaper]{amsart}

\usepackage{amsmath}
\usepackage{amssymb}
\usepackage{amsthm}
\usepackage{amsfonts}
\usepackage{microtype}
\usepackage{mathrsfs}
\usepackage{graphicx}
\usepackage{color}
\usepackage{latexsym}
\usepackage{rotating}
\usepackage{xspace}
\usepackage[all]{xy}
\usepackage{longtable}
\usepackage{leftidx}
\usepackage{mathtools}
\usepackage{titlesec}
\usepackage{setspace}
\usepackage{tikz}
\usepackage{tikz-cd}
\usepackage{url}
\usepackage[autostyle]{csquotes}
\usetikzlibrary{calc}
\usetikzlibrary{arrows}
\usetikzlibrary{decorations.markings,intersections}
\usepackage{geometry}
\usepackage[final]{pdfpages}
\usepackage{fancyhdr}
\usepackage{tabularx}

\newtheorem{theorem}{Theorem}[section]
\numberwithin{equation}{theorem}
\newtheorem{lemma}[theorem]{Lemma}
\newtheorem{proposition}[theorem]{Proposition}
\newtheorem{corollary}[theorem]{Corollary}

\theoremstyle{definition}
\newtheorem{definition}[theorem]{Definition}
\newtheorem{example}[theorem]{Example}

\newtheorem{remark}[theorem]{Remark}

\theoremstyle{conjecture}

\newcommand{\Ker}{\operatorname{Ker}}
\newcommand{\Coker}{\operatorname{Coker}}
\newcommand{\im}{\operatorname{Im}}

\newcommand{\Cone}{\operatorname{Cone}}
\newcommand{\Ext}{\operatorname{Ext}}

\newcommand{\Tor}{\operatorname{Tor}}
\newcommand{\Hom}{\operatorname{Hom}}

\newcommand{\Arr}{\operatorname{Arr}}
\newcommand{\Obj}{\operatorname{Obj}}
\newcommand{\Mor}{\operatorname{Mor}}
\newcommand{\LLP}{\operatorname{LLP}}
\newcommand{\RLP}{\operatorname{RLP}}
\newcommand{\op}{\operatorname{op}}
\newcommand{\Nor}{\operatorname{Nor}}
\newcommand{\DK}{\operatorname{DK}}
\newcommand{\Map}{\operatorname{Map}}
\newcommand{\Sh}{\operatorname{Sh}}
\newcommand{\sgn}{\operatorname{sgn}}
\newcommand{\nrv}{\operatorname{nrv}}
\newcommand{\Cyl}{\operatorname{Cyl}}
\newcommand{\Path}{\operatorname{Path}}
\newcommand{\Sym}{\operatorname{Sym}}
\newcommand{\Ho}{\operatorname{Ho}}
\newcommand{\Fun}{\operatorname{Fun}}
\newcommand{\Der}{\operatorname{Der}}
\newcommand{\LKan}{\operatorname{LKan}}
\newcommand{\RKan}{\operatorname{RKan}}

\newcommand{\suchthat}{\;\ifnum\currentgrouptype=16 \middle\fi|\;}

\newenvironment{prf}[1][Proof]{\begin{proof}[\bf #1]}{\end{proof}}

\makeatletter
\newcommand{\hocolim@}[2]{%
  \vtop{\m@th\ialign{##\cr
    \hfil$#1\operator@font holim$\hfil\cr
    \noalign{\nointerlineskip\kern1.5\ex@}#2\cr
    \noalign{\nointerlineskip\kern-\ex@}\cr}}%
}
\newcommand{\hocolim}{%
  \mathop{\mathpalette\hocolim@{\rightarrowfill@\textstyle}}\nmlimits@
}
\makeatother

\makeatletter
\newcommand{\holim@}[2]{%
  \vtop{\m@th\ialign{##\cr
    \hfil$#1\operator@font holim$\hfil\cr
    \noalign{\nointerlineskip\kern1.5\ex@}#2\cr
    \noalign{\nointerlineskip\kern-\ex@}\cr}}%
}
\newcommand{\holim}{%
  \mathop{\mathpalette\holim@{\leftarrowfill@\textstyle}}\nmlimits@
}
\makeatother

\makeatletter
\def\@secnumfont{\bfseries}
\def\section{\@startsection{section}{1}%
  \z@{.7\linespacing\@plus\linespacing}{.5\linespacing}%
  {\normalfont\Large\bfseries\filcenter}}
\def\subsection{\@startsection{subsection}{2}%
  \z@{.5\linespacing\@plus.7\linespacing}{-.5em}%
  {\normalfont\large\bfseries}}
\makeatother

\DeclareFontFamily{OT1}{pzc}{}
\DeclareFontShape{OT1}{pzc}{m}{it}{<-> s * [1.20] pzcmi7t}{}
\DeclareMathAlphabet{\mathpzc}{OT1}{pzc}{m}{it}

\makeatletter
\def\moverlay{\mathpalette\mov@rlay}
\def\mov@rlay#1#2{\leavevmode\vtop{%
   \baselineskip\z@skip \lineskiplimit-\maxdimen
   \ialign{\hfil$\m@th#1##$\hfil\cr#2\crcr}}}
\newcommand{\charfusion}[3][\mathord]{
    #1{\ifx#1\mathop\vphantom{#2}\fi
        \mathpalette\mov@rlay{#2\cr#3}
      }
    \ifx#1\mathop\expandafter\displaylimits\fi}
\makeatother

\newcommand{\cupdot}{\charfusion[\mathbin]{\cup}{\cdot}}

\makeatletter
\providecommand{\bigsqcap}{%
  \mathop{%
    \mathpalette\@updown\bigsqcup
  }%
}
\newcommand*{\@updown}[2]{%
  \rotatebox[origin=c]{180}{$\m@th#1#2$}%
}
\makeatother

\begin{document}

\author[Hossein Faridian]{Hossein Faridian}

\title[Model Category Structure on Simplicial Algebras via Dold-Kan Correspondence]
{Model Category Structure on Simplicial Algebras \\ via Dold-Kan Correspondence}

\address{Hossein Faridian, School of Mathematical and Statistical Sciences, Clemson University, SC 29634, USA.}
\email{hfaridi@g.clemson.edu}

\subjclass[2010]{18N40; 18G31; 18G50; 18N50.}

\keywords {model category; connective chain complex; simplicial module; simplicial algebra}

\begin{abstract}
This expository article sets forth a self-contained and purely algebraic proof of a deep result of Quillen stating that the category of simplicial commutative algebras over a commutative ring is a model category. This is accomplished by starting from the model structure on the category of connective chain complexes, transferring it to the category of simplicial modules via Dold-Kan Correspondence, and further transferring it to the category of simplicial commutative algebras through Quillen-Kan Transfer Machine. The subtlety of overcoming the acyclicity condition is addressed by introducing and studying the shuffle product of connective chain complexes, establishing a variant of Eilenberg-Zilber Theorem, and carefully scrutinizing the subtle structures under study.
\end{abstract}

\maketitle

\sloppy
\raggedbottom

\tableofcontents

\section{Introduction}

Homological ideas grew out of attempts to investigate deeper properties of topological spaces by means of studying certain algebraic objects, namely homology and cohomology groups. These ideas took abstract forms later and planted the early seeds of homological algebra in the framework of category theory. In the grand scheme of things, the main goal of homological algebra is to construct derived functors of a given functor which are in some sense the best approximations of the functor at hand.

To get a flavor of derived functors in a classical setting, let $R$ and $S$ be two rings, and let $\mathcal{F}:\mathcal{M}\mathpzc{od}(R) \rightarrow \mathcal{M}\mathpzc{od}(S)$ be an additive covariant functor between module categories. For any $i\geq 0$, the $i$th left derived functor of $\mathcal{F}$ is a functor $\textrm{L}_{i}\mathcal{F}:\mathcal{M}\mathpzc{od}(R) \rightarrow \mathcal{M}\mathpzc{od}(S)$, whose effect on a left $R$-module $M$ is given by $\textrm{L}_{i}\mathcal{F}(M)=H_{i}\left(\mathcal{F}(P)\right)$ where $P$ is a chosen projective resolution of $M$, and the $i$th right derived functor of $\mathcal{F}$ is a functor $\textrm{R}^{i}\mathcal{F}:\mathcal{M}\mathpzc{od}(R) \rightarrow \mathcal{M}\mathpzc{od}(S)$, whose effect on a left $R$-module $M$ is given by $\textrm{R}^{i}\mathcal{F}(M)=H_{-i}\left(\mathcal{F}(I)\right)$ where $I$ is a chosen injective resolution of $M$. To go one step further, one can assemble all these derived functors into a single total derived functor. More specifically, let $\mathcal{F}:\mathcal{C}(R) \rightarrow \mathcal{C}(S)$ be a covariant functor between categories of chain complexes that preserves homotopy. Then the total left derived functor of $\mathcal{F}$ is a functor $\textrm{L}\mathcal{F}:\mathcal{D}(R) \rightarrow \mathcal{D}(S)$ between derived categories, whose effect on an $R$-complex $X$ is given by $\textrm{L}\mathcal{F}(X)=\mathcal{F}(P)$ where $P$ is a chosen semi-projective resolution of $X$, and the total right derived functor of $\mathcal{F}$ is a functor $\textrm{R}\mathcal{F}:\mathcal{D}(R) \rightarrow \mathcal{D}(S)$, whose effect on an $R$-complex $X$ is given by $\textrm{R}\mathcal{F}(X)=\mathcal{F}(I)$ where $I$ is a chosen semi-injective resolution of $X$. This machinery can be developed in abelian categories and their corresponding derived categories. However, there are certain contexts in which one requires to have some sort of derived functors in a non-abelian setting where the notion of a chain complex is missing.

One such situation arises in the context of commutative algebras over commutative rings. Let $R$ be a commutative ring, $A$ a commutative $R$-algebra, and $M$ an $A$-module. A derivation of $A$ over $R$ with coefficients in $M$ is an $R$-homomorphism $D:A\rightarrow M$ that satisfies the Leibniz rule, i.e. $D(ab)=aD(b)+ bD(a)$ for every $a,b \in A$. The set of all derivations of $A$ over $R$ with coefficients in $M$, denoted by $\Der_{R}(A,M)$, is an $A$-submodule of $\Hom_{R}(A,M)$. In addition, the induced functor $\Der_{R}(A,-):\mathcal{M}\mathpzc{od}(A)\rightarrow \mathcal{M}\mathpzc{od}(A)$ is representable, i.e. there exists an $A$-module $\Omega_{A|R}$, the so-called module of differentials of $A$ over $R$, and a universal derivation $d_{A|R}:A \rightarrow \Omega_{A|R}$ that induces a natural $A$-isomorphism $\Der_{R}(A,M) \cong \Hom_{A}\left(\Omega_{A|R},M \right)$. The module of differentials $\Omega_{A|R}$ can be concretely described as $\Omega_{A|R}=\mathfrak{k}/\mathfrak{k}^{2}$ where $\mathfrak{k}=\Ker(\mu)$ in which $\mu: A\otimes_{R}A \rightarrow A$ is given by $\mu(a\otimes b)=ab$ for every $a,b\in A$. In addition, the universal derivation $d_{A|R}:A \rightarrow \Omega_{A|R}$ is given by $d_{A|R}(a)=(1\otimes a - a\otimes 1) + \mathfrak{k}^{2}$ for every $a\in A$. Now assume that $R \rightarrow S \rightarrow T$ are two homomorphisms of commutative rings, and $M$ is a $T$-module. Then one has the following exact sequences:
$$\Omega_{S|R}\otimes_{S}M \rightarrow \Omega_{T|R}\otimes_{T}M \rightarrow \Omega_{T|S}\otimes_{T}M \rightarrow 0$$
and
$$0 \rightarrow \Der_{S}(T,M) \rightarrow \Der_{R}(T,M) \rightarrow \Der_{R}(S,M)$$
Historically speaking, people grew interested in extending these sequences beyond the above-displayed three terms. Viewing derivations and differentials as functors of commutative algebras, one requires to have some sort of derived functors in order to extend the sequences. However, the category of commutative algebras is not abelian; it fails to have a zero object or a zero morphism. As a consequence, the machinery of homological algebra as we know it fails in this situation.

The solution comes from Quillen's model category theory. The ideas behind the theory stem from homotopy theory in algebraic topology that are suitably adapted to the language of category theory. This has led to the emergence of non-abelian homological algebra. A model category is a locally small bicomplete category equipped with three distinguished classes of morphisms called fibrations, cofibrations, and weak equivalences that satisfy certain axioms. A model category supports a notion of replacement and a notion of homotopy that allow one to construct derived functors in a more general setting than the derived categories of abelian categories. We outline the general theory of derived functors as follows. If $\mathcal{F}:\mathcal{C}\rightarrow \mathcal{D}$ and $\mathcal{G}:\mathcal{C}\rightarrow \mathcal{E}$ are two functors between arbitrary categories, and $\mathcal{G}^{\ast}:\mathcal{D}^{\mathcal{E}}\rightarrow \mathcal{D}^{\mathcal{C}}$ denotes the right composition functor, then a left Kan extension of $\mathcal{F}$ along $\mathcal{G}$ is a functor $\LKan_{\mathcal{G}}(\mathcal{F}):\mathcal{E}\rightarrow \mathcal{D}$ together with a natural transformation $\sigma: \mathcal{F} \rightarrow \LKan_{\mathcal{G}}(\mathcal{F})\mathcal{G}$ such that the pair $\left(\LKan_{\mathcal{G}}(\mathcal{F}),\sigma\right)$ is an initial object in the comma category $\left(\mathcal{F}\downarrow \mathcal{G}^{\ast}\right)$, and a right Kan extension of $\mathcal{F}$ along $\mathcal{G}$ is a functor $\RKan_{\mathcal{G}}(\mathcal{F}):\mathcal{E}\rightarrow \mathcal{D}$ together with a natural transformation $\varsigma: \RKan_{\mathcal{G}}(\mathcal{F})\mathcal{G} \rightarrow \mathcal{F}$ such that the pair $\left(\RKan_{\mathcal{G}}(\mathcal{F}),\varsigma\right)$ is a terminal object in the comma category $\left(\mathcal{G}^{\ast}\downarrow \mathcal{F}\right)$. Now if $\mathcal{F}:\mathcal{C}\rightarrow \mathcal{D}$ is a functor between model categories, then the left derived functor of $\mathcal{F}$ is $\textrm{L}\mathcal{F}=\RKan_{\mathcal{L}_{\mathcal{C}}}(\mathcal{\mathcal{L}_{\mathcal{D}}F}):\Ho(\mathcal{C})\rightarrow \Ho(\mathcal{D})$, and the right derived functor of $\mathcal{F}$ is $\textrm{R}\mathcal{F}=\LKan_{\mathcal{L}_{\mathcal{C}}}(\mathcal{\mathcal{L}_{\mathcal{D}}F}):\Ho(\mathcal{C})\rightarrow \Ho(\mathcal{D})$, where $\mathcal{L}_{\mathcal{C}}:\mathcal{C}\rightarrow \Ho(\mathcal{C})$ and $\mathcal{L}_{\mathcal{D}}:\mathcal{D}\rightarrow \Ho(\mathcal{D})$ are localization functors. Derived functors at this level of generality can be shown to exist under mild conditions by leveraging fibrant and cofibrant replacements which in turn can be thought of as generalizations of resolutions. It is also worth noting that if $R$ is a ring, then the category $\mathcal{C}(R)$ of $R$-complexes has two model structures under any of which, one has $\Ho\left(\mathcal{C}(R)\right)\cong \mathcal{D}(R)$, and the general theory of derived functors reduces to the special abelian case described before.

In order to effectively apply Quillen's theory to the situation of interest, i.e. commutative algebras over commutative rings, one needs a non-abelian substitute for chain complexes. The Dold-Kan Correspondence is inspiring in this direction. It states that the category of connective chain complexes over a ring is equivalent to the category of simplicial modules over the ring. Although chain complexes are not available in non-abelian categories, simplicial objects are always at our disposal in any category. The punch line is that simplicial objects can serve as the non-abelian analogues of chain complexes. In particular, one should consider the category of simplicial commutative algebras in this situation and equip it with a model structure. Unfortunately, this task is by far harder than it seems. In addition, all the existing arguments in the literature rely one way or another on the model structure of simplicial sets; see for example \cite[Chapter II, Section 3]{Qu2}. Simplicial sets can be regarded as combinatorial models for topological spaces up to homotopy. Both topological spaces and simplicial sets are primordial prototypes of model categories. Let us look at the classical model structure of simplicial sets. Fibrations are Kan fibrations, cofibrations are injective morphisms, and weak equivalences are morphisms whose geometric realizations induce isomorphisms on topological homotopy groups. As it stands, the weak equivalences lack a purely algebraic description, so the proof of model structure of simplicial sets and hence simplicial commutative algebras is based partly on topological foundations. From an algebraist's point of view, it can be reasonable to look for a purely algebraic proof of the model structure of simplicial commutative algebras. This is in fact the motivation of writing this article.

Back to the commutative algebra setting, we observe that after equipping the category of simplicial commutative algebras with a model structure, one can apply the general theory of derived functors to construct the cotangent complex $\mathbb{L}^{A|R}\in \Obj\left(\mathcal{D}(A)\right)$ as the image of the identity map $1^{A}:A\rightarrow A$ under the left derived functor of the abelianization functor followed by the right derived functor of the normalization functor. The cotangent complex is in turn used to define the Andr\'{e}-Quillen homology and cohomology modules as $H_{i}^{AQ}\left(A|R;M \right)= \Tor_{i}^{A}\left(\mathbb{L}^{A|R},M\right)$ and $H_{AQ}^{i}\left(A|R;M \right)= \Ext_{A}^{i}\left(\mathbb{L}^{A|R},M\right)$ for every $i\geq 0$. These modules enable one to extend the above-mentioned sequences to the following exact sequences:
\begin{gather*}
\cdots \rightarrow H_{2}^{AQ}\left(S|R;M \right)\rightarrow H_{2}^{AQ}\left(T|R;M \right)\rightarrow H_{2}^{AQ}\left(T|S;M \right)\rightarrow H_{1}^{AQ}\left(S|R;M \right)\rightarrow \\
H_{1}^{AQ}\left(T|R;M \right) \rightarrow
H_{1}^{AQ}\left(T|S;M \right)\rightarrow \Omega_{S|R}\otimes_{S}M \rightarrow \Omega_{T|R}\otimes_{T}M \rightarrow \Omega_{T|S}\otimes_{T}M \rightarrow 0
\end{gather*}
and
\begin{gather*}
  0 \rightarrow \Der_{S}(T,M) \rightarrow \Der_{R}(T,M) \rightarrow \Der_{R}(S,M) \rightarrow H_{AQ}^{1}\left(T|S;M\right) \rightarrow H_{AQ}^{1}\left(T|R;M\right) \rightarrow \\
H_{AQ}^{1}\left(S|R;M\right) \rightarrow H_{AQ}^{2}\left(T|S;M\right) \rightarrow H_{AQ}^{2}\left(T|R;M\right) \rightarrow H_{AQ}^{2}\left(S|R;M\right) \rightarrow \cdots
\end{gather*}

\noindent
As it turns out, Andr\'{e}-Quillen homology and cohomology modules are powerful tools in commutative algebra that can characterize various classes of rings and algebras such as regular rings, complete intersection rings, smooth algebras, \'{e}tale algebras, etc. They have numerous applications in number theory and deformation theory as well; see for example \cite{An}, \cite{Av}, \cite{BI}, \cite{BIK}, \cite{DI}, \cite{Fa}, \cite{Il}, \cite{Iy}, \cite{MR}, \cite{Qu1}, \cite{Qu2}, and \cite{Qu3}. Despite being a desirable object of study for algebraists, Andr\'{e}-Quillen homology and cohomology modules have not received the kind of attention they merit from the algebra community. This could be mostly due to its foundation that has not been properly developed in an algebraic language. The surprising fact is that the foundation of the theory, i.e. the model structure of simplicial commutative algebras, is rooted in topology, while its applications, i.e. the characterization of various classes of rings and algebras, fit completely within the realm of commutative algebra. In this article, we strive to provide a purely algebraic proof of the model structure of simplicial commutative algebras that does not rely on the model structure of simplicial sets or topological spaces.

The article is designed as follows. In section two, we collect all we need from the theory of model categories with suitable references in order to set up the language and notation. We further prove a few lemmas that will be used later. In section three, we provide a proof of the model structure of connective chain complexes. The direct proof of this fact is only outlined in the literature in a disperse and incomplete manner; see for example \cite[Theorem 1.5]{GS} or \cite[Theorem 7.2]{DS}. We give a neat and complete proof for the sake of reference. This model structure is described completely algebraically and serves as our archetypal example of a model category. Furthermore, we show that this model structure is cofibrantly generated which allows us to transfer it along adjunctions. In section four, we review simplicial objects, the normalization and Dold-Kan functors, and the Dold-Kan Correspondence. In addition, we transfer the model structure of connective chain complexes to simplicial modules. In section five, we take advantage of the Dold-Kan functor to introduce and study the shuffle product of connective chain complexes. The shuffle product of elements of chain complexes has been studied in the literature; see \cite[Theorem 8.8]{Mc2} and \cite[Proposition 2.5.7.1 and Notation 2.5.7.2]{Lu}. However, the shuffle product of connective chain complexes is a new perspective. We then prove a variant of the Eilenberg-Zilber Theorem. The contents of this section are of grave significance in the sequel. In section six, we deploy the materials of section five and the Quillen-Kan Transfer Machine to transport the model structure of connective chain complexes to simplicial commutative algebras. One should note that \cite{GS} adopts the same approach, but at the most crucial step, i.e. the establishment of the acyclicity condition, it resorts to the model structure of simplicial sets; see \cite[Proposition 4.12, Corollary 4.14, and Theorem 4.17]{GS} and \cite[Proposition 11.5, Axiom 3.1, Theorem 4.1, and Example 5.2]{GJ}. We mainly follow the ideas of \cite{GS}, but leverage the shuffle product to avoid the model structure of simplicial sets altogether. As it happens, we wind up with an organized, self-contained, and completely algebraic proof with no reliance on topology.

\section{Preliminaries on Model Categories}

Quillen's theory of model categories provides a general framework to practice homological algebra far beyond the realm of abelian categories and their corresponding derived categories. In this section, we collect some basic definitions and facts from the theory of model categories and also prove a few lemmas that will be needed later. For general references on model categories, refer to \cite{Ci}, \cite{DS}, \cite{GJ}, \cite{Hi}, \cite{Ho}, \cite{Qu2}, and \cite{Ri}.

\begin{definition} \label{2.1}
A \textit{model category} is a locally small bicomplete category $\mathcal{C}$ endowed with three classes of morphisms called \textit{fibrations}, \textit{cofibrations}, and \textit{weak equivalences}. A morphism that is a fibration as well as a weak equivalence is called a \textit{trivial fibration}, and a morphism that is a cofibration as well as a weak equivalence is called a \textit{trivial cofibration}. Moreover, these morphisms should satisfy the following conditions:

\begin{enumerate}
\item[(i)] \textit{Lifting}: Any commutative diagram on the left in which $f$ is a trivial cofibration and $g$ is a fibration, or $f$ is a cofibration and $g$ is a trivial fibration, can be completed to the commutative diagram on the right:
\[
 \begin{tikzcd}
  A \arrow{r} \arrow{d}[swap]{f}
  & C \arrow{d}{g}
  \\
  B \arrow{r}
  & D
\end{tikzcd}
\quad \Longrightarrow \quad
 \begin{tikzcd}
  A \arrow{r} \arrow{d}[swap]{f}
  & C \arrow{d}{g}
  \\
 B \arrow{ru} \arrow{r} & D
\end{tikzcd}
\]

\item[(ii)] \textit{Retract}: If $f$ is a retract of $g$ in the arrow category $\Arr(\mathcal{C})$, i.e. there is a commutative diagram
\begin{equation*}
  \begin{tikzcd}
  A \arrow{r}{h} \arrow{d}{f}
  & C \arrow{d}{g} \arrow{r}{l}
  & A \arrow{d}{f}
  \\
  B \arrow{r}{p}
  & D \arrow{r}{q}
  & B
\end{tikzcd}
\end{equation*}

\noindent
in which $lh=1^{A}$ and $qp=1^{B}$, and if $g$ is a fibration (cofibration, or weak equivalence), then $f$ is a fibration (cofibration, or weak equivalence).

\item[(iii)] \textit{Two-Out-of-Three}: If $f=gh$, and any two of $f$, $g$, or $h$ are weak equivalences, then so is the third.

\item[(iv)] \textit{Factorization}: For any morphism $f$, there is a factorization $f=gh$ for a fibration $g$ and a trivial cofibration $h$, and also a factorization $f=g'h'$ for a trivial fibration $g'$ and a cofibration $h'$.
\end{enumerate}
\end{definition}

A model category supports a notion of homotopy which is the bedrock of constructing derived functors. In order to define this notion, we need the definitions of cylinder and path objects which could be shown to exist in every model category.

\begin{definition} \label{2.2}
Let $\mathcal{C}$ be a model category, and $A\in \Obj(\mathcal{C})$. Then:
\begin{enumerate}
\item[(i)] A \textit{cylinder object} for $A$ is an object $\Cyl(A)$ through which the codiagonal morphism $\nabla_{A}:A\sqcup A \rightarrow A$ factors as
\begin{equation*}
  \begin{tikzcd}
  A\sqcup A \arrow{r}{\nabla_{A}}\arrow{d}[swap]{\kappa_{A}}
  & A
  \\
  \Cyl(A) \arrow{ru}[swap]{\xi_{A}}
\end{tikzcd}
\end{equation*}

\noindent
in which $\kappa_{A}$ is a cofibration and $\xi_{A}$ is a weak equivalence.

\item[(ii)] A \textit{path object} for $A$ is an object $\Path(A)$ through which the diagonal morphism $\Delta_{A}:A\rightarrow A\sqcap A$ factors as
\begin{equation*}
  \begin{tikzcd}
  A \arrow{r}{\Delta_{A}}\arrow{d}[swap]{\zeta_{A}}
  & A\sqcap A
  \\
  \Path(A) \arrow{ru}[swap]{\nu_{A}}
\end{tikzcd}
\end{equation*}

\noindent
in which $\nu_{A}$ is a fibration and $\zeta_{A}$ is a weak equivalence.
\end{enumerate}
\end{definition}

\begin{definition} \label{2.3}
Let $\mathcal{C}$ be a model category, and $f,g:A \rightarrow B$ two morphisms in $\mathcal{C}$. Then:
\begin{enumerate}
\item[(i)] $f$ is said to be \textit{left homotopic} to $g$, denoted by $f\sim_{l}g$, if there is a cylinder object $\Cyl(A)$ for $A$, as in Definition \ref{2.2}, such that $f$ and $g$ factor through $\Cyl(A)$ as
\begin{equation*}
  \begin{tikzcd}
  A \arrow{r}{f} \arrow{rd}[swap]{\kappa_{A}\iota_{1}^{A}}
  & B
  & A \arrow{l}[swap]{g} \arrow{ld}{\kappa_{A}\iota_{2}^{A}}
  \\
  & \Cyl(A) \arrow{u}{\phi}
\end{tikzcd}
\end{equation*}

\noindent
in which $\iota_{1}^{A}:A\rightarrow A\sqcup A$ and $\iota_{2}^{A}:A\rightarrow A\sqcup A$ are canonical injections. Moreover in this case, $\phi$ is said to be a \textit{left homotopy} from $f$ to $g$.

\item[(ii)] $f$ is said to be \textit{right homotopic} to $g$, denoted by $f\sim_{r}g$, if there is a path object $\Path(B)$ for $B$, as in Definition \ref{2.2}, such that $f$ and $g$ factor through $\Path(B)$ as
\begin{equation*}
  \begin{tikzcd}
  B
  & A \arrow{l}[swap]{f} \arrow{r}{g} \arrow{d}{\psi}
  & B
  \\
  & \Path(B) \arrow{lu}{\pi_{1}^{B}\nu_{B}} \arrow{ru}[swap]{\pi_{2}^{B}\nu_{B}}
\end{tikzcd}
\end{equation*}

\noindent
in which $\pi_{1}^{B}:B\sqcap B \rightarrow B$ and $\pi_{2}^{B}:B\sqcap B \rightarrow B$ are canonical projections. Moreover in this case, $\psi$ is said to be a \textit{right homotopy} from $f$ to $g$.

\item[(iii)] $f$ is said to be \textit{homotopic} to $g$, denoted by $f\sim g$, if $f\sim_{l}g$ and $f\sim_{r}g$.
\end{enumerate}
\end{definition}

Recall that a model category is bicomplete, i.e. it has all small direct and inverse limits. In particular, a model category has initial and terminal objects. If $\mathcal{C}$ is a model category, $\circ$ is its initial object, and $\ast$ is its terminal object, then an object $A$ in $\mathcal{C}$ is said to be \textit{fibrant} if the unique morphism $A \rightarrow \ast$ is a fibration, and is said to be \textit{cofibrant} if the unique morphism $\circ \rightarrow A$ is a cofibration. Now we can see that homotopy is an equivalence relation on morphisms.

\begin{proposition} \label{2.4}
Let $\mathcal{C}$ be a model category, and $f,g:A \rightarrow B$ two morphisms in $\mathcal{C}$. Then the following assertions hold:
\begin{enumerate}
\item[(i)] If $f\sim_{l}g$ and $A$ is cofibrant, then $f\sim_{r}g$ and $\sim_{l}$ is an equivalence relation on $\Mor_{\mathcal{C}}(A,B)$.
\item[(ii)] If $f\sim_{r}g$ and $B$ is fibrant, then $f\sim_{l}g$ and $\sim_{r}$ is an equivalence relation on $\Mor_{\mathcal{C}}(A,B)$.
\item[(iii)] If $A$ is cofibrant and $B$ is fibrant, then $f\sim_{l}g$ if and only if $f\sim_{r}g$, so $\sim$ is an equivalence relation on $\Mor_{\mathcal{C}}(A,B)$.
\end{enumerate}
\end{proposition}

\begin{proof}
See \cite[Proposition 7.4.5]{Hi}.
\end{proof}

A model category enjoys a concept of replacement which can be thought of as a generalization as well as a non-abelian analogue of a resolution in abelian homological algebra. If $\mathcal{C}$ is a model category, and $A$ is an object in $\mathcal{C}$, then a \textit{fibrant replacement} of $A$ is a trivial cofibration $\varrho_{A}:A\rightarrow A_{{\textrm{f}}}$ in which $A_{{\textrm{f}}}$ is fibrant, and a \textit{cofibrant replacement} of $A$ is a trivial fibration $\rho_{A}:A_{{\textrm{c}}}\rightarrow A$ in which $A_{{\textrm{c}}}$ is cofibrant. Now we have a comparison lemma for replacements as expected.

\begin{proposition} \label{2.5}
Let $\mathcal{C}$ be a model category. Then the following assertions hold:
\begin{enumerate}
\item[(i)] Every object $A$ of $\mathcal{C}$ has a fibrant replacement $\varrho_{A}:A\rightarrow A_{\emph{f}}$. Choosing a fibrant replacement for each object of $\mathcal{C}$, if $f:A \rightarrow B$ is a morphism in $\mathcal{C}$, then there is a morphism $f_{\emph{f}}:A_{\emph{f}} \rightarrow B_{\emph{f}}$ which is unique up to homotopy and makes the following diagram commutative:
\begin{equation*}
  \begin{tikzcd}
  A \arrow{r}{\varrho_{A}} \arrow{d}[swap]{f}
  & A_{\emph{f}} \arrow{d}{f_{\emph{f}}}
  \\
  B \arrow{r}{\varrho_{B}}
  & B_{\emph{f}}
\end{tikzcd}
\end{equation*}

\noindent
Moreover, if $g:B \rightarrow C$ is another morphism in $\mathcal{C}$, then $(gf)_{\emph{f}}\sim g_{\emph{f}}f_{\emph{f}}$, and $(1^{A})_{\emph{f}}\sim 1^{A_{\emph{f}}}$.

\item[(ii)] Every object $A$ of $\mathcal{C}$ has a cofibrant replacement $\rho_{A}:A_{\emph{c}}\rightarrow A$. Choosing a cofibrant replacement for each object of $\mathcal{C}$, if $f:A \rightarrow B$ is a morphism in $\mathcal{C}$, then there is a morphism $f_{\emph{c}}:A_{\emph{c}} \rightarrow B_{\emph{c}}$ which is unique up to homotopy and makes the following diagram commutative:
\begin{equation*}
  \begin{tikzcd}
  A_{\emph{c}} \arrow{r}{\rho_{A}} \arrow{d}[swap]{f_{\emph{c}}}
  & A \arrow{d}{f}
  \\
  B_{\emph{c}} \arrow{r}{\rho_{B}}
  & B
\end{tikzcd}
\end{equation*}

\noindent
Moreover, if $g:B \rightarrow C$ is another morphism in $\mathcal{C}$, then $(gf)_{\emph{c}}\sim g_{\emph{c}}f_{\emph{c}}$, and $(1^{A})_{\emph{c}}\sim 1^{A_{\emph{c}}}$.
\end{enumerate}
\end{proposition}

\begin{proof}
See \cite[Corollary 8.1.8]{Hi}.
\end{proof}

The notions of fibrant and cofibrant replacements also allow us to define the homotopy category of a model category which can be thought of as a generalization as well as a non-abelian analogue of the derived category of an abelian category.

\begin{definition} \label{2.6}
Let $\mathcal{C}$ be a model category. Choose a fibrant replacement $\varrho_{A}:A\rightarrow A_{{\textrm{f}}}$ as well as a cofibrant replacement $\rho_{A}:A_{{\textrm{c}}}\rightarrow A$ for every $A\in \Obj(\mathcal{C})$. Then the \textit{homotopy category} $\Ho(\mathcal{C})$ of $\mathcal{C}$ is defined as follows:
\begin{enumerate}
\item[(i)] $\Obj\left({\textrm{Ho}}(\mathcal{C})\right):=\Obj(\mathcal{C})$.
\item[(ii)] $\Mor_{{\textrm{Ho}}(\mathcal{C})}(A,B):=\Mor_{\mathcal{C}}(A_{{\textrm{fc}}},B_{{\textrm{fc}}})/ \sim$ for every $A,B\in \Obj\left({\textrm{Ho}}(\mathcal{C})\right)$.
\end{enumerate}
\end{definition}

We next have the universal property of the homotopy category which in turn shows that the homotopy category is the localization of the model category with respect to the class of weak equivalences.

\begin{proposition} \label{2.7}
Let $\mathcal{C}$ be a model category. Choose a fibrant replacement $\varrho_{A}:A\rightarrow A_{\emph{f}}$ as well as a cofibrant replacement $\rho_{A}:A_{\emph{c}}\rightarrow A$ for every $A\in \Obj(\mathcal{C})$. Then there is a covariant functor $\mathcal{L}_{\mathcal{C}}:\mathcal{C}\rightarrow \Ho(\mathcal{C})$ defined as $\mathcal{L}_{\mathcal{C}}(A):=A$ for every $A\in \Obj(\mathcal{C})$, and $\mathcal{L}_{\mathcal{C}}(f):=[f_{\emph{fc}}]$ for every $f\in \Mor_{\mathcal{C}}(A,B)$. Moreover, $\mathcal{L}_{\mathcal{C}}$ has the following properties:
\begin{enumerate}
\item[(i)] A morphism $f$ in $\mathcal{C}$ is a weak equivalence if and only if $\mathcal{L}_{\mathcal{C}}(f)$ is an isomorphism in $\Ho(\mathcal{C})$.
\item[(ii)] If $\mathcal{F}:\mathcal{C}\rightarrow \mathcal{D}$ is a functor that maps weak equivalences in $\mathcal{C}$ to isomorphisms in $\mathcal{D}$, then $\mathcal{F}$ factors uniquely through $\mathcal{L}_{\mathcal{C}}$, i.e. there is a unique functor $\mathcal{\check{F}}:\Ho(\mathcal{C})\rightarrow \mathcal{D}$ that makes the following diagram commutative:
\begin{equation*}
  \begin{tikzcd}
 \mathcal{C} \arrow{r}{\mathcal{F}}\arrow{d}[swap]{\mathcal{L}_{\mathcal{C}}}
  & \mathcal{D}
  \\
  \Ho(\mathcal{C}) \arrow{ru}[swap]{\mathcal{\check{F}}}
\end{tikzcd}
\end{equation*}
\end{enumerate}
\end{proposition}

\begin{proof}
See \cite[Proposition 8.3.7 and Theorem 8.3.10]{Hi}.
\end{proof}

We next focus on cofibrantly generated model categories whose importance stems from the transferability of their model structure along adjunctions.

Let $\mathcal{C}$ be a category, and $f:A\rightarrow B$ and $g:C\rightarrow D$ two morphisms in $\mathcal{C}$. If any commutative diagram as the one on the left can be completed to the commutative diagram on the right
\[
 \begin{tikzcd}
  A \arrow{r} \arrow{d}[swap]{f}
  & C \arrow{d}{g}
  \\
  B \arrow{r}
  & D
\end{tikzcd}
\quad \Longrightarrow \quad
 \begin{tikzcd}
  A \arrow{r} \arrow{d}[swap]{f}
  & C \arrow{d}{g}
  \\
 B \arrow{ru} \arrow{r} & D
\end{tikzcd}
\]

\noindent
then we say that $f$ has the \textit{left lifting property against} $g$ and $g$ has the \textit{right lifting property against} $f$. Given a class $\mathcal{X}$ of morphisms in $\mathcal{C}$, let $\LLP(\mathcal{X})$ denote the class of morphisms in $\mathcal{C}$ that have the left lifting property against all morphisms in $\mathcal{X}$, and $\RLP(\mathcal{X})$ denote the class of morphisms in $\mathcal{C}$ that have the right lifting property against all morphisms in $\mathcal{X}$. If $\mathcal{C}$ is a model category, then by \cite[Proposition 7.2.3]{Hi}, $g$ is a fibration if and only if it has the right lifting property against all trivial cofibrations; $g$ is a trivial fibration if and only if it has the right lifting property against all cofibrations; $f$ is a cofibration if and only if it has the left lifting property against all trivial fibrations; and $f$ is a trivial cofibration if and only if it has the left lifting property against all fibrations.

Recall that if $\mathcal{C}$ is a locally small cocomplete category and $\mathcal{X}$ is a class of morphisms in $\mathcal{C}$, then an object $A$ of $\mathcal{C}$ is said to be \textit{small} with respect to $\mathcal{X}$ if the functor $\Mor_{\mathcal{C}}(A,-):\mathcal{C}\rightarrow \mathcal{S}\mathpzc{et}$ preserves sequential direct limits of direct systems in $\mathcal{X}$, i.e. for any sequence $\left\{g_{n}:B_{n}\rightarrow B_{n+1}\right\}_{n\geq 1}$ of morphisms in $\mathcal{X}$, the natural morphism
$$\underset{n\geq 1}{\varinjlim}\Mor_{\mathcal{C}}(A,B_{n})\rightarrow \Mor_{\mathcal{C}}\left(A,\underset{n\geq 1}{\varinjlim}B_{n}\right)$$
is an isomorphism.

Quillen devised a method to transfer model structures along adjunctions which is primarily based on his well-known "small object argument". A set $\mathcal{X}$ of morphisms in a locally small cocomplete category $\mathcal{C}$ is said to \textit{permit the small object argument} if the source of any morphism in $\mathcal{X}$ is small with respect to $\LLP\left(\RLP(\mathcal{X})\right)$. This condition is necessary in order to be able to run the small object argument. We now recall the notion of a cofibrantly generated model category.

\begin{definition} \label{2.8}
A model category $\mathcal{C}$ is said to be \textit{cofibrantly generated} if there exist two sets $\mathcal{X}$ and $\mathcal{Y}$ of morphisms in $C$ that satisfy the following conditions:
\begin{enumerate}
\item[(i)] $\mathcal{X}$ and $\mathcal{Y}$ permit the small object argument.
\item[(ii)] $\RLP(\mathcal{X})$ is the class of trivial fibrations.
\item[(iii)] $\RLP(\mathcal{Y})$ is the class of fibrations.
\end{enumerate}
\end{definition}

The next theorem is known as Quillen-Kan Transfer Machine.

\begin{theorem} \label{2.9}
Let $\mathcal{C}$ be a cofibrantly generated model category with the corresponding sets of morphisms $\mathcal{X}$ and $\mathcal{Y}$, $\mathcal{D}$ a locally small bicomplete category, and $(\mathcal{F},\mathcal{G}):\mathcal{C} \leftrightarrows \mathcal{D}$ an adjoint pair of functors. Then:
\begin{enumerate}
\item[(i)] Let a morphism $f:A\rightarrow B$ in $\mathcal{D}$ be a fibration if $\mathcal{G}(f):\mathcal{G}(A)\rightarrow \mathcal{G}(B)$ is a fibration in $\mathcal{C}$.
\item[(ii)] Let a morphism $f:A\rightarrow B$ in $\mathcal{D}$ be a weak equivalence if $\mathcal{G}(f):\mathcal{G}(A)\rightarrow \mathcal{G}(B)$ is a weak equivalence in $\mathcal{C}$.
\item[(iii)] Let a morphism $f:A\rightarrow B$ in $\mathcal{D}$ be a cofibration if it has the left lifting property against trivial fibrations defined in (i) and (ii).
\end{enumerate}
With the above definitions of fibrations, cofibrations, and weak equivalences in place, if the sets of morphisms $\mathcal{F}(\mathcal{X})$ and $\mathcal{F}(\mathcal{Y})$ permit the small object argument, and any cofibration in $\mathcal{D}$ that has the left lifting property against fibrations is also a weak equivalence, then $\mathcal{D}$ becomes a cofibrantly generated model category.
\end{theorem}

\begin{proof}
See \cite[Theorem 11.3.2]{Hi}.
\end{proof}

The last condition in Theorem \ref{2.9} is commonly known as the \enquote{acyclicity condition} which is in practice notoriously hard to establish. We next prove a lemma that facilitates the permission of small object argument in our cases of interest.

\begin{lemma} \label{2.10}
Let $\mathcal{C}$ and $\mathcal{D}$ be two locally small cocomplete categories, and $(\mathcal{F},\mathcal{G}):\mathcal{C} \leftrightarrows \mathcal{D}$ an adjoint pair of functors in which $\mathcal{G}$ preserves sequential direct limits. Suppose that $\mathcal{X}$ is a set of morphisms in $\mathcal{C}$ such that the source of any morphism in $\mathcal{X}$ is small with respect to $\mathcal{G}\left(\LLP\left(\RLP\left(\mathcal{F}(\mathcal{X})\right)\right)\right)$. Then $\mathcal{F}(\mathcal{X})=\left\{F(f) \suchthat f\in \mathcal{X}\right\}$ permits the small object argument.
\end{lemma}

\begin{proof}
Let $f:A\rightarrow B$ be a morphism in $\mathcal{X}$, and consider the morphism $\mathcal{F}(f):\mathcal{F}(A)\rightarrow \mathcal{F}(B)$ in $\mathcal{F}(\mathcal{X})$. We need to show that $\mathcal{F}(A)$ is small with respect to $\LLP\left(\RLP\left(\mathcal{F}(\mathcal{X})\right)\right)$. To this end, let $\left\{h_{n}:C_{n}\rightarrow C_{n+1}\right\}_{n\geq 1}$ be a sequence of morphisms in $\LLP\left(\RLP\left(\mathcal{F}(\mathcal{X})\right)\right)$. Let $\phi_{m}:C_{m}\rightarrow \underset{n\geq 1}{\varinjlim}C_{n}$ be the canonical injection for every $m\geq 1$. Then there is a unique morphism $\eta:\underset{n\geq 1}{\varinjlim}\mathcal{G}(C_{n}) \rightarrow \mathcal{G}\left(\underset{n\geq 1}{\varinjlim}C_{n}\right)$ that makes the following diagram commutative for every $m\geq 1$:
\begin{equation*}
\begin{tikzcd}
  \mathcal{G}(C_{m}) \arrow{rr}{\mathcal{G}(h_{m})} \arrow{dr}[swap]{\psi_{m}} \arrow[bend right, swap]{ddr}{\mathcal{G}(\phi_{m})} &  & \mathcal{G}(C_{m+1}) \arrow{dl}{\psi_{m+1}} \arrow[bend left]{ddl}{\mathcal{G}(\phi_{m+1})} \\
   & \underset{n\geq 1}{\varinjlim}\mathcal{G}(C_{n}) \arrow{d}{\eta} & \\
   & \mathcal{G}\left(\underset{n\geq 1}{\varinjlim}C_{n}\right) &
\end{tikzcd}
\end{equation*}

\noindent
Since $\mathcal{G}$ preserves sequential direct limits, $\eta$ is an isomorphism. Moreover, there is a unique morphism $\theta:\underset{n\geq 1}{\varinjlim}\Mor_{\mathcal{C}}\left(A,\mathcal{G}(C_{n})\right) \rightarrow \Mor_{\mathcal{C}}\left(A,\underset{n\geq 1}{\varinjlim}\mathcal{G}(C_{n})\right)$ that makes the following diagram commutative for every $m\geq 1$:
\begin{equation*}
\begin{tikzcd}
  \Mor_{\mathcal{C}}\left(A,\mathcal{G}(C_{m})\right) \arrow{rr}{\Mor_{\mathcal{C}}\left(A,\mathcal{G}(h_{m})\right)} \arrow{dr}[swap]{\chi_{m}} \arrow[bend right, swap]{ddr}{\Mor_{\mathcal{C}}(A,\psi_{m})} &  & \Mor_{\mathcal{C}}\left(A,\mathcal{G}(C_{m+1})\right) \arrow{dl}{\chi_{m+1}} \arrow[bend left]{ddl}{\Mor_{\mathcal{C}}(A,\psi_{m+1})} \\
   & \underset{n\geq 1}{\varinjlim}\Mor_{\mathcal{C}}\left(A,\mathcal{G}(C_{n})\right) \arrow{d}{\theta} & \\
   & \Mor_{\mathcal{C}}\left(A,\underset{n\geq 1}{\varinjlim}\mathcal{G}(C_{n})\right) &
\end{tikzcd}
\end{equation*}

\noindent
Also, there is a unique morphism $\vartheta:\underset{n\geq 1}{\varinjlim}\Mor_{\mathcal{C}}\left(A,\mathcal{G}(C_{n})\right) \rightarrow \Mor_{\mathcal{C}}\left(A,\mathcal{G}\left(\underset{n\geq 1}{\varinjlim}C_{n}\right)\right)$ that makes the following diagram commutative for every $m\geq 1$:
\begin{equation*}
\begin{tikzcd}
  \Mor_{\mathcal{C}}\left(A,\mathcal{G}(C_{m})\right) \arrow{rr}{\Mor_{\mathcal{C}}\left(A,\mathcal{G}(h_{m})\right)} \arrow{dr}[swap]{\chi_{m}} \arrow[bend right, swap]{ddr}{\Mor_{\mathcal{C}}\left(A,\mathcal{G}(\phi_{m})\right)} &  & \Mor_{\mathcal{C}}\left(A,\mathcal{G}(C_{m+1})\right) \arrow{dl}{\chi_{m+1}} \arrow[bend left]{ddl}{\Mor_{\mathcal{C}}\left(A,\mathcal{G}(\phi_{m+1})\right)} \\
   & \underset{n\geq 1}{\varinjlim}\Mor_{\mathcal{C}}\left(A,\mathcal{G}(C_{n})\right) \arrow{d}{\vartheta} & \\
   & \Mor_{\mathcal{C}}\left(A,\mathcal{G}\left(\underset{n\geq 1}{\varinjlim}C_{n}\right)\right) &
\end{tikzcd}
\end{equation*}

\noindent
Now we see that the following diagram is commutative:
\begin{equation*}
\begin{tikzcd}
  \underset{n\geq 1}{\varinjlim}\Mor_{\mathcal{C}}\left(A,\mathcal{G}(C_{n})\right) \arrow{r}{\vartheta} \arrow{d}[swap]{\theta}
  & \Mor_{\mathcal{C}}\left(A,\mathcal{G}\left(\underset{n\geq 1}{\varinjlim}C_{n}\right)\right)
  \\
  \Mor_{\mathcal{C}}\left(A,\underset{n\geq 1}{\varinjlim}\mathcal{G}(C_{n})\right) \arrow{ru}[swap]{\Mor_{\mathcal{C}}(A,\eta)} &
\end{tikzcd}
\end{equation*}

\noindent
Indeed, we have for every $m\geq 1$:
$$\Mor_{\mathcal{C}}(A,\eta)\theta\chi_{m}= \Mor_{\mathcal{C}}(A,\eta)\Mor_{\mathcal{C}}(A,\psi_{m})= \Mor_{\mathcal{C}}(A,\eta\psi_{m})= \Mor_{\mathcal{C}}\left(A,\mathcal{G}(\phi_{m})\right)= \vartheta\chi_{m}$$
It follows that $\Mor_{\mathcal{C}}(A,\eta)\theta=\vartheta$ as desired. We next note that for any $n\geq 1$, we have $h_{n}\in \LLP\left(\RLP\left(\mathcal{F}(\mathcal{X})\right)\right)$, so $\mathcal{G}(h_{n})\in \mathcal{G}\left(\LLP\left(\RLP\left(\mathcal{F}(\mathcal{X})\right)\right)\right)$. By the hypothesis, $A$ is small with respect to $\mathcal{G}\left(\LLP\left(\RLP\left(\mathcal{F}(\mathcal{X})\right)\right)\right)$, so we conclude that $\theta$ is an isomorphism. We had also observed above that $\eta$ was an isomorphism, so $\Mor_{\mathcal{C}}(A,\eta)$ is an isomorphism. Therefore, $\vartheta=\Mor_{\mathcal{C}}(A,\eta)\theta$ is an isomorphism. On the other hand, there is a unique morphism $\lambda:\underset{n\geq 1}{\varinjlim}\Mor_{\mathcal{D}}\left(\mathcal{F}(A),C_{n}\right) \rightarrow \Mor_{\mathcal{D}}\left(\mathcal{F}(A),\underset{n\geq 1}{\varinjlim}C_{n}\right)$ that makes the following diagram commutative for every $m\geq 1$:
\begin{equation*}
\begin{tikzcd}
  \Mor_{\mathcal{D}}\left(\mathcal{F}(A),C_{m}\right) \arrow{rr}{\Mor_{\mathcal{D}}\left(\mathcal{F}(A),h_{m}\right)} \arrow{dr}[swap]{\rho_{m}} \arrow[bend right, swap]{ddr}{\Mor_{\mathcal{D}}\left(\mathcal{F}(A),\phi_{m}\right)} &  & \Mor_{\mathcal{D}}\left(\mathcal{F}(A),C_{m+1}\right) \arrow{dl}{\rho_{m+1}} \arrow[bend left]{ddl}{\Mor_{\mathcal{D}}\left(\mathcal{F}(A),\phi_{m+1}\right)} \\
   & \underset{n\geq 1}{\varinjlim}\Mor_{\mathcal{D}}\left(\mathcal{F}(A),C_{n}\right) \arrow{d}{\lambda} & \\
   & \Mor_{\mathcal{D}}\left(\mathcal{F}(A),\underset{n\geq 1}{\varinjlim}C_{n}\right) &
\end{tikzcd}
\end{equation*}

\noindent
Let $\sigma_{U,V}:\Mor_{\mathcal{D}}\left(\mathcal{F}(U),V\right)\rightarrow \Mor_{\mathcal{C}}\left(U,\mathcal{G}(V)\right)$ be the natural bijection of the adjunction for every $U\in \Obj(\mathcal{C})$ and $V\in \Obj(\mathcal{D})$. Then the following diagram is commutative:
\begin{equation*}
\begin{tikzcd}
  \underset{n\geq 1}{\varinjlim}\Mor_{\mathcal{D}}\left(\mathcal{F}(A),C_{n}\right) \arrow{r}{\underset{n\geq 1}{\varinjlim}\sigma_{A,C_{n}}} \arrow{d}[swap]{\lambda}
  & [2em] \underset{n\geq 1}{\varinjlim}\Mor_{\mathcal{C}}\left(A,\mathcal{G}(C_{n})\right) \arrow{d}{\vartheta}
  \\
  \Mor_{\mathcal{D}}\left(\mathcal{F}(A),\underset{n\geq 1}{\varinjlim}C_{n}\right) \arrow{r}{\sigma_{A,\underset{n\geq 1}{\varinjlim}C_{n}}}
  & \Mor_{\mathcal{C}}\left(A,\mathcal{G}\left(\underset{n\geq 1}{\varinjlim}C_{n}\right)\right)
\end{tikzcd}
\end{equation*}

\noindent
Indeed, the commutativity of the above diagram follows from the naturality of the vertical morphisms applied to the natural transformation $\sigma_{A,-}:\Mor_{\mathcal{D}}\left(\mathcal{F}(A),-\right)\rightarrow \Mor_{\mathcal{C}}\left(A,\mathcal{G}(-)\right)$. Now since all the morphisms in the above diagram are isomorphisms except possibly for $\lambda$, we infer that $\lambda$ must also be an isomorphism. This means that $\mathcal{F}(A)$ is small with respect to $\LLP\left(\RLP\left(\mathcal{F}(\mathcal{X})\right)\right)$. Therefore, $\mathcal{F}(\mathcal{X})$ permits the small object argument.
\end{proof}

\begin{corollary} \label{2.11}
Let $\mathcal{C}$ and $\mathcal{D}$ be two locally small cocomplete categories, and $(\mathcal{F},\mathcal{G}):\mathcal{C} \leftrightarrows \mathcal{D}$ an adjoint pair of functors in which $\mathcal{G}$ preserves sequential direct limits. Suppose that $\mathcal{X}$ is a set of morphisms in $\mathcal{C}$ such that the source of any morphism in $\mathcal{X}$ is small with respect to $\Mor(\mathcal{C})$. Then the source of any morphism in $\mathcal{F}(\mathcal{X})$ is small with respect to $\Mor(\mathcal{D})$. In particular, $\mathcal{F}(\mathcal{X})$ permits the small object argument.
\end{corollary}

\begin{proof}
Follows from the proof of Lemma \ref{2.10}.
\end{proof}

The following lemma might be of independent interest.

\begin{lemma} \label{2.12}
Let $\mathcal{C}$ and $\mathcal{D}$ be two locally small categories, and $(\mathcal{F},\mathcal{G}):\mathcal{C} \leftrightarrows \mathcal{D}$ an adjoint pair of functors. Let $\sigma_{A,B}:\Mor_{\mathcal{D}}\left(\mathcal{F}(A),B\right)\rightarrow \Mor_{\mathcal{C}}\left(A,\mathcal{G}(B)\right)$ be the natural bijection of the adjunction for every $A\in \Obj(\mathcal{C})$ and $B\in \Obj(\mathcal{D})$. Then the following assertions hold:
\begin{enumerate}
\item[(i)] For any $f\in \Mor_{\mathcal{D}}\left(\mathcal{F}(A),B\right)$ and $g\in \Mor_{\mathcal{D}}(B,B')$, we have $\sigma_{A,B'}(gf)=\mathcal{G}(g)\sigma_{A,B}(f)$.
\item[(ii)] For any $f\in \Mor_{\mathcal{C}}\left(A,\mathcal{G}(B)\right)$ and $g\in \Mor_{\mathcal{D}}(B,B')$, we have $\sigma_{A,B'}^{-1}\left(\mathcal{G}(g)f\right)=g\sigma_{A,B}^{-1}(f)$.
\item[(iii)] For any $g\in \Mor_{\mathcal{C}}\left(A,\mathcal{G}(B)\right)$ and $f\in \Mor_{\mathcal{C}}(A',A)$, we have $\sigma_{A',B}^{-1}(gf)=\sigma_{A,B}^{-1}(g)\mathcal{F}(f)$.
\item[(iv)] For any $g\in \Mor_{\mathcal{D}}\left(\mathcal{F}(A),B\right)$ and $f\in \Mor_{\mathcal{C}}(A',A)$, we have $\sigma_{A',B}\left(g\mathcal{F}(f)\right)=\sigma_{A,B}(g)f$.
\end{enumerate}
\end{lemma}

\begin{proof}
(i): Let $f\in \Mor_{\mathcal{D}}\left(\mathcal{F}(A),B\right)$ and $g\in \Mor_{\mathcal{D}}(B,B')$. Let $u_{A}:A\rightarrow \mathcal{G}\left(\mathcal{F}(A)\right)$ be the unit morphism of $A$. Then using \cite[Page 82, Eq. (5)]{Mc1}, we have $\sigma_{A,B'}(gf)=\mathcal{G}(gf)u_{A}=\mathcal{G}(g)\mathcal{G}(f)u_{A}=\mathcal{G}(g)\sigma_{A,B}(f)$.

(ii): Let $f\in \Mor_{\mathcal{C}}\left(A,\mathcal{G}(B)\right)$ and $g\in \Mor_{\mathcal{D}}(B,B')$. Set $h=\sigma_{A,B'}^{-1}\left(\mathcal{G}(g)f\right)$ and $l=\sigma_{A,B}^{-1}(f)$. Using (i), we get $\sigma_{A,B'}(gl)=\mathcal{G}(g)\sigma_{A,B}(l)=\mathcal{G}(g)f=\sigma_{A,B'}(h)$. Since $\sigma_{A,B'}$ is bijective, we conclude that $ g\sigma_{A,B}^{-1}(f)=gl=h=\sigma_{A,B'}^{-1}\left(\mathcal{G}(g)f\right)$.

(iii): Similar to (i) using the counit morphism.

(iv): Similar to (ii).
\end{proof}

We can now compare the lifting properties along an adjunction. Note that the next lemma is in \cite[Theorem 7.2.17]{Hi}, but the proof there skips some crucial steps which we include here for clarity.

\begin{lemma} \label{2.13}
Let $\mathcal{C}$ and $\mathcal{D}$ be two locally small categories, and $(\mathcal{F},\mathcal{G}):\mathcal{C} \leftrightarrows \mathcal{D}$ an adjoint pair of functors. Let $f:A\rightarrow B$ be a morphism in $\mathcal{C}$, and $g:C\rightarrow D$ a morphism in $\mathcal{D}$. Then $\mathcal{F}(f)$ has the left lifting property against $g$ if and only if $f$ has the left lifting property against $\mathcal{G}(g)$.
\end{lemma}

\begin{proof}
Let $\sigma_{U,V}:\Mor_{\mathcal{D}}\left(\mathcal{F}(U),V\right)\rightarrow \Mor_{\mathcal{C}}\left(U,\mathcal{G}(V)\right)$ be the natural bijection of the adjunction for every $U\in \Obj(\mathcal{C})$ and $V\in \Obj(\mathcal{D})$. Suppose that $\mathcal{F}(f)$ has the left lifting property against $g$. Consider a commutative diagram as follows:
\begin{equation*}
\begin{tikzcd}
  A \arrow{r}{h} \arrow{d}[swap]{f}
  & \mathcal{G}(C) \arrow{d}{\mathcal{G}(g)}
  \\
  B \arrow{r}{l}
  & \mathcal{G}(D)
\end{tikzcd}
\end{equation*}

\noindent
Then the following diagram is commutative:
\begin{equation*}
\begin{tikzcd}
  \mathcal{F}(A) \arrow{r}{\sigma_{A,C}^{-1}(h)} \arrow{d}[swap]{\mathcal{F}(f)}
  & [1.5em] C \arrow{d}{g}
  \\
  \mathcal{F}(B) \arrow{r}{\sigma_{B,D}^{-1}(l)}
  & D
\end{tikzcd}
\end{equation*}

\noindent
Indeed, using Lemma \ref{2.12} (ii) and (iii), we have: $$\sigma_{B,D}^{-1}(l)\mathcal{F}(f)=\sigma_{A,D}^{-1}(lf)=\sigma_{A,D}^{-1}\left(\mathcal{G}(g)h\right)=g\sigma_{A,C}^{-1}(h)$$
By the assumption, there is a morphism $u:\mathcal{F}(B)\rightarrow C$ that makes the following diagram commutative:
\begin{equation*}
\begin{tikzcd}
  \mathcal{F}(A) \arrow{r}{\sigma_{A,C}^{-1}(h)} \arrow{d}[swap]{\mathcal{F}(f)}
  & [1.5em] C \arrow{d}{g}
  \\
  \mathcal{F}(B) \arrow{r}{\sigma_{B,D}^{-1}(l)} \arrow{ru}{u}
  & D
\end{tikzcd}
\end{equation*}

\noindent
Then the following diagram is commutative:
\begin{equation*}
\begin{tikzcd}
  A \arrow{r}{h} \arrow{d}[swap]{f}
  & [4em] \mathcal{G}(C) \arrow{d}{\mathcal{G}(g)}
  \\
  B \arrow{r}{l} \arrow{ru}{\sigma_{B,C}(u)}
  & \mathcal{G}(D)
\end{tikzcd}
\end{equation*}

\noindent
Indeed, using Lemma \ref{2.12} (i) and (iv), we have:
$$\sigma_{B,C}(u)f=\sigma_{A,C}\left(u\mathcal{F}(f)\right)=\sigma_{A,C}\left(\sigma_{A,C}^{-1}(h)\right)=h$$
and
$$\mathcal{G}(g)\sigma_{B,C}(u)=\sigma_{B,D}(gu)=\sigma_{B,D}\left(\sigma_{B,D}^{-1}(l)\right)=l$$
This shows that $f$ has the left lifting property against $\mathcal{G}(g)$. The converse is similarly established.
\end{proof}

\begin{corollary} \label{2.14}
Let $\mathcal{C}$ and $\mathcal{D}$ be two locally small categories, and $(\mathcal{F},\mathcal{G}):\mathcal{C} \leftrightarrows \mathcal{D}$ an adjoint pair of functors. Let $\mathcal{X}$ be a class of morphisms in $\mathcal{C}$, and $\mathcal{Y}$ a class of morphisms in $\mathcal{D}$. Then the following assertions hold:
\begin{enumerate}
\item[(i)]	For any morphism $f$ in $\mathcal{C}$, we have $\mathcal{F}(f)\in \LLP(\mathcal{Y})$ if and only if $f\in \LLP\left(\mathcal{G}(\mathcal{Y})\right)$.
\item[(ii)] For any morphism $g$ in $\mathcal{D}$, we have $g\in \RLP\left(\mathcal{F}(\mathcal{X})\right)$ if and only if $\mathcal{G}(g)\in \RLP\left(\mathcal{X}\right)$.
\end{enumerate}
\end{corollary}

\begin{proof}
(i): Let $f$ be a morphism in $\mathcal{C}$. By Lemma \ref{2.13}, $\mathcal{F}(f)\in \LLP(\mathcal{Y})$ if and only if $\mathcal{F}(f)$ has the left lifting property against every $g\in \mathcal{Y}$ if and only if $f$ has the left lifting property against every $\mathcal{G}(g)\in \mathcal{G}(\mathcal{Y})$ if and only if $f\in \LLP\left(\mathcal{G}(\mathcal{Y})\right)$.

(ii): Similar to (i).
\end{proof}

The next proposition transfers model structures along equivalences.

\begin{proposition} \label{2.15}
Let $\mathcal{C}$ be a model category, $\mathcal{D}$ a category, and $\mathcal{G}:\mathcal{D}\rightarrow \mathcal{C}$ a covariant equivalence of categories. Then $\mathcal{D}$ becomes a model category by letting fibrations (cofibrations, or weak equivalences) be morphisms $f:A\rightarrow B$ in $\mathcal{D}$ such that $\mathcal{G}(f):\mathcal{G}(A)\rightarrow \mathcal{G}(B)$ is a fibration (cofibration, or weak equivalence) in $\mathcal{C}$, respectively.
\end{proposition}

\begin{proof}
One can directly check the conditions bearing in mind that $\mathcal{G}$ is fully faithful.
\end{proof}

\begin{corollary} \label{2.16}
Let $\mathcal{C}$ be a cofibrantly generated model category, $\mathcal{D}$ a category, and $(\mathcal{F},\mathcal{G}):\mathcal{C} \leftrightarrows \mathcal{D}$ an adjoint equivalence of categories. Then $\mathcal{D}$ becomes a cofibrantly generated model category by letting fibrations (cofibrations, or weak equivalences) be morphisms $f:A\rightarrow B$ in $\mathcal{D}$ such that $\mathcal{G}(f):\mathcal{G}(A)\rightarrow \mathcal{G}(B)$ is a fibration (cofibration, or weak equivalence) in $\mathcal{C}$, respectively.
\end{corollary}

\begin{proof}
By Proposition \ref{2.15}, $\mathcal{D}$ is a model category with the specified classes of fibrations, cofibrations, and weak equivalences. Moreover, since $\mathcal{G}$ is a covariant equivalence of categories, $\mathcal{G}$ preserves direct limits. Now let $\mathcal{X}$ and $\mathcal{Y}$ be sets of morphisms as in Definition \ref{2.8}. Then $\RLP(\mathcal{X})$ is the class of trivial fibrations and $\RLP(\mathcal{Y})$ is the class of fibrations in $\mathcal{C}$. Since $\mathcal{C}$ is a model category, \cite[Proposition 7.2.3]{Hi} implies that $\LLP\left(\RLP(\mathcal{X})\right)$ is the class of cofibrations and $\LLP\left(\RLP(\mathcal{Y})\right)$ is the class of trivial cofibrations in $\mathcal{C}$. By Corollary \ref{2.14} (ii), for a given morphism $f$ in $\mathcal{D}$, we have $f\in \RLP\left(\mathcal{F}(\mathcal{X})\right)$ if and only if $\mathcal{G}(f)\in \RLP(\mathcal{X})$ if and only if $\mathcal{G}(f)$ is a trivial fibration in $C$ if and only if $f$ is a trivial fibration in $\mathcal{D}$. As a result, $\RLP\left(\mathcal{F}(\mathcal{X})\right)$ is the class of trivial fibrations in $\mathcal{D}$. A similar argument shows that $\RLP\left(\mathcal{F}(\mathcal{Y})\right)$ is the class of fibrations in $\mathcal{D}$. As $\mathcal{D}$ is a model category, $\LLP\left(\RLP\left(\mathcal{F}(\mathcal{X})\right)\right)$ is the class of cofibrations and $\LLP\left(\RLP\left(\mathcal{F}(\mathcal{Y})\right)\right)$ is the class of trivial cofibrations in $\mathcal{D}$. By definition, if $f$ is a cofibration in $\mathcal{D}$, then $\mathcal{G}(f)$ is a cofibration in $\mathcal{C}$, meaning that $\mathcal{G}\left(\LLP\left(\RLP\left(\mathcal{F}(\mathcal{X})\right)\right)\right) \subseteq \LLP\left(\RLP\left(\mathcal{X}\right)\right)$. But $\mathcal{C}$ is cofibrantly generated, so $\mathcal{X}$ permits the small object argument, whence the source of any morphism in $\mathcal{X}$ is small with respect to $\LLP\left(\RLP\left(\mathcal{X}\right)\right)$, hence with respect to $\mathcal{G}\left(\LLP\left(\RLP\left(\mathcal{F}(\mathcal{X})\right)\right)\right)$. By Lemma \ref{2.10}, $\mathcal{F}(\mathcal{X})$ permits the small object argument. A similar argument shows that $\mathcal{F}(\mathcal{Y})$ permits the small object argument. Therefore, $\mathcal{D}$ is cofibrantly generated with the corresponding sets $\mathcal{F}(\mathcal{X})$ and $\mathcal{F}(\mathcal{Y})$.
\end{proof}

\section{Model Structure on Connective Chain Complexes}

In this section, we use some ideas from \cite[Chapter 2, Section 2.3]{Ho} and \cite[Section 7]{DS} to cook up a straightforward and clear-cut proof of the model structure of connective chain complexes.

Let $R$ be a ring. Recall that an $R$-complex is a $\mathbb{Z}$-indexed sequence
$$X: \cdots \rightarrow X_{i+1} \xrightarrow{\partial_{i+1}^{X}} X_{i} \xrightarrow{\partial_{i}^{X}} X_{i-1} \rightarrow \cdots$$
of left $R$-modules and $R$-homomorphisms such that $\partial_{i}^{X}\partial_{i+1}^{X}=0$ for every $i\in \mathbb{Z}$. Given an $R$-complex $X$, we utilize the conventional notations $Z_{i}(X)= \Ker\left(\partial_{i}^{X}\right)$ and $B_{i}(X)= \im\left(\partial_{i+1}^{X}\right)$ for every $i\in \mathbb{Z}$. An $R$-complex $X$ is said to be \textit{connective} if $X_{i}=0$ for every $i<0$. The category of connective $R$-complexes is denoted by $\mathcal{C}_{\geq 0}(R)$ which is a full subcategory of the category $\mathcal{C}(R)$ of unbounded $R$-complexes. In this section, we directly show that $\mathcal{C}_{\geq 0}(R)$ is a cofibrantly generated model category. The cofibrant generation property of this model structure allows us to transfer it along adjunctions.

We first recall the Eilenberg-MacLane complexes as follows. For any $n\in \mathbb{Z}$, we set $S(n):=\Sigma^{n}R$, i.e. the $R$-complex
$$\cdots \rightarrow 0 \rightarrow R \rightarrow 0 \rightarrow \cdots$$
where $R$ is located in degree $n$. Then it is easy to verify that if $X$ is an $R$-complex, then the map $\phi_{n}^{X}: \Mor_{\mathcal{C}(R)}\left(S(n),X\right)\rightarrow Z_{n}(X)$, given by $f\mapsto f_{n}(1)$, is a natural isomorphism. Similarly, for any $n\in \mathbb{Z}$, we set $D(n):=\Sigma^{n-1}\Cone\left(1^{R}\right)$, i.e. the $R$-complex
$$\cdots \rightarrow 0 \rightarrow R \xrightarrow{1^{R}} R \rightarrow 0 \rightarrow \cdots$$
where $R$ is located in degrees $n$ and $n-1$. Then it is easy to verify that if $X$ is an $R$-complex, then the map $\psi^{X}_{n}: \Mor_{\mathcal{C}(R)}\left(D(n),X\right)\rightarrow X_{n}$, given by $f\mapsto f_{n}(1)$, is a natural isomorphism. We make repeated use of these facts in the proof of the following theorem.

\begin{theorem} \label{3.1}
Let $R$ be a ring. Then $\mathcal{C}_{\geq0}(R)$ is a cofibrantly generated model category with the model structure given as follows:
\begin{enumerate}
\item[(i)] Fibrations are morphisms $f:X\rightarrow Y$ in $\mathcal{C}_{\geq0}(R)$ such that $f_{i}:X_{i}\rightarrow Y_{i}$ is surjective for every $i\geq1$.
\item[(ii)] Cofibrations are morphisms $f:X\rightarrow Y$ in $\mathcal{C}_{\geq0}(R)$ such that $f_{i}:X_{i}\rightarrow Y_{i}$ is injective and $\Coker(f_{i})$ is a projective left $R$-module for every $i\geq 0$.
\item[(iii)] Weak equivalences are quasi-isomorphisms, i.e. morphisms $f:X\rightarrow Y$ in $\mathcal{C}_{\geq0}(R)$ such that $H_{i}(f):H_{i}(X)\rightarrow H_{i}(Y)$ is an isomorphism for every $i\geq 0$.
\end{enumerate}
\end{theorem}

\begin{prf}
We first note that $\mathcal{C}_{\geq 0}(R)$ is a locally small bicomplete category. Let
$$\mathcal{X}=\left\{\iota^{n}:S(n-1)\rightarrow D(n)\suchthat n\geq 1 \right\} \cup \left\{\lambda:0\rightarrow S(0)\right\}$$
where $\iota^{n}$ is the inclusion morphism for every $n\geq 1$. Also, let
$$\mathcal{Y}=\left\{\kappa^{n}:0\rightarrow D(n)\suchthat n\geq 1\right\}.$$
We show that $\mathcal{C}_{\geq 0}(R)$ is a cofibrantly generated model category with the corresponding sets $\mathcal{X}$ and $\mathcal{Y}$ of morphisms as in Definition \ref{2.8}. For any $n\geq 0$, we have $\Mor_{\mathcal{C}_{\geq 0}(R)}\left(S(n),-\right)\cong Z_{n}(-)$, so since the functor $Z_{n}(-)$ preserves direct limits of direct systems, we conclude that the functor $\Mor_{\mathcal{C}_{\geq 0}(R)}\left(S(n),-\right)$ preserves direct limits of direct systems as well, thereby $S(n)$ is small with respect to $\Mor\left(\mathcal{C}_{\geq0}(R)\right)$. Also, the source of $\lambda$ is $0$, so it is clearly small with respect to $\Mor\left(\mathcal{C}_{\geq0}(R)\right)$. In particular, the sources of morphisms in $\mathcal{X}$ are small with respect to $\LLP\left(\RLP(\mathcal{X})\right)$, so $\mathcal{X}$ permits the small object argument. Similarly, since the source of any morphism in $\mathcal{Y}$ is $0$, we see that $\mathcal{Y}$ permits the small object argument.

We now show that $\RLP(\mathcal{X})$ is the class of trivial fibrations in $\mathcal{C}_{\geq0}(R)$. Let $f:X\rightarrow Y$ be a morphism in $\mathcal{C}_{\geq0}(R)$. Suppose that $f\in \RLP(\mathcal{X})$. We show that $f$ is a trivial fibration, i.e. $f$ is a quasi-isomorphism and $f_{n}$ is surjective for every $n\geq 1$. Let $n\geq 0$. We first show that $H_{n}(f):H_{n}(X)\rightarrow H_{n}(Y)$ is an isomorphism. To see that $H_{n}(f)$ is injective, let $x\in Z_{n}(X)$ be such that $H_{n}(f)\left(x+B_{n}(X)\right)=f_{n}(x)+B_{n}(Y)=0$. Then $f_{n}(x)\in B_{n}(Y)$, so $f_{n}(x)=\partial^{Y}_{n+1}(y)$ for some $y\in Y_{n+1}$. Consider the natural isomorphisms $\phi_{n}^{X}:\Mor_{\mathcal{C}_{\geq 0}(R)}\left(S(n),X\right)\rightarrow Z_{n}(X)$ and $\psi^{Y}_{n+1}:\Mor_{\mathcal{C}_{\geq 0}(R)}\left(D(n+1),Y\right)\rightarrow Y_{n+1}$. There is a $g\in \Mor_{\mathcal{C}_{\geq 0}(R)}\left(S(n),X\right)$ such that $\phi_{n}^{X}(g)=g_{n}(1)=x$, and an $h\in \Mor_{\mathcal{C}_{\geq 0}(R)}\left(D(n+1),Y\right)$ such that $\psi^{Y}_{n+1}(h)=h_{n+1}(1)=y$. Then the following diagram is commutative:
\begin{equation*}
\begin{tikzcd}
  S(n) \arrow{r}{g} \arrow{d}[swap]{\iota^{n+1}}
  & X \arrow{d}{f}
  \\
  D(n+1) \arrow{r}{h}
  & Y
\end{tikzcd}
\end{equation*}

\noindent
Indeed, it is clear from the structure of $S(n)$ that $\iota^{n+1}_{i}=0=g_{i}$ for every $i\neq n$. Moreover, we have:
$$h_{n}\left(\iota_{n}^{n+1}(1)\right)=h_{n}(1)=\partial^{Y}_{n+1}\left(h_{n+1}(1)\right)=\partial^{Y}_{n+1}(y)=f_{n}(x)=f_{n}\left(g_{n}(1)\right)$$
By the choice of $f$, the above diagram can be completed to the following commutative diagram:
\begin{equation*}
  \begin{tikzcd}
  S(n) \arrow{r}{g} \arrow{d}[swap]{\iota^{n+1}}
  & X \arrow{d}{f}
  \\
  D(n+1) \arrow{r}{h} \arrow{ru}{l}
  & Y
\end{tikzcd}
\end{equation*}

\noindent
Thus $\partial_{n+1}^{X}\left(l_{n+1}(1)\right)=l_{n}(1)=l_{n}\left(\iota_{n}^{n+1}(1)\right)=g_{n}(1)=x$, so $x\in B_{n}(X)$. This means that $H_{n}(f)$ is injective. To see that $H_{n}(f)$ is surjective, let $z\in Z_{n}(Y)$. If $n=0$, then consider the isomorphism $\phi_{0}^{Y}: \Mor_{\mathcal{C}_{\geq 0}(R)}\left(S(0),Y\right)\rightarrow Z_{0}(Y)$. There is an $r\in \Mor_{\mathcal{C}_{\geq 0}(R)}\left(S(0),Y\right)$ such that $\phi_{0}^{Y}(r)=r_{0}(1)=z$. Then it is clear that the following diagram is commutative:
\begin{equation*}
  \begin{tikzcd}
  0 \arrow{r} \arrow{d}[swap]{\lambda}
  & X \arrow{d}{f}
  \\
  S(0) \arrow{r}{r}
  & Y
\end{tikzcd}
\end{equation*}

\noindent
By the choice of $f$, the above diagram can be completed to the following commutative diagram:
\begin{equation*}
  \begin{tikzcd}
  0 \arrow{r} \arrow{d}[swap]{\lambda}
  & X \arrow{d}{f}
  \\
  S(0) \arrow{r}{r}\arrow{ru}{s}
  & Y
\end{tikzcd}
\end{equation*}

\noindent
Set $w=s_{0}(1)\in X_{0}=Z_{0}(X)$. Then $f_{0}(w)=f_{0}\left(s_{0}(1)\right)=r_{0}(1)=z$. If $n\geq 1$, then consider the isomorphism $\psi^{Y}_{n}:\Mor_{\mathcal{C}_{\geq 0}(R)}\left(D(n),Y\right)\rightarrow Y_{n}$. There is an $r\in \Mor_{\mathcal{C}_{\geq 0}(R)}\left(D(n),Y\right)$ such that $\psi^{Y}_{n}(r)=r_{n}(1)=z$. Then the following diagram is commutative:
\begin{equation*}
  \begin{tikzcd}
  S(n-1) \arrow{r}{0} \arrow{d}[swap]{\iota^{n}}
  & X \arrow{d}{f}
  \\
  D(n) \arrow{r}{r}
  & Y
\end{tikzcd}
\end{equation*}

\noindent
Indeed, it is clear that $\iota_{i}^{n}=0$ for every $i\neq n-1$. Moreover, since $z\in Z_{n}(Y)$, we have $r_{n-1}\left(\iota_{n-1}^{n}(1)\right)=r_{n-1}(1)=\partial_{n}^{Y}\left(r_{n}(1)\right)=\partial_{n}^{Y}(z)=0$. By the choice of $f$, the above diagram can be completed to the following commutative diagram:
\begin{equation*}
  \begin{tikzcd}
  S(n-1) \arrow{r}{0} \arrow{d}[swap]{\iota^{n}}
  & X \arrow{d}{f}
  \\
  D(n) \arrow{r}{r} \arrow{ru}{s}
  & Y
\end{tikzcd}
\end{equation*}

\noindent
Set $w=s_{n}(1)\in X_{n}$. Then $f_{n}(w)=f_{n}\left(s_{n}(1)\right)=r_{n}(1)=z$. Moreover, $\partial_{n}^{X}(w)=\partial_{n}^{X}\left(s_{n}(1)\right)=s_{n-1}(1)=s_{n-1}\left(\iota_{n-1}^{n}(1)\right)=0$, so $w\in Z_{n}(X)$. In other words, for any $n\geq 0$ and $z\in Z_{n}(Y)$, we have found an element $w\in Z_{n}(X)$ such that $f_{n}(w)=z$. In particular, we have $H_{n}(f)\left(w+B_{n}(X)\right)=f_{n}(w)+B_{n}(Y)=z+B_{n}(Y)$. This shows that $H_{n}(f)$ is surjective. Therefore, $H_{n}(f)$ is an isomorphism. Next let $n\geq 1$. We show that $f_{n}$ is surjective. Let $v\in Y_{n}$. Then $\partial_{n}^{Y}(v)\in B_{n-1}(Y)\subseteq Z_{n-1}(Y)$, so as noted above, there is an element $u\in Z_{n-1}(X)$ such that $f_{n-1}(u)=\partial_{n}^{Y}(v)$. Consider the isomorphisms $\phi_{n-1}^{X}:\Mor_{\mathcal{C}_{\geq 0}(R)}\left(S(n-1),X\right)\rightarrow Z_{n-1}(X)$ and $\psi^{Y}_{n}:\Mor_{\mathcal{C}_{\geq 0}(R)}\left(D(n),Y\right)\rightarrow Y_{n}$. There is a $p\in \Mor_{\mathcal{C}_{\geq 0}(R)}\left(S(n-1),X\right)$ such that $\phi_{n-1}^{X}(p)=p_{n-1}(1)=u$, and a $q\in \Mor_{\mathcal{C}_{\geq 0}(R)}\left(D(n),Y\right)$ such that $\psi^{Y}_{n}(q)=q_{n}(1)=v$. Then the following diagram is commutative:
\begin{equation*}
  \begin{tikzcd}
  S(n-1) \arrow{r}{p} \arrow{d}[swap]{\iota^{n}}
  & X \arrow{d}{f}
  \\
  D(n) \arrow{r}{q}
  & Y
\end{tikzcd}
\end{equation*}

\noindent
Indeed, it is clear that $\iota_{i}^{n}=0=p_{i}$ for every $i\neq n-1$. Moreover, we have:
$$q_{n-1}\left(\iota_{n-1}^{n}(1)\right)=q_{n-1}(1)=\partial^{Y}_{n}\left(q_{n}(1)\right)=\partial^{Y}_{n}(v)=f_{n-1}(u)=f_{n-1}\left(p_{n-1}(1)\right)$$
By the choice of $f$, the above diagram can be completed to the following commutative diagram:
\begin{equation*}
  \begin{tikzcd}
  S(n-1) \arrow{r}{p} \arrow{d}[swap]{\iota^{n}}
  & X \arrow{d}{f}
  \\
  D(n) \arrow{r}{q} \arrow{ru}{t}
  & Y
\end{tikzcd}
\end{equation*}

\noindent
Then $f_{n}\left(t_{n}(1)\right)=q_{n}(1)=v$. Thus $f_{n}$ is surjective. As a result, $f$ is a trivial fibration.

Conversely, suppose that $f$ is a trivial fibration, i.e. $f$ is a quasi-isomorphism and $f_{n}$ is surjective for every $n\geq 1$. We first show that $f_{0}$ is also surjective. Consider the following commutative diagram:
\begin{equation*}
  \begin{tikzcd}
  X_{1} \arrow{r}{\overline{\partial_{1}^{X}}} \arrow{d}[swap]{f_{1}}
  & B_{0}(X) \arrow{d}{B_{0}(f)}
  \\
  Y_{1} \arrow{r}{\overline{\partial_{1}^{Y}}}
  & B_{0}(Y)
\end{tikzcd}
\end{equation*}

\noindent
As $f_{1}$ is surjective, the above diagram implies that $B_{0}(f)$ is surjective. Now consider the following commutative diagram:
\begin{equation*}
  \begin{tikzcd}
  0 \arrow{r}
  & B_{0}(X) \arrow{r} \arrow{d}{B_{0}(f)}
  & X_{0} \arrow{r} \arrow{d}{f_{0}}
  & H_{0}(X) \arrow{r} \arrow{d}{H_{0}(f)}
  & 0
  \\
  0 \arrow{r}
  & B_{0}(Y) \arrow{r}
  & Y_{0} \arrow{r}
  & H_{0}(Y) \arrow{r}
  & 0
\end{tikzcd}
\end{equation*}

\noindent
Since $B_{0}(f)$ is surjective and $H_{0}(f)$ is an isomorphism, the Short Five Lemma implies that $f_{0}$ is surjective. Therefore, $f$ is a surjective quasi-isomorphism. We next show that $f\in \RLP(\mathcal{X})$. Let $n\geq 1$, and consider a commutative diagram as follows:
\begin{equation*}
  \begin{tikzcd}
  S(n-1) \arrow{r}{p} \arrow{d}[swap]{\iota^{n}}
  & X \arrow{d}{f}
  \\
  D(n) \arrow{r}{q}
  & Y
\end{tikzcd}
\end{equation*}

\noindent
Let $x=p_{n-1}(1)\in X_{n-1}$ and $y=q_{n}(1)\in Y_{n}$. Then we have:
$$f_{n-1}(x)=f_{n-1}\left(p_{n-1}(1)\right)=q_{n-1}\left(\iota_{n-1}^{n}(1)\right)=q_{n-1}(1)=\partial_{n}^{Y}\left(q_{n}(1)\right)=\partial_{n}^{Y}(y)$$
Moreover, $\partial_{n-1}^{X}(x)=\partial_{n-1}^{X}\left(p_{n-1}(1)\right)=0$, so $x\in Z_{n-1}(X)$. But then $H_{n-1}(f)\left(x+B_{n-1}(X)\right)=f_{n-1}(x)+B_{n-1}(Y)=\partial_{n}^{Y}(y)+B_{n-1}(Y)=0$, so since $H_{n-1}(f)$ is an isomorphism, we get $x+B_{n-1}(X)=0$, i.e. $x\in B_{n-1}(X)$, so $x=\partial_{n}^{X}(u)$ for some $u\in X_{n}$. As $f$ is a morphism of $R$-complexes, we have $\partial_{n}^{Y}(y)=f_{n-1}(x)=f_{n-1}\left(\partial_{n}^{X}(u)\right)=\partial_{n}^{Y}\left(f_{n}(u)\right)$, so $y-f_{n}(u)\in Z_{n}(Y)$. But $f$ is a surjective quasi-isomorphism, so it is easy to see that $Z(f):Z(X)\rightarrow Z(Y)$ is surjective. In particular, there is an element $v\in Z_{n}(X)$ such that $f_{n}(v)=y-f_{n}(u)$. Set $z=u+v\in X_{n}$. Then $\partial_{n}^{X}(z)=\partial_{n}^{X}(u)+\partial_{n}^{X}(v)=\partial_{n}^{X}(u)=x$ and $f_{n}(z)=f_{n}(u)+f_{n}(v)=y$. Consider the isomorphism $\psi^{X}_{n}:\Mor_{\mathcal{C}_{\geq 0}(R)}\left(D(n),X\right)\rightarrow X_{n}$. There is an $l\in \Mor_{\mathcal{C}_{\geq 0}(R)}\left(D(n),X\right)$ such that $\psi^{X}_{n}(l)=l_{n}(1)=z$. Then the following diagram is commutative:
\begin{equation*}
  \begin{tikzcd}
  S(n-1) \arrow{r}{p} \arrow{d}[swap]{\iota^{n}}
  & X \arrow{d}{f}
  \\
  D(n) \arrow{r}{q} \arrow{ru}{l}
  & Y
\end{tikzcd}
\end{equation*}

\noindent
Indeed, $\iota_{i}^{n}=0$ for every $i\neq n-1$. Furthermore, we have:
$$l_{n-1}\left(\iota_{n-1}^{n}(1)\right)=l_{n-1}(1)=\partial_{n}^{X}\left(l_{n}(1)\right)=\partial_{n}^{X}(z)=x=p_{n-1}(1)$$
On the other hand, $q_{i}=0=l_{i}$ for every $i\neq n,n-1$. Moreover, we have $f_{n}\left(l_{n}(1)\right)=f_{n}(z)=y=q_{n}(1)$ and
$$f_{n-1}\left(l_{n-1}(1)\right)=f_{n-1}\left(\partial_{n}^{X}\left(l_{n}(1)\right)\right)=f_{n-1}\left(\partial_{n}^{X}(z)\right) =f_{n-1}(x)=\partial_{n}^{Y}(y)=\partial_{n}^{Y}\left(q_{n}(1)\right)=q_{n-1}(1).$$
Next consider a commutative diagram as follows:
\begin{equation*}
  \begin{tikzcd}
  0 \arrow{r} \arrow{d}[swap]{\lambda}
  & X \arrow{d}{f}
  \\
  S(0) \arrow{r}{r}
  & Y
\end{tikzcd}
\end{equation*}

\noindent
By the surjectivity of $f_{0}$, there is an element $w\in X_{0}=Z_{0}(X)$ such that $f_{0}(w)=r_{0}(1)$. Consider the isomorphism $\phi_{0}^{X}: \Mor_{\mathcal{C}_{\geq 0}(R)}\left(S(0),X\right)\rightarrow Z_{0}(X)$. There is an $s\in \Mor_{\mathcal{C}_{\geq 0}(R)}\left(S(0),X\right)$ such that $\phi_{0}^{X}(s)=s_{0}(1)=w$. Then the following diagram is commutative:
\begin{equation*}
  \begin{tikzcd}
  0 \arrow{r} \arrow{d}[swap]{\lambda}
  & X \arrow{d}{f}
  \\
  S(0) \arrow{r}{r} \arrow{ru}{s}
  & Y
\end{tikzcd}
\end{equation*}

\noindent
Indeed, it is clear that the upper triangle is commutative. Moreover, for the lower triangle, we note that $s_{i}=0=r_{i}$ for every $i\neq 0$. Also, $f_{0}\left(s_{0}(1)\right)=f_{0}(w)=r_{0}(1)$. This means that $f\in \RLP(\mathcal{X})$. Therefore, $\RLP(\mathcal{X})$ is the class of trivial fibrations.

We next show that $\RLP(\mathcal{Y})$ is the class of fibrations in $\mathcal{C}_{\geq0}(R)$. Let $f:X\rightarrow Y$ be a morphism in $\mathcal{C}_{\geq0}(R)$. Suppose that $f\in \RLP(\mathcal{Y})$. We show that $f$ is a fibration, i.e. $f_{n}$ is surjective for every $n\geq 1$. Let $n\geq 1$ and $v\in Y_{n}$. Consider the isomorphism $\psi^{Y}_{n}:\Mor_{\mathcal{C}_{\geq 0}(R)}\left(D(n),Y\right)\rightarrow Y_{n}$. There is a $q\in \Mor_{\mathcal{C}_{\geq 0}(R)}\left(D(n),Y\right)$ such that $\psi^{Y}_{n}(q)=q_{n}(1)=v$. By the choice of $f$, the commutative diagram
\begin{equation*}
  \begin{tikzcd}
  0 \arrow{r} \arrow{d}[swap]{\kappa^{n}}
  & X \arrow{d}{f}
  \\
  D(n) \arrow{r}{q}
  & Y
\end{tikzcd}
\end{equation*}

\noindent
can be completed to the following commutative diagram:
\begin{equation*}
  \begin{tikzcd}
  0 \arrow{r} \arrow{d}[swap]{\kappa^{n}}
  & X \arrow{d}{f}
  \\
  D(n) \arrow{r}{q} \arrow{ru}{t}
  & Y
\end{tikzcd}
\end{equation*}

\noindent
Then $f_{n}\left(t_{n}(1)\right)=q_{n}(1)=v$. This shows that $f_{n}$ is surjective. As a result, $f$ is a fibration.

Conversely, suppose that $f$ is a fibration, i.e. $f_{n}$ is surjective for every $n\geq 1$. Let $n\geq 1$, and consider a commutative diagram as follows:
\begin{equation*}
  \begin{tikzcd}
  0 \arrow{r} \arrow{d}[swap]{\kappa^{n}}
  & X \arrow{d}{f}
  \\
  D(n) \arrow{r}{q}
  & Y
\end{tikzcd}
\end{equation*}

\noindent
Since $f_{n}$ is surjective, there is an element $x\in X_n$ such that $f_{n}(x)=q_{n}(1)$. Consider the isomorphism $\psi^{X}_{n}:\Mor_{\mathcal{C}_{\geq 0}(R)}\left(D(n),X\right)\rightarrow X_{n}$. There is an $l\in \Mor_{\mathcal{C}_{\geq 0}(R)}\left(D(n),X\right)$ such that $\psi^{X}_{n}(l)=l_{n}(1)=x$. Then the following diagram is commutative:
\begin{equation*}
  \begin{tikzcd}
  0 \arrow{r} \arrow{d}[swap]{\kappa^{n}}
  & X \arrow{d}{f}
  \\
  D(n) \arrow{r}{q} \arrow{ru}{l}
  & Y
\end{tikzcd}
\end{equation*}

\noindent
Indeed, it is obvious that the upper triangle is commutative. For the lower triangle, we note that $q_{i}= 0 =l_{i}$ for every $i\neq n,n-1$. Moreover, using the fact that $f$ is a morphism of $R$-complexes, we have $f_{n}\left(l_{n}(1)\right)=f_{n}(x)=q_{n}(1)$ and
$$f_{n-1}\left(l_{n-1}(1)\right)=f_{n-1}\left(\partial_{n}^{X}\left(l_{n}(1)\right)\right)=f_{n-1}\left(\partial_{n}^{X}(x)\right) =\partial_{n}^{Y}\left(f_{n}(x)\right)=\partial_{n}^{Y}\left(q_{n}(1)\right)=q_{n-1}(1).$$
This implies that $f\in \RLP(\mathcal{Y})$. Therefore, $\RLP(\mathcal{Y})$ is the class of fibrations.

We are now ready to check the axioms of the model category as follows:

Retract: Suppose that $f$ is a retract of $g$, i.e. there is a commutative diagram
\begin{equation*}
  \begin{tikzcd}
  X \arrow{r}{h} \arrow{d}{f}
  & Z \arrow{d}{g} \arrow{r}{l}
  & X \arrow{d}{f}
  \\
  Y \arrow{r}{p}
  & W \arrow{r}{q}
  & Y
\end{tikzcd}
\end{equation*}

\noindent
in $\mathcal{C}_{\geq0}(R)$ in which $lh=1^{X}$ and $qp=1^{Y}$. If $g$ is a fibration, then $g_{i}$ is surjective for every $i\geq 1$. On the other hand, the relation $q_{i}p_{i}=1^{Y_{i}}$ implies that $q_{i}$ is surjective for every $i\geq 0$. Therefore, $q_{i}g_{i}$ is surjective for every $i\geq 1$. Now the relation $f_{i}l_{i}=q_{i}g_{i}$ yields that $f_{i}$ is surjective for every $i\geq 1$. That means that $f$ is a fibration.

If $g$ is a cofibration, then $g_{i}$ is injective and $\Coker(g_{i})$ is projective for every $i\geq 0$. The relation $l_{i}h_{i}=1^{X_{i}}$ implies that $h_{i}$ is injective for every $i\geq 0$. Therefore, $g_{i}h_{i}$ is injective for every $i\geq 0$. Now the relation $p_{i}f_{i}=g_{i}h_{i}$ yields that $f_{i}$ is injective for every $i\geq 0$. Furthermore, we get the following induced commutative diagram:
\begin{equation*}
  \begin{tikzcd}
  X \arrow{r}{h} \arrow{d}{f}
  & Z \arrow{d}{g} \arrow{r}{l}
  & X \arrow{d}{f}
  \\
  Y \arrow{r}{p} \arrow{d}
  & W \arrow{r}{q} \arrow{d}
  & Y \arrow{d}
  \\
  \Coker(f) \arrow{r}{\bar{p}}
  & \Coker(g) \arrow{r}{\bar{q}}
  & \Coker(f)
\end{tikzcd}
\end{equation*}

\noindent
It is also evident that $\bar{q}\bar{p}=1^{\Coker(f)}$. It follows that $\Coker(f)$ is isomorphic to a direct summand of $\Coker(g)$. Therefore, $\Coker(f_{i})$ is projective for every $i\geq 0$. Thus $f$ is a cofibration.

Finally, if $g$ is a weak equivalence, then $g$ is a quasi-isomorphism. Fix $i\geq 0$. The functoriality of homology yields the following commutative diagram:
\begin{equation*}
  \begin{tikzcd}[column sep=3em,row sep=2em]
  H_{i}(X) \arrow{r}{H_{i}(h)} \arrow{d}{H_{i}(f)}
  & H_{i}(Z) \arrow{d}{H_{i}(g)} \arrow{r}{H_{i}(l)}
  & H_{i}(X) \arrow{d}{H_{i}(f)}
  \\
  H_{i}(Y) \arrow{r}{H_{i}(p)}
  & H_{i}(W) \arrow{r}{H_{i}(q)}
  & H_{i}(Y)
\end{tikzcd}
\end{equation*}

\noindent
Setting $\phi=H_{i}(l)H_{i}(g)^{-1}H_{i}(p):H_{i}(Y)\rightarrow H_{i}(X)$, we observe from the above commutative diagram that $\phi$ is the inverse of $H_{i}(f)$, so $H_{i}(f)$ is an isomorphism. That is to say, $f$ is a quasi-isomorphism, hence a weak equivalence.

Two-Out-of-Three: Suppose that we have morphisms $f=gh$ in $\mathcal{C}_{\geq0}(R)$ in which two of $f$, $g$, or $h$ are quasi-isomorphisms. Then by the functoriality of homology, it is clear that the third morphism must be a quasi-isomorphism as well.

First Half of Lifting: Consider a commutative diagram
\begin{equation*}
  \begin{tikzcd}
  X \arrow{r}{h} \arrow{d}[swap]{f}
  & Z \arrow{d}{g}
  \\
  Y \arrow{r}{l}
  & W
\end{tikzcd}
\end{equation*}

\noindent
in $\mathcal{C}_{\geq0}(R)$ in which $f$ is a cofibration and $g$ is a trivial fibration. By definition, $f_{i}:X_{i}\rightarrow Y_{i}$ is injective and $\Coker(f_{i})$ is projective for every $i\geq 0$. As a result, for any $i\geq 0$, the short exact sequence
$$0\rightarrow X_{i}\xrightarrow{f_{i}} Y_{i}\rightarrow \Coker(f_{i})\rightarrow 0$$
is split, resulting in a decomposition $Y_{i} = \im(f_{i})\oplus P_{i}$ for some projective left $R$-module $P_{i} \cong \Coker(f_{i})$. Moreover, as noted before, $g:X\rightarrow Y$ is a surjective quasi-isomorphism, so there is a short exact sequence
$$0 \rightarrow \Ker(g) \xrightarrow{\iota} Z \xrightarrow{g} W \rightarrow 0$$
of $R$-complexes in which $\Ker(g)$ is exact. Now we construct a morphism $\phi =(\phi_{i})_{i\geq 0}:Y \rightarrow Z$ of $R$-complexes inductively. Let $i=0$, and consider the following commutative diagram:
\begin{equation*}
  \begin{tikzcd}
  X_{0} \arrow{r}{h_{0}} \arrow{d}[swap]{f_{0}}
  & Z_{0} \arrow{d}{g_{0}}
  \\
  Y_{0} \arrow{r}{l_{0}}
  & W_{0}
\end{tikzcd}
\end{equation*}

\noindent
Since $g_{0}$ is surjective and $P_{0}$ is projective, there is an $R$-homomorphism $\rho_{0}:P_{0} \rightarrow Z_{0}$ that makes the following diagram commutative:
\begin{equation*}
  \begin{tikzcd}
  & P_{0} \arrow{d}{l_{0}\mid _{P_{0}}} \arrow{dl}[swap]{\rho_{0}}
  \\
  Z_{0} \arrow{r}{g_{0}}
  & W_{0}
\end{tikzcd}
\end{equation*}

\noindent
If $y\in Y_{0} = \im(f_{0}) \oplus P_{0}$, then we can uniquely write $y=f_{0}(x)+z$ for some $x\in X_{0}$ and $z\in P_{0}$. Define a map $\phi_{0}:Y_{0} \rightarrow Z_{0}$ by setting $\phi_{0}(y)=h_{0}(x)+\rho_{0}(z)$. Then it is clear that $\phi_{0}$ is an $R$-homomorphism. Moreover, the following diagram is commutative:
\begin{equation*}
  \begin{tikzcd}
  X_{0} \arrow{r}{h_{0}} \arrow{d}[swap]{f_{0}}
  & Z_{0} \arrow{d}{g_{0}}
  \\
  Y_{0} \arrow{r}{l_{0}} \arrow{ur}{\phi_{0}}
  & W_{0}
\end{tikzcd}
\end{equation*}

\noindent
Indeed, if $x \in X_{0}$, then using the definition, we have $\phi_{0}\left(f_{0}(x)\right)=h_{0}(x)$. Further, if $y=f_{0}(x)+z \in Y_{0} = \im(f_{0}) \oplus P_{0}$, then we have:
$$g_{0}\left(\phi_{0}(y)\right)=g_{0}\left(h_{0}(x)\right)+g_{0}\left(\rho_{0}(z)\right)=l_{0}\left(f_{0}(x)\right)+l_{0}(z)=l_{0}\left(f_{0}(x)+z\right) =l_{0}(y)$$
Now suppose that $n\geq 1$, and for any $0\leq i \leq n-1$, the $R$-homomorphism $\phi_{i}:Y_{i} \rightarrow Z_{i}$ is constructed in a way that the following diagrams are commutative:
\[
 \begin{tikzcd}
  X_{i} \arrow{r}{h_{i}} \arrow{d}[swap]{f_{i}}
  & Z_{i} \arrow{d}{g_{i}}
  \\
  Y_{i} \arrow{r}{l_{i}} \arrow{ur}{\phi_{i}}
  & W_{i}
\end{tikzcd}
\quad \text{and} \quad
 \begin{tikzcd}
  Y_{i} \arrow{r}{\partial_{i}^{Y}} \arrow{d}[swap]{\phi_{i}}
  & Y_{i-1} \arrow{d}{\phi_{i-1}}
  \\
  Z_{i} \arrow{r}{\partial_{i}^{Z}}
  & Z_{i-1}
\end{tikzcd}
\]

\noindent
That is, $\phi_{i}f_{i}=h_{i}$, $g_{i}\phi_{i}=l_{i}$, and $\partial_{i}^{Z}\phi_{i}=\phi_{i-1}\partial_{i}^{Y}$ for every $0 \leq i \leq n-1$. By the same argument as in the case $i=0$, we can obtain an $R$-homomorphism $\rho_{n}:P_{n} \rightarrow Z_{n}$ that makes the diagram
\begin{equation*}
  \begin{tikzcd}
  & P_{n} \arrow{d}{l_{n}\mid _{P_{n}}} \arrow{dl}[swap]{\rho_{n}}
  \\
  Z_{n} \arrow{r}{g_{n}}
  & W_{n}
\end{tikzcd}
\end{equation*}

\noindent
commutative, and we can construct an $R$-homomorphism $\psi_{n}:Y_{n} \rightarrow Z_{n}$, given by $\psi_{n}(y)=h_{n}(x)+\rho_{n}(z)$ for every $y=f_{n}(x)+z \in Y_{n}=\im(f_{n})\oplus P_{n}$, that makes the following diagram commutative:
\begin{equation*}
  \begin{tikzcd}
  X_{n} \arrow{r}{h_{n}} \arrow{d}[swap]{f_{n}}
  & Z_{n} \arrow{d}{g_{n}}
  \\
  Y_{n} \arrow{r}{l_{n}} \arrow{ur}{\psi_{n}}
  & W_{n}
\end{tikzcd}
\end{equation*}

\noindent
That is, $\psi_{n}f_{n}=h_{n}$ and $g_{n}\psi_{n}=l_{n}$. However, the following diagram need not be commutative:
\begin{equation*}
  \begin{tikzcd}
  Y_{n} \arrow{r}{\partial_{n}^{Y}} \arrow{d}[swap]{\psi_{n}}
  & Y_{n-1} \arrow{d}{\phi_{n-1}}
  \\
  Z_{n} \arrow{r}{\partial_{n}^{Z}}
  & Z_{n-1}
\end{tikzcd}
\end{equation*}

\noindent
To fix this, we set $\zeta_{n}=\partial_{n}^{Z}\psi_{n}-\phi_{n-1}\partial_{n}^{Y}:Y_{n} \rightarrow Z_{n-1}$. Using the fact that $l:Y \rightarrow W$ and $g:Z \rightarrow W$ are morphisms of $R$-complexes, we have for every $z \in P_{n} \subseteq Y_{n}$:
\begin{equation*}
\begin{split}
 g_{n-1}\left(\zeta_{n}(z)\right) & = g_{n-1}\left(\partial_{n}^{Z}\left(\psi_{n}(z)\right)\right)-g_{n-1}\left(\phi_{n-1}\left(\partial_{n}^{Y}(z)\right)\right) = g_{n-1}\left(\partial_{n}^{Z}\left(\rho_{n}(z)\right)\right)-l_{n-1}\left(\partial_{n}^{Y}(z)\right) \\
 & = \partial_{n}^{W}\left(g_{n}\left(\rho_{n}(z)\right)\right)-\partial_{n}^{W}\left(l_{n}(z)\right) = \partial_{n}^{W}\left(l_{n}(z)\right)-\partial_{n}^{W}\left(l_{n}(z)\right) = 0
\end{split}
\end{equation*}

\noindent
That is, $\zeta_{n}(z)\in \Ker(g_{n-1})$. Therefore, we have for every $z \in P_{n} \subseteq Y_{n}$:
\begin{equation*}
\begin{split}
 \partial_{n-1}^{\Ker(g)}\left(\zeta_{n}(z)\right) & = \partial_{n-1}^{Z}\left(\zeta_{n}(z)\right) = \partial_{n-1}^{Z}\left(\partial_{n}^{Z}\left(\psi_{n}(z)\right)\right)-\partial_{n-1}^{Z}\left(\phi_{n-1}\left(\partial_{n}^{Y}(z)\right)\right) \\
 & = \partial_{n-1}^{Z}\left(\partial_{n}^{Z}\left(\psi_{n}(z)\right)\right)-\phi_{n-2}\left(\partial_{n-1}^{Y}\left(\partial_{n}^{Y}(z)\right)\right) = 0
\end{split}
\end{equation*}

\noindent
That is, $\zeta_{n}(z) \in Z_{n-1}\left(\Ker(g)\right)$. But $\Ker(g)$ is exact, so $\zeta_{n}(z) \in B_{n-1}\left(\Ker(g)\right)$. Therefore, $\zeta_{n}$ induces an $R$-homomorphism $\overline{\zeta_{n}\mid_{P_{n}}}:P_{n} \rightarrow B_{n-1}\left(\Ker(g)\right)$. By the projectivity of $P_{n}$, there is an $R$-homomorphism $\mu_{n}:P_{n} \rightarrow \Ker(g_{n})$ that makes the following diagram commutative:
\begin{equation*}
  \begin{tikzcd}[column sep=4em,row sep=2.5em]
  & P_{n} \arrow{d}{\overline{\zeta_{n}\mid_{P_{n}}}} \arrow[dl, bend right, swap, "\mu_{n}"]
  \\
  \Ker(g_{n}) \arrow{r}{\overline{\partial_{n}^{\Ker(g)}}}
  & B_{n-1}\left(\Ker(g)\right)
\end{tikzcd}
\end{equation*}

\noindent
Let $\pi_{n}:Y_{n} = \im(f_{n})\oplus P_{n} \rightarrow P_{n}$ be the canonical projection. Also, let $\chi_{n}:Y_{n} \rightarrow Z_{n}$ be the following composition:
$$Y_{n} \xrightarrow{\pi_{n}} P_{n} \xrightarrow{\mu_{n}} \Ker(g_{n}) \xrightarrow{\iota_{n}} Z_{n}$$
Set $\phi_{n}=\psi_{n}-\chi_{n}:Y_{n} \rightarrow Z_{n}$. Then the following diagram is commutative:
\begin{equation*}
  \begin{tikzcd}
  X_{n} \arrow{r}{h_{n}} \arrow{d}[swap]{f_{n}}
  & Z_{n} \arrow{d}{g_{n}}
  \\
  Y_{n} \arrow{r}{l_{n}} \arrow{ur}{\phi_{n}}
  & W_{n}
\end{tikzcd}
\end{equation*}

\noindent
Indeed, we have for every $x \in X_{n}$:
$$\phi_{n}\left(f_{n}(x)\right)=\psi_{n}\left(f_{n}(x)\right)-\chi_{n}\left(f_{n}(x)\right)= h_{n}(x)-\iota_{n}\left(\mu_{n}\left(\pi_{n}\left(f_{n}(x)\right)\right)\right)=h_{n}(x)$$
Also, we have for every $y=f_{n}(x)+z \in Y_{n} = Im(f_{n}) \oplus P_{n}$:
$$g_{n}\left(\phi_{n}(y)\right) = g_{n}\left(\psi_{n}(y)\right)-g_{n}\left(\chi_{n}(y)\right) = l_{n}(y)-g_{n}\left(\iota_{n}\left(\mu_{n}\left(\pi_{n}(y)\right)\right)\right) = l_{n}(y)-g_{n}\left(\mu_{n}(z)\right) = l_{n}(y)$$
In addition, the following diagram is commutative:
\begin{equation*}
  \begin{tikzcd}
  Y_{n} \arrow{r}{\partial_{n}^{Y}} \arrow{d}[swap]{\phi_{n}}
  & Y_{n-1} \arrow{d}{\phi_{n-1}}
  \\
  Z_{n} \arrow{r}{\partial_{n}^{Z}}
  & Z_{n-1}
\end{tikzcd}
\end{equation*}

\noindent
Indeed, noting that $f:X \rightarrow Y$ and $h:X \rightarrow Z$ are morphisms of $R$-complexes, we have for every $y=f_{n} (x)+z \in Y_{n} = \im(f_{n}) \oplus P_{n}$:
\begin{equation*}
\begin{split}
 \partial_{n}^{Z}\left(\phi_{n}(y)\right) & = \partial_{n}^{Z}\left(\psi_{n}(y)\right)- \partial_{n}^{Z}\left(\chi_{n}(y)\right) \\
 & = \partial_{n}^{Z}\left(\psi_{n}(y)\right) - \partial_{n}^{Z}\left(\iota_{n}\left(\mu_{n}\left(\pi_{n}(y)\right)\right)\right) \\
 & = \partial_{n}^{Z}\left(\psi_{n}(y)\right) - \partial_{n}^{Z}\left(\mu_{n}(z)\right) \\
 & = \partial_{n}^{Z}\left(\psi_{n}(y)\right) - \partial_{n}^{\Ker(g)}\left(\mu_{n}(z)\right) \\
 & = \partial_{n}^{Z}\left(\psi_{n}(y)\right) - \zeta_{n}(z) \\
 & = \partial_{n}^{Z}\left(\psi_{n}(y)\right) - \partial_{n}^{Z}\left(\psi_{n}(z)\right) + \phi_{n-1}\left(\partial_{n}^{Y}(z)\right) \\
 & = \partial_{n}^{Z}\left(\psi_{n}(y-z)\right) + \phi_{n-1}\left(\partial_{n}^{Y}\left(y-f_{n}(x)\right)\right) \\
 & = \partial_{n}^{Z}\left(\psi_{n}\left(f_{n}(x)\right)\right) + \phi_{n-1}\left(\partial_{n}^{Y}(y)\right) - \phi_{n-1}\left(\partial_{n}^{Y}\left(f_{n}(x)\right)\right) \\
 & = \partial_{n}^{Z}\left(\psi_{n}\left(f_{n}(x)\right)\right) + \phi_{n-1}\left(\partial_{n}^{Y}(y)\right) - \phi_{n-1}\left(f_{n-1}\left(\partial_{n}^{X}(x)\right)\right) \\
 & = \partial_{n}^{Z}\left(h_{n}(x)\right) + \phi_{n-1}\left(\partial_{n}^{Y}(y)\right) - h_{n-1}\left(\partial_{n}^{X}(x)\right) \\
 & = \partial_{n}^{Z}\left(h_{n}(x)\right) + \phi_{n-1}\left(\partial_{n}^{Y}(y)\right) - \partial_{n}^{Z}\left(h_{n}(x)\right) \\
 & = \phi_{n-1}\left(\partial_{n}^{Y}(y)\right)
\end{split}
\end{equation*}

\noindent
As a consequence, we can continue this construction inductively and find a morphism $\phi=(\phi_{i})_{i \geq 0}:Y \rightarrow Z$ of $R$-complexes that makes the following diagram commutative:
\begin{equation*}
  \begin{tikzcd}
  X \arrow{r}{h} \arrow{d}[swap]{f}
  & Z \arrow{d}{g}
  \\
  Y \arrow{r}{l} \arrow{ru}{\phi}
  & W
\end{tikzcd}
\end{equation*}

First Half of Factorization: Let $f:X \rightarrow Y$ be a morphism in $\mathcal{C}_{\geq 0}(R)$. For any $n \geq 1$, consider the isomorphism $\psi_{n}^{Y}:\Mor_{\mathcal{C}_{\geq 0}(R)}\left(D(n),Y\right) \rightarrow Y_n$. For any $n \geq 1$ and $x \in Y_{n}$, there is a $p^{n,x} \in \Mor_{\mathcal{C}_{\geq 0}(R)}\left(D(n),Y\right)$ such that $\psi_{n}^{Y}\left(p^{n,x}\right)=p^{n,x}_{n}(1)=x$. Let $P=\bigoplus_{n=1}^{\infty}\bigoplus_{x \in Y_{n}}D(n)$. By the universal property of direct sum, there is a unique morphism $p:P \rightarrow Y$ of $R$-complexes that makes the diagram
\begin{equation*}
  \begin{tikzcd}
  D(n) \arrow{r}{\iota^{n,x}} \arrow{d}[swap]{p^{n,x}}
  & P \arrow{dl}{p}
  \\
  Y
  &
\end{tikzcd}
\end{equation*}

\noindent
commutative in which $\iota^{n,x}$ is the canonical injection for every $n \geq 1$ and $x \in Y_{n}$. Thus if $n\geq 1$ and $x\in Y_{n}$, then $p_{n}\left(\iota_{n}^{n,x}(1)\right)=p_{n}^{n,x}(1)=x$. This shows that $p_{n}$ is surjective for every $n\geq 1$. In addition, there is a unique morphism $\eta:X \oplus P \rightarrow U$ of $R$-complexes that makes the diagram
\begin{equation*}
  \begin{tikzcd}
  X \arrow{r}{\kappa} \arrow{dr}[swap]{f}
  & X \oplus P \arrow{d}{\eta}
  & P \arrow{l}[swap]{\nu} \arrow{dl}{p}
  \\
  & Y
  &
\end{tikzcd}
\end{equation*}

\noindent
commutative in which $\kappa$ and $\nu$ are canonical injections. For any $i \geq 1$, $\eta_{i}\nu_{i}=p_{i}$ is surjective, so $\eta_{i}$ is surjective. This implies that $\eta$ is a fibration. On the other hand, for any $i \geq 0$, $\kappa_{i}$ is injective, and $\Coker(\kappa_{i})\cong P_{i}$ is a direct sum of copies of $R$, so $\Coker(\kappa_{i})$ is projective. That is, $\kappa$ is a cofibration. Also, for any $i \geq 0$, we have
$$H_{i}(P)=H_{i}\left(\textstyle\bigoplus_{n=1}^{\infty}\textstyle\bigoplus_{x \in Y_{n}}D(n)\right)\cong \textstyle\bigoplus_{n=1}^{\infty}\bigoplus_{x \in Y_{n}}H_{i}\left(D(n)\right)=0,$$
which implies that $H_{i}(\kappa):H_{i}(X) \rightarrow H_{i}(X \oplus P)$ is an isomorphism, so $\kappa$ is a quasi-isomorphism. Therefore, $\kappa$ is a trivial cofibration. Now $f=\eta\kappa$ is a factorization of $f$ in which $\eta$ is a fibration and $\kappa$ is a trivial cofibration. We further show that $\kappa$ has the left lifting property against all fibrations. To this end, let $t:U \rightarrow V$ be a fibration in $\mathcal{C}_{\geq 0}(R)$. Consider a commutative diagram as follows:
\begin{equation*}
  \begin{tikzcd}
  X \arrow{r}{r} \arrow{d}[swap]{\kappa}
  & U \arrow{d}{t}
  \\
  X \oplus P \arrow{r}{s}
  & V
\end{tikzcd}
\end{equation*}

\noindent
For any $n \geq 1$ and $x \in Y_{n}$, let $u^{n,x}:D(n) \rightarrow V$ be the following composition:
$$D(n) \xrightarrow{\iota^{n,x}} P \xrightarrow{\nu} X \oplus P \xrightarrow{s} V$$
As we observed above, $\RLP(\mathcal{Y})$ is the class of fibrations, so $t \in \RLP(\mathcal{Y})$. It follows that for any $n \geq 1$ and $x\in Y_{n}$, there is a morphism $q^{n,x}:D(n) \rightarrow U$ that makes the following diagram commutative:
\begin{equation*}
  \begin{tikzcd}[column sep=3em,row sep=2em]
  0 \arrow{r} \arrow{d}[swap]{\kappa^{n}}
  & U \arrow{d}{t}
  \\
  D(n) \arrow{r}{u^{n,x}} \arrow{ru}{q^{n,x}}
  & V
\end{tikzcd}
\end{equation*}

\noindent
By the universal property of direct sum, there is a unique morphism $q:P \rightarrow U$ of $R$-complexes that makes the following diagram commutative for every $n \geq 1$ and $x \in Y_{n}$:
\begin{equation*}
  \begin{tikzcd}
  D(n) \arrow{r}{\iota^{n,x}} \arrow{d}[swap]{q^{n,x}}
  & P \arrow{dl}{q}
  \\
  U
  &
\end{tikzcd}
\end{equation*}

\noindent
In addition, there is a unique morphism $\theta:X \oplus P \rightarrow U$ of $R$-complexes that makes the following diagram commutative:
\begin{equation*}
  \begin{tikzcd}
  X \arrow{r}{\kappa} \arrow{dr}[swap]{r}
  & X \oplus P \arrow{d}{\theta}
  & P \arrow{l}[swap]{\nu} \arrow{dl}{q}
  \\
  & U
  &
\end{tikzcd}
\end{equation*}

\noindent
Then the following diagram is commutative:
\begin{equation*}
  \begin{tikzcd}
  X \arrow{r}{r} \arrow{d}[swap]{\kappa}
  & U \arrow{d}{t}
  \\
  X \oplus P \arrow{r}{s} \arrow{ru}{\theta}
  & V
\end{tikzcd}
\end{equation*}

\noindent
Clearly, $\theta\kappa=r$. Also, $tq\iota^{n,x}=tq^{n,x}=u^{n,x}=s\nu\iota^{n,x}$ for every $n \geq 1$ and $x \in Y_{n}$, implying that $tq=s\nu$. Further, $t\theta\kappa=tr=s\kappa$ and $t\theta\nu=tq=s\nu$, showing that $t\theta=s$. As a consequence, $\kappa$ has the left lifting property against all fibrations.

Second Half of Lifting: Consider a commutative diagram
\begin{equation*}
  \begin{tikzcd}
  X \arrow{r}{h} \arrow{d}[swap]{f}
  & Z \arrow{d}{g}
  \\
  Y \arrow{r}{l}
  & W
\end{tikzcd}
\end{equation*}

\noindent
in $\mathcal{C}_{\geq0}(R)$ in which $f$ is a trivial cofibration and $g$ is a fibration. Consider the factorization
\begin{equation*}
  \begin{tikzcd}
  X \arrow{r}{f} \arrow{d}[swap]{\kappa}
  & Y
  \\
  X \oplus P \arrow{ru}[swap]{\eta}
  &
\end{tikzcd}
\end{equation*}

\noindent
of $f$ that was established above in which $P=\bigoplus_{n=1}^{\infty}\bigoplus_{x \in Y_{n}}D(n)$, $\eta$ is a fibration, and $\kappa$ is a trivial cofibration that further has the left lifting property against all fibrations. Since $f$ and $\kappa$ are quasi-isomorphisms, we deduce that $\eta$ is also a quasi-isomorphism, so $\eta$ is a trivial fibration. Consider the following commutative diagram:
\begin{equation*}
  \begin{tikzcd}
  X \arrow{r}{\kappa} \arrow{d}[swap]{f}
  & X \oplus P \arrow{d}{\eta}
  \\
  Y \arrow{r}{1^{Y}}
  & Y
\end{tikzcd}
\end{equation*}

\noindent
As $f$ is a cofibration and $\eta$ is a trivial fibration, the first half of lifting that was established above provides a morphism $r:Y \rightarrow X \oplus P$ that makes the following diagram commutative:
\begin{equation*}
  \begin{tikzcd}
  X \arrow{r}{\kappa} \arrow{d}[swap]{f}
  & X \oplus P \arrow{d}{\eta}
  \\
  Y \arrow{r}{1^{Y}} \arrow{ru}{r}
  & Y
\end{tikzcd}
\end{equation*}

\noindent
This yields the following commutative diagram:
\begin{equation*}
  \begin{tikzcd}
  X \arrow{r}{1^{X}} \arrow{d}{f}
  & X \arrow{r}{1^{X}} \arrow{d}{\kappa}
  & X \arrow{d}{f}
  \\
  Y \arrow{r}{r}
  & X \oplus P \arrow{r}{\eta}
  & Y
\end{tikzcd}
\end{equation*}

\noindent
This shows that $f$ is a retract of $\kappa$. But $\kappa$ has the left lifting property against all fibrations, so the proof of \cite[Lemma 7.2.8]{Hi} that makes no use of the model category assumption, shows that $f$ has the left lifting property against all fibrations. This implies that the above diagram can be completed as follows:
\begin{equation*}
  \begin{tikzcd}
  X \arrow{r}{h} \arrow{d}[swap]{f}
  & Z \arrow{d}{g}
  \\
  Y \arrow{r}{l} \arrow{ru}
  & W
\end{tikzcd}
\end{equation*}

Second Half of Factorization: Let $f:X \rightarrow Y$ be a morphism in $\mathcal{C}_{\geq 0}(R)$. We construct the desired factorization inductively. Let $\rho_{0}:P_{0} \rightarrow Y_{0}$ be an $R$-epimorphism in which $P_{0}$ is a projective left $R$-module, and set $Q_{0}=X_{0} \oplus P_{0}$. By the universal property of direct sum, there is a unique morphism $\eta_{0}:Q_{0} \rightarrow Y_{0}$ of $R$-complexes that makes the diagram
\begin{equation*}
  \begin{tikzcd}
  X_{0} \arrow{r}{\kappa_{0}} \arrow{dr}[swap]{f_{0}}
  & Q_{0} \arrow{d}{\eta_{0}}
  & P_{0} \arrow{l}[swap]{\nu_{0}} \arrow{dl}{\rho_{0}}
  \\
  & Y_{0}
  &
\end{tikzcd}
\end{equation*}

\noindent
commutative in which $\kappa_{0}$ and $\nu_{0}$ are canonical injections. Then $f_{0}=\eta_{0}\kappa_{0}$ and $\rho_{0}=\eta_{0}\nu_{0}$. It is clear that $\kappa_{0}$ is injective whose $\Coker(\kappa_{0})\cong P_{0}$ is projective. On the other hand, since $\rho_{0}$ is surjective, we infer that $\eta_{0}$ is surjective. Now suppose that $n\geq 1$, and for any $0 \leq i \leq n-1$, the $R$-module $Q_{i}$ and the $R$-homomorphisms $\partial_{i}^{Q}:Q_{i} \rightarrow Q_{i-1}$, $\kappa_{i}:X_{i} \rightarrow Q_{i}$, and $\eta_{i}:Q_{i} \rightarrow Y_{i}$ are constructed in a way that $\partial_{i-1}^{Q}\partial_{i}^{Q}=0$, $f_{i}=\eta_{i}\kappa_{i}$, $\kappa_{i}$ is injective whose $\Coker(\kappa_{i})$ is projective, $\eta_{i}$ is surjective with $H_{i-1}(\eta)$ an isomorphism, and the following diagram is commutative:
\begin{equation*}
  \begin{tikzcd}
  X_{i} \arrow{r}{\partial_{i}^{X}} \arrow{d}[swap]{\kappa_{i}}
  & X_{i-1} \arrow{d}{\kappa_{i-1}}
  \\
  Q_{i} \arrow{r}{\partial_{i}^{Q}} \arrow{d}[swap]{\eta_{i}}
  & Q_{i-1} \arrow{d}{\eta_{i-1}}
  \\
  Y_{i} \arrow{r}{\partial_{i}^{Y}}
  & Y_{i-1}
\end{tikzcd}
\end{equation*}

\noindent
Using the above diagram, we see that if $x \in X_{n}$, then $\partial_{n-1}^{Q}\left(\kappa_{n-1}\left(\partial_{n}^{X}(x)\right)\right)=\kappa_{n-2}\left(\partial_{n-1}^{X}\left(\partial_{n}^{X}(x)\right)\right)=0$, so $\kappa_{n-1}\left(\partial_{n}^{X}(x)\right) \in Z_{n-1}(Q)$. Therefore, we get an induced $R$-homomorphism $\overline{\kappa_{n-1}\partial_{n}^{X}}:X_{n} \rightarrow Z_{n-1}(Q)$. Consider the following commutative diagram:
\begin{equation*}
  \begin{tikzcd}[column sep=4em,row sep=2em]
  X_{n} \arrow{r}{f_{n}} \arrow{d}[swap]{\overline{\kappa_{n-1}\partial_{n}^{X}}}
  & Y_{n} \arrow{d}{\overline{\partial_{n}^{Y}}}
  \\
  Z_{n-1}(Q) \arrow[scale=3]{r}{Z_{n-1}(\eta)}
  & Z_{n-1}(Y)
\end{tikzcd}
\end{equation*}

\noindent
Indeed, using the fact that $f$ is a morphism of $R$-complexes, we have $\eta_{n-1}\left(\kappa_{n-1}\left(\partial_{n}^{X}(x)\right)\right)= f_{n-1}\left(\partial_{n}^{X}(x)\right)=\partial_{n}^{Y}\left(f_{n}(x)\right)$ for every $x\in X_{n}$. Set:
$$L_{n}=Z_{n-1}(Q)\textstyle\bigsqcap_{Z_{n-1}(Y)}Y_{n}= \left\{(x,y)\in Z_{n-1}(Q)\oplus Y_{n}\suchthat \eta_{n-1}(x)=\partial_{n}^{Y}(y)\right\}$$
Then there is a unique morphism $\varrho_{n}:X_{n} \rightarrow L_{n}$ that makes the following pullback diagram commutative:
\begin{equation*}
  \begin{tikzcd}[column sep=4em,row sep=2em]
  X_{n} \arrow{dr}{\varrho_{n}} \arrow[ddr, bend right, swap, "\overline{\kappa_{n-1}\partial_{n}^{X}}"] \arrow[drr, bend left, "f_{n}"] & &
  \\
  & L_{n} \arrow{r}{\zeta_{n}} \arrow{d}[swap]{\xi_{n}} & Y_{n} \arrow{d}{\overline{\partial_{n}^{Y}}}
  \\
  & Z_{n-1}(Q) \arrow{r}{Z_{n-1}(\eta)} & Z_{n-1}(Y)
\end{tikzcd}
\end{equation*}

\noindent
As in the case $i=0$, let $\rho_{n}:P_{n} \rightarrow L_{n}$ be an $R$-epimorphism in which $P_{n}$ is a projective left $R$-module, and set $Q_{n}=X_{n} \oplus P_{n}$. By the universal property of direct sum, there is a unique morphism $\theta_{n}:Q_{n} \rightarrow L_{n}$ of $R$-complexes that makes the diagram
\begin{equation*}
  \begin{tikzcd}
  X_{n} \arrow{r}{\kappa_{n}} \arrow{dr}[swap]{\varrho_{n}}
  & Q_{n} \arrow{d}{\theta_{n}}
  & P_{n} \arrow{l}[swap]{\nu_{n}} \arrow{dl}{\rho_{n}}
  \\
  & L_{n}
  &
\end{tikzcd}
\end{equation*}

\noindent
commutative in which $\kappa_{n}$ and $\nu_{n}$ are canonical injections. Set $\eta_{n}=\zeta_{n}\theta_{n}:Q_{n} \rightarrow Y_{n}$. Then $f_{n}=\zeta_{n}\varrho_{n}=\zeta_{n}\theta_{n}\kappa_{n}=\eta_{n}\kappa_{n}$. It is also clear that $\kappa_{n}$ is injective whose $\Coker(\kappa_{n})\cong P_{n}$ is projective.

Consider the following compositions of $R$-homomorphisms:
\[
 X_{n} \xrightarrow{\partial_{n}^{X}} X_{n-1} \xrightarrow{\kappa_{n-1}} Q_{n-1}
\quad \text{and} \quad
 P_{n} \xrightarrow{\rho_{n}} L_{n} \xrightarrow{\xi_{n}} Z_{n-1}(Q) \xrightarrow{\iota_{n-1}^{Q}} Q_{n-1}
\]

\noindent
By the universal property of direct sum, there is a unique morphism $\partial_{n}^{Q}:Q_{n} \rightarrow Q_{n-1}$ of $R$-complexes that makes the following diagram commutative:
\begin{equation*}
  \begin{tikzcd}
  X_{n} \arrow{r}{\kappa_{n}} \arrow{dr}[swap]{\kappa_{n-1}\partial_{n}^{X}}
  & Q_{n} \arrow{d}{\partial_{n}^{Q}}
  & P_{n} \arrow{l}[swap]{\nu_{n}} \arrow{dl}{\iota_{n-1}^{Q}\xi_{n}\rho_{n}}
  \\
  & Q_{n-1}
  &
\end{tikzcd}
\end{equation*}

\noindent
Then we have $\partial_{n-1}^{Q}\partial_{n}^{Q}\kappa_{n}=\partial_{n-1}^{Q}\kappa_{n-1}\partial_{n}^{X}=\kappa_{n-2}\partial_{n-1}^{X}\partial_{n}^{X}=0$ and $\partial_{n-1}^{Q}\partial_{n}^{Q}\nu_{n}=\partial_{n-1}^{Q}\iota_{n-1}^{Q}\xi_{n}\rho_{n}=0$ as $\im(\xi_{n}) \subseteq Z_{n-1}(Q)$. It follows that $\partial_{n-1}^{Q}\partial_{n}^{Q}=0$.

We next show that the following diagram is commutative:
\begin{equation*}
  \begin{tikzcd}
  X_{n} \arrow{r}{\partial_{n}^{X}} \arrow{d}[swap]{\kappa_{n}}
  & X_{n-1} \arrow{d}{\kappa_{n-1}}
  \\
  Q_{n} \arrow{r}{\partial_{n}^{Q}} \arrow{d}[swap]{\eta_{n}}
  & Q_{n-1} \arrow{d}{\eta_{n-1}}
  \\
  Y_{n} \arrow{r}{\partial_{n}^{Y}}
  & Y_{n-1}
\end{tikzcd}
\end{equation*}

\noindent
Indeed, we have $\partial_{n}^{Q}\kappa_{n}=\kappa_{n-1}\partial_{n}^{X}$ from the previous commutative diagram. Moreover, we have:
$$\partial_{n}^{Y}\eta_{n}\kappa_{n}=\partial_{n}^{Y}\zeta_{n}\theta_{n}\kappa_{n}=\partial_{n}^{Y}\zeta_{n}\varrho_{n}=\partial_{n}^{Y}f_{n}=f_{n-1}
\partial_{n}^{X}=\eta_{n-1}\kappa_{n-1}\partial_{n}^{X}=\eta_{n-1}\partial_{n}^{Q}\kappa_{n}$$
and
$$\partial_{n}^{Y}\eta_{n}\nu_{n}=\partial_{n}^{Y}\zeta_{n}\theta_{n}\nu_{n}=\partial_{n}^{Y}\zeta_{n}\rho_{n}=\iota_{n-1}^{Y}\overline{\partial_{n}^{Y}}
\zeta_{n}\rho_{n}=\iota_{n-1}^{Y}Z_{n-1}(\eta)\xi_{n}\rho_{n}=\eta_{n-1}\iota_{n-1}^{Q}\xi_{n}\rho_{n}=\eta_{n-1}\partial_{n}^{Q}\nu_{n}$$
It follows that $\partial_{n}^{Y}\eta_{n}=\eta_{n-1}\partial_{n}^{Q}$.

We now show that $H_{n-1}(\eta)$ is injective. Suppose that $H_{n-1}(\eta)\left(x+B_{n-1}(Q)\right)=\eta_{n-1}(x)+B_{n-1}(Y)=0$ for some $x\in Z_{n-1}(Q)$. Then $\eta_{n-1}(x)\in B_{n-1}(Y)$, so $\eta_{n-1}(x)=\partial_{n}^{Y}(y)$ for some $y\in Y_{n}$. This implies that $(x,y)\in L_{n}$. On the other hand, $\rho_{n}=\theta_{n}\nu_{n}$ and $\rho_{n}$ is surjective, so $\theta_{n}:Q_{n} \rightarrow L_{n}$ is surjective. Therefore, there is an element $(z,w) \in Q_{n}=X_{n} \oplus P_{n}$ such that $\theta_{n}(z,w)=(x,y)$. Then we have:
\begin{equation*}
\begin{split}
 x & = \xi_{n}(x,y)=\xi_{n}\left(\theta_{n}(z,w)\right)=\xi_{n}\left(\varrho_{n}(z)\right)+\xi_{n}\left(\rho_{n}(w)\right)= \kappa_{n-1}\left(\partial_{n}^{X}(z)\right) + \partial_{n}^{Q}\left(\nu_{n}(w)\right) \\
 & = \partial_{n}^{Q}\left(\kappa_{n}(z)\right)+\partial_{n}^{Q}\left(\nu_{n}(w)\right)= \partial_{n}^{Q}\left(\kappa_{n}(z)+\nu_{n}(w)\right)=\partial_{n}^{Q}(z,w)
\end{split}
\end{equation*}

\noindent
Hence $x\in B_{n-1}(Q)$, so $x+B_{n-1}(Q)=0$. This shows that $H_{n-1}(\eta)$ is injective.

We next show that $H_{n-1}(\eta)$ is surjective. Let $y\in Z_{n}(Y)$. Then $\eta_{n-1}(0)=0=\partial_{n}^{Y}(y)$, so $(0,y) \in L_{n}$. By the surjectivity of $\theta_{n}:Q_{n} \rightarrow L_{n}$, there is an element $(z,w)\in Q_{n}=X_{n} \oplus P_{n}$ such that $\theta_{n}(z,w)=(0,y)$. Then $y=\zeta_{n}(0,y)=\zeta_{n}\left(\theta_{n}(z,w)\right)=\eta_{n}(z,w)$. Also, a similar calculation as in the previous paragraph shows that $\partial_{n}^{Q}(z,w)=0$, so $(z,w)\in Z_{n}(Q)$. This shows that $Z_{n}(\eta):Z_{n}(Q) \rightarrow Z_{n}(Y)$ is surjective. Since we have constructed all these $R$-homomorphisms in the same way in each step, we can infer that $Z_{n-1}(\eta)$ is surjective. It follows that $H_{n-1}(\eta)$ is surjective. Therefore, $H_{n-1}(\eta)$ is an isomorphism.

We finally show that $\eta_{n}$ is surjective. Let $y \in Y_{n}$. Then $\partial_{n}^{Y}(y)\in B_{n-1}(Y) \subseteq Z_{n-1}(Y)$, so by the surjectivity of $Z_{n-1}(\eta)$, there is an element $x \in Z_{n-1}(Q)$ such that $\eta_{n-1}(x)=\partial_{n}^{Y}(y)$. Then $(x,y) \in L_{n}$, so by the surjectivity of $\theta_{n}$, there is an element $(z,w) \in Q_{n}=X_{n} \oplus P_{n}$ such that $\theta_{n}(z,w)=(x,y)$. It follows that $y=\zeta_{n}(x,y)=\zeta_{n}\left(\theta_{n}(z,w)\right)=\eta_{n}(z,w)$. Thus $\eta_{n}$ is surjective.

As a consequence, we can inductively construct the morphisms $\kappa:X \rightarrow Q$ and $\eta:Q \rightarrow Y$ of $R$-complexes such that $f=\eta\kappa$, $\eta$ is a trivial fibration, and $\kappa$ is a cofibration.

All in all, we have a cofibrantly generated model structure on $\mathcal{C}_{\geq 0}(R)$.
\end{prf}

\section{Simplicial Objects and Dold-Kan Correspondence Theorem}

Simplicial objects are nowadays widely used in modern mathematics. Simplicial sets in particular are combinatorial models for topological spaces. They also serve as the bedrock of the language of $\infty$-categories. Our central focus of attention is on simplicial modules and simplicial commutative algebras which can be thought of as generalizations as well as non-abelian analogues of chain complexes. In this section, we recall the basic definitions and examples of simplicial objects. Moreover, we fix our notations on the normalization and Dld-Kan functors. Deploying the Dold-Kan Correspondence, we transfer the model structure on connective chain complexes to simplicial modules.

\begin{definition} \label{4.1}
Let $\mathcal{C}$ be a category. A \textit{simplicial object} over $\mathcal{C}$ is a family $A=\{A_{n}\}_{n\geq 0}$ of objects of $\mathcal{C}$ together with \textit{face morphisms} $d_{n,i}^{A}:A_{n} \rightarrow A_{n-1}$ for every $n\geq 1$ and $0\leq i\leq n$, and \textit{degeneracy morphisms} $s_{n,i}^{A}:A_{n} \rightarrow A_{n+1}$ for every $n\geq 0$ and $0\leq i\leq n$, that satisfy the following relations for every $n\geq 0$ and $0\leq i,j \leq n$:
\begin{enumerate}
\item[(i)] $d_{n,i}^{A}d_{n+1,j}^{A}=d_{n,j-1}^{A}d_{n+1,i}^{A}$ if $i<j$.
\item[(ii)] $s_{n,i}^{A}s_{n-1,j}^{A}=s_{n,j+1}^{A}s_{n-1,i}^{A}$ if $i\leq j$.
\item[(iii)]
 \label{eqn:damage piecewise}
 $d_{n,i}^{A}s_{n-1,j}^{A}=
      \begin{cases}
        s_{n-2,j-1}^{A}d_{n-1,i}^{A} & \text{if } i<j \\
        1^{A_{n-1}} & \text{if } i=j \text{ or } j+1 \\
        s_{n-2,j}^{A}d_{n-1,i-1}^{A} & \text{if } i>j+1
      \end{cases}$
\end{enumerate}

\noindent
Given two simplicial objects $A=\{A_{n}\}_{n\geq 0}$ and $B=\{B_{n}\}_{n\geq 0}$ over $\mathcal{C}$, a \textit{morphism} $\tau:A \rightarrow B$ \textit{of simplicial objects} is a collection $\tau=(\tau_{n})_{n\geq 0}$ of morphisms $\tau_{n}:A_{n} \rightarrow B_{n}$ of $\mathcal{C}$ for every $n \geq 0$, such that the following diagram is commutative for every $n\geq 1$ and $0\leq i\leq n$:
\begin{equation*}
  \begin{tikzcd}[column sep=3.5em,row sep=2em]
  A_{n-1} \arrow{r}{s_{n-1,i}^{A}} \arrow{d}{\tau_{n-1}}
  & A_{n} \arrow{r}{d_{n,i}^{A}} \arrow{d}{\tau_{n}}
  & A_{n-1} \arrow{d}{\tau_{n-1}}
  \\
  B_{n-1} \arrow{r}{s_{n-1,i}^{B}}
  & B_{n} \arrow{r}{d_{n,i}^{B}}
  & B_{n-1}
\end{tikzcd}
\end{equation*}
\end{definition}

We speak of simplicial sets, simplicial modules, or simplicial commutative algebras when we consider simplicial objects over the categories of sets, modules, or commutative algebras.

\begin{remark} \label{4.2}
It is immediate from Definition \ref{4.1} that we have a category $\mathpzc{s}\mathcal{C}$ of simplicial objects over $\mathcal{C}$. As a matter of fact, $\mathpzc{s}\mathcal{C}$ is a functor category. Let $\Delta$ denote the \textit{simplex category} described by
$$\Obj(\Delta)=\left\{[n]=\{0,1,2,...,n\} \suchthat n\geq 0\right\}$$
and
$$\Mor_{\Delta}\left([n],[m]\right) = \left\{f\in \Mor_{\mathcal{S}\mathpzc{et}}\left([n],[m]\right) \suchthat f \textrm{ is order-preserving}\right\}$$
for every $[n],[m] \in \Obj(\Delta)$. Then it is folklore that $\mathpzc{s}\mathcal{C}=\mathcal{C}^{\Delta^{\op}}$. In other words, a simplicial object over $\mathcal{C}$ is nothing but a contravariant functor $\Delta \rightarrow \mathcal{C}$. Accordingly, we sometimes use the functorial point of view when defining a simplicial object.

For any $n\geq 0$ and $0\leq i\leq n$, the \textit{face map} $\delta_{n,i}:[n-1]\rightarrow [n]$ is the unique injective map that "misses" $i$, i.e.
\begin{equation*}
 \label{eqn:damage piecewise}
\delta_{n,i}(k)=
 \begin{dcases}
  k & \textrm{if } k<i \\
  k+1 & \textrm{if } k\geq i
 \end{dcases}
\end{equation*}

\noindent
for every $0\leq k\leq n-1$, and the \textit{degeneracy map} $\sigma_{n,i}:[n+1]\rightarrow [n]$ is the unique surjective map that "hits" $i$ twice, i.e.
\begin{equation*}
 \label{eqn:damage piecewise}
\sigma_{n,i}(k)=
 \begin{dcases}
  k & \textrm{if } k\leq i \\
  k-1 & \textrm{if } k> i
 \end{dcases}
\end{equation*}

\noindent
for every $0\leq k\leq n+1$. It is clear that the face and the degeneracy maps are morphisms in the simplex category. Moreover, $\delta_{n,n}$ is the inclusion map for every $n\geq 0$. One can show by brute force that the following face-degeneracy relations hold for every $0\leq i,j \leq n$:
\begin{enumerate}
\item[(i)] $\delta_{n+1,j}\delta_{n,i}=\delta_{n+1,i}\delta_{n,j-1}$ if $i<j$.
\item[(ii)] $\sigma_{n-1,j}\sigma_{n,i}=\sigma_{n-1,i}\sigma_{n,j+1}$ if $i\leq j$.
\item[(iii)]
 \label{eqn:damage piecewise}
 $\sigma_{n-1,j}\delta_{n,i}=
      \begin{cases}
        \delta_{n-1,i}\sigma_{n-2,j-1} & \text{if } i<j \\
        1^{[n-1]} & \text{if } i=j \text{ or } j+1 \\
        \delta_{n-1,i-1}\sigma_{n-2,j} & \text{if } i>j+1.
      \end{cases}$
\end{enumerate}

Suppose that $n,m\geq 0$ and $f\in \Mor_{\Delta}\left([n],[m]\right)$. If $f$ is bijective, then $n=m$ and $f=1^{[n]}$. Now assume that $f$ is not bijective. If $f$ is injective, then there is a unique factorization $f=\delta_{m,i_{1}}\cdots \delta_{n+1,i_{m-n}}$ where $0\leq i_{m-n}<\cdots < i_{1}\leq m$ are the elements of $[m]$ that are not in the image of $f$. If $f$ is surjective, then there is a unique factorization $f=\sigma_{m,j_{1}}\cdots \sigma_{n-1,j_{n-m}}$ where $0\leq j_{1}<\cdots < j_{n-m}< n$ are the elements of $[n]$ for which $f(j)=f(j+1)$. In general, there is a unique factorization $f=gh$ where $g\in \Mor_{\Delta}\left([k],[m]\right)$ is injective and $h\in \Mor_{\Delta}\left([n],[k]\right)$ is surjective; See \cite[Lemma 8.1.2]{We}.
\end{remark}

A trivial example of a simplicial object is the constant one in which we have the same object in each degree and the face and the degeneracy morphisms are all identity morphisms. We next look at some non-trivial examples.

\begin{example} \label{4.3}
Fix $n\geq 0$. We have the contravariant functor $\Delta^{n}= \Mor_{\Delta}\left(-,[n]\right):\Delta \rightarrow \mathcal{S}\mathpzc{et}$, so $\Delta^{n}$ is a simplicial set called the \textit{combinatorial} $n$-\textit{simplex}. If $m\geq 0$ and $f \in \Mor_{\Delta}\left([n],[m]\right)$, then $\Delta^{f}= \Mor_{\Delta}(-,f):\Delta^{n}= \Mor_{\Delta}\left(-,[n]\right) \rightarrow \Mor_{\Delta}\left(-,[m]\right)=\Delta^{m}$ is a natural transformation of functors, i.e. a morphism of simplicial sets. For any $m\geq 0$, considering the images of order-preserving maps $[m]\rightarrow [n]$, we can make the following identification:
$$\Delta_{m}^{n}= \Mor_{\Delta}\left([m],[n]\right) \cong \left\{(\alpha_{0},\alpha_{1},...,\alpha_{m})\in \mathbb{N}^{m+1} \suchthat 0\leq \alpha_{0} \leq \alpha_{1} \leq \cdots \leq \alpha_{m} \leq n \right\}$$
The above identification shows that $|\Delta_{m}^{n}|= \binom{n+m+1}{n}$. On the other hand, under this identification, for any $m\geq 0$ and $0\leq i\leq m$, the face morphism $d_{m,i}^{\Delta^{n}}:\Delta_{m}^{n}\rightarrow \Delta_{m-1}^{n}$ is given by $d_{m,i}^{\Delta^{n}}(\alpha_{0},\alpha_{1},...,\alpha_{m})=(\alpha_{0},...,\alpha_{i-1},\alpha_{i+1},...,\alpha_{m})$, and the degeneracy morphism $s_{m,i}^{\Delta^{n}}:\Delta_{m}^{n}\rightarrow \Delta_{m+1}^{n}$ is given by $s_{m,i}^{\Delta^{n}}(\alpha_{0},\alpha_{1},...,\alpha_{m})=(\alpha_{0},...,\alpha_{i-1},\alpha_{i},\alpha_{i},\alpha_{i+1},...,\alpha_{m})$. Given $m\geq 0$, any element of $\Delta_{m}^{n}$ is called an $m$-\textit{cell}. The $m$-cells that have repeated coordinates, i.e. the ones that belong to the images of degeneracy morphisms, are called \textit{degenerate}, and the rest are called \textit{non-degenerate}.

We further have the contravariant functor $\partial\Delta^{n}:\Delta\rightarrow \mathcal{S}\mathpzc{et}$, given by
$$\partial\Delta^{n}_{m} = \left\{f\in \Delta_{m}^{n}=\Mor_{\Delta}\left([m],[n]\right) \suchthat f \textrm{ is not surjective}\right\}$$
for every $m\geq 0$, and $(\partial\Delta^{n})(\eta):= \Delta^{n}(\eta)= \Mor_{\Delta}\left(\eta,[n]\right)$ for every $\eta\in \Mor(\Delta)$. Thus $\partial\Delta^{n}$ is a simplicial set called the \textit{boundary} of $\Delta^{n}$. If $\iota_{m}:\partial\Delta_{m}^{n}\rightarrow \Delta_{m}^{n}$ is the inclusion map for every $m\geq 0$, then $\iota=(\iota_{m})_{m\geq 0}:\partial\Delta^{n}\rightarrow \Delta^{n}$ is a morphism of simplicial sets. Under the above identification, we have for every $m\geq 0$:
$$\partial\Delta_{m}^{n} \cong \left\{(\alpha_{0},\alpha_{1},...,\alpha_{m})\in \Delta_{m}^{n} \suchthat \{\alpha_{0},\alpha_{1},...,\alpha_{m}\} \subset [n] \right\}$$
For any $m\geq 0$, we notice that $\partial\Delta_{m}^{1}=\{(\underbrace{0,...,0}_{m+1 \textrm{ times}}),(\underbrace{1,...,1}_{m+1 \textrm{ times}})\}$, so in light of $\Delta_{m}^{0}=\{(\underbrace{0,...,0}_{m+1 \textrm{ times}})\}$, we see that there is an isomorphism $\partial\Delta^{1} \cong \Delta^{0}\sqcup \Delta^{0}$ of simplicial sets.
\end{example}

We next recall some basic constructions of simplicial objects which will be needed later.

\begin{remark} \label{4.4}
If $R$ is a ring, $M$ is a simplicial right $R$-module, and $N$ is a simplicial left $R$-module, then the \textit{simplicial tensor product} $M\otimes_{R}N$ is defined as follows. If $n\geq 0$, then $(M\otimes_{R}N)_{n}=M_{n}\otimes_{R}N_{n}$. If $\eta \in \Mor_{\Delta}\left([n],[m]\right)$, then $(M\otimes_{R}N)(\eta)=M(\eta)\otimes_{R}N(\eta):M_{m}\otimes_{R}N_{m}\rightarrow M_{n}\otimes_{R}N_{n}$. This construction is functorial in each argument.

If $\mathcal{C}$ is a locally small bicomplete category, $A$ is a simplicial object over $\mathcal{C}$, and $U$ is a simplicial set, then the \textit{copower simplicial object} $A^{(U)}$ is defined as follows. If $n\geq 0$, then $(A^{(U)})_{n}=A_{n}^{(U_{n})}$. If $\eta \in \Mor_{\Delta}\left([n],[m]\right)$, then we consider the morphism $A(\eta):A_{m}\rightarrow A_{n}$ in $\mathcal{C}$ as well as the map $U(\eta):U_{m}\rightarrow U_{n}$, and then define $A^{(U)}(\eta):A_{m}^{(U_{m})} \rightarrow A_{n}^{(U_{n})}$ to be the diagonal of the following commutative diagram:
\begin{equation*}
\begin{tikzcd}[column sep=4.5em,row sep=2.5em]
  A_{m}^{(U_{m})} \arrow{r}{A(\eta)^{(U_{m})}} \arrow{d}[swap]{\overline{U(\eta)}} \arrow{rd}
  & A_{n}^{(U_{m})} \arrow{d}{\overline{U(\eta)}}
  \\
  A_{m}^{(U_{n})} \arrow{r}{A(\eta)^{(U_{n})}}
  & A_{n}^{(U_{n})}
\end{tikzcd}
\end{equation*}

\noindent
This construction is functorial in each argument. Moreover, if $R$ is a ring, $M$ is a simplicial left $R$-module, and $U$ is a simplicial set, then there is a natural isomorphism $M^{(U)}\cong R^{(U)}\otimes_{R}M$ of simplicial left $R$-modules.

If $\mathcal{C}$ is a locally small bicomplete category, and $A$ and $B$ are two simplicial objects over $\mathcal{C}$, then the \textit{mapping simplicial set} $\Map_{\mathpzc{s}\mathcal{C}}(A,B)$ is defined as follows. If $n\geq 0$, then $\Map_{\mathpzc{s}\mathcal{C}}(A,B)_{n}=\Mor_{\mathpzc{s}\mathcal{C}}\left(A^{(\Delta^{n})},B\right)$. If $\eta \in \Mor_{\Delta}\left([n],[m]\right)$ and $f \in \Mor_{\mathpzc{s}\mathcal{C}}\left(A^{(\Delta^{m})},B\right)$, then we consider the morphisms
$$A^{(\Delta^{n})} \xrightarrow{A^{(\Delta^{\eta})}} A^{(\Delta^{m})} \xrightarrow{f} B$$
and define the map $\Map_{\mathpzc{s}\mathcal{C}}(A,B)(\eta):\Mor_{\mathpzc{s}\mathcal{C}}\left(A^{(\Delta^{m})},B\right)\rightarrow \Mor_{\mathpzc{s}\mathcal{C}}\left(A^{(\Delta^{n})},B\right)$ by setting $\Map_{\mathpzc{s}\mathcal{C}}(A,B)(\eta)(f)=f A^{(\Delta^{\eta})}$. This construction is functorial in each argument. Moreover, if $R$ is a ring, $M$ and $N$ are two simplicial left $R$-modules, and $U$ is a simplicial set, then there is a natural bijection of sets as follows:
$$\Mor_{\mathpzc{s}\mathcal{M}\mathpzc{od}(R)}\left(M^{(U)},N\right)\cong \Mor_{\mathpzc{s}\mathcal{M}\mathpzc{od}(R)}\left(M,\Map_{\mathpzc{s}\mathcal{S}\mathpzc{et}}(U,N)\right)$$

Let $U$ be a simplicial set. If $R$ is a ring and $M$ is a simplicial left $R$-module, then we can forget the module structure of $M$, consider it as a simplicial set, and then form $\Map_{\mathpzc{s}\mathcal{S}\mathpzc{et}}(U,M)$. Then $\Map_{\mathpzc{s}\mathcal{S}\mathpzc{et}}(U,M)$ is a simplicial set which becomes a simplicial left $R$-module with the module structure inherited from $M$ by pointwise operations. Similarly, if $R$ is a commutative ring and $A$ is a simplicial commutative $R$-algebra, then we can forget the algebra structure of $A$, consider it as a simplicial set, and then form $\Map_{\mathpzc{s}\mathcal{S}\mathpzc{et}}(U,A)$. Then $\Map_{\mathpzc{s}\mathcal{S}\mathpzc{et}}(U,A)$ is a simplicial set which becomes a simplicial commutative $R$-algebra with the algebra structure inherited from $A$ by pointwise operations.
\end{remark}

We now recall the normalization functor; see \cite[Definition 8.3.6]{We}.

\begin{definition} \label{4.5}
Let $R$ be a ring. The \textit{normalization functor} $\Nor:\mathpzc{s}\mathcal{M}\mathpzc{od}(R) \rightarrow \mathcal{C}_{\geq 0}(R)$ is defined as follows:
\begin{enumerate}
\item[(i)] $\Nor$ assigns to each simplicial left $R$-module $M$, the normalized $R$-complex $\Nor(M)$, whose terms are $\Nor(M)_{0}=M_{0}$ and $\Nor(M)_{n}=\bigcap_{i=0}^{n-1}\Ker\left(d_{n,i}^{M}\right)\subseteq M_{n}$ for every $n\geq 1$, and whose differential $\partial_{n}^{\Nor(M)}:\Nor(M)_{n} \rightarrow \Nor(M)_{n-1}$ is given by $\partial_{n}^{\Nor(M)}(x)=(-1)^{n}d_{n,n}^{M}(x)$ for every $n\geq 1$ and $x\in \Nor(M)_{n}$.
\item[(ii)] $\Nor$ assigns to each morphism $f:M \rightarrow N$ of simplicial left $R$-modules, the morphism $\Nor(f):\Nor(M)\rightarrow \Nor(N)$ of $R$-complexes, where $\Nor(f)_{n}:\Nor(M)_{n}\rightarrow \Nor(N)_{n}$ is given by $\Nor(f)_{n}(x)=f_{n}(x)$ for every $n\geq 0$ and $x\in \Nor(M)_{n}$.
\end{enumerate}
\end{definition}

\begin{remark} \label{4.6}
Let $R$ be a ring, and $M$ a simplicial left $R$-module. The \textit{Moore complex} $C(M)$ is defined as follows. If $n\geq 0$, then $C(M)_{n}=M_{n}$. If $n\geq 1$, then $\partial_{n}^{C(M)}=\sum_{i=0}^{n}(-1)^{i}d_{n,i}^{M}: C(M)_{n}\rightarrow C(M)_{n-1}$. It can be verified by inspection that $C(M)$ is a connective $R$-complex. Moreover, if we set $D(M)_{0}=0$ and $D(M)_{n}=\sum_{i=0}^{n-1}\im\left(s_{n-1,i}^{M}\right)$ for every $n\geq 1$, and also $\partial_{n}^{D(M)}(x)=\partial_{n}^{C(M)}(x)$ for every $n\geq 1$ and $x\in D(M)_{n}$, then $D(M)$ turns out to be a subcomplex of $C(M)$. Moreover, we have $C(M)=\Nor(M)\oplus D(M)$ and $\Nor(M)\cong C(M)/D(M)$ naturally in $M$. See \cite[Lemma 8.3.7]{We}.
\end{remark}

We recall a quintessential example of simplicial sets.

\begin{example} \label{4.7}
Let $\mathcal{C}$ be a small category. The \textit{nerve} of $\mathcal{C}$ is a simplicial set defined as follows. If $[n]\in \Obj(\Delta)$, then viewing the poset $[n]$ as a small category, we set $\nrv_{n}(\mathcal{C}) = \Fun\left([n],\mathcal{C}\right)$, i.e. the set of covariant functors $[n]\rightarrow \mathcal{C}$. If $\eta\in \Mor_{\Delta}\left([n],[m]\right)$, then the map $\nrv(\mathcal{C})(\eta):\Fun\left([m],\mathcal{C}\right)\rightarrow \Fun\left([n],\mathcal{C}\right)$ is defined by $\nrv(\mathcal{C})(\eta)(\theta) = \theta\eta$ for every $\theta\in \Fun\left([m],\mathcal{C}\right)$. Then $\nrv(\mathcal{C}):\Delta \rightarrow \mathcal{S}\mathpzc{et}$ is a contravariant functor, so it is a simplicial set. More specifically, we have $\nrv_{0}(\mathcal{C})= \Obj(\mathcal{C})$, and for any $n\geq 1$, $\nrv_{n}(\mathcal{C})$ is the set of $n$-tuples of composable morphisms
$$A_{0}\rightarrow A_{1}\rightarrow \cdots \rightarrow A_{n-1}\rightarrow A_{n}$$
in $\mathcal{C}$. Moreover, the face morphism $d_{n,i}^{\nrv(\mathcal{C})}: \nrv_{n}(\mathcal{C})\rightarrow \nrv_{n-1}(\mathcal{C})$ is given by
$$d_{n,i}^{\nrv(\mathcal{C})}\left(A_{0}\rightarrow \cdots \rightarrow A_{i-1}\xrightarrow{f_{i-1}} A_{i}\xrightarrow{f_{i}} A_{i+1}\rightarrow \cdots \rightarrow A_{n}\right) = \left(A_{0}\rightarrow \cdots \rightarrow A_{i-1}\xrightarrow{f_{i}f_{i-1}} A_{i+1}\rightarrow \cdots \rightarrow A_{n}\right)$$
and the degeneracy morphism $s_{n,i}^{\nrv(\mathcal{C})}: \nrv_{n}(\mathcal{C})\rightarrow \nrv_{n+1}(\mathcal{C})$ is given by
\small
$$s_{n,i}^{\nrv(\mathcal{C})}\left(A_{0}\rightarrow \cdots \rightarrow A_{i-1}\rightarrow A_{i}\rightarrow A_{i+1}\rightarrow \cdots \rightarrow A_{n}\right) = \left(A_{0}\rightarrow \cdots \rightarrow A_{i-1}\rightarrow A_{i}\xrightarrow{1^{A_{i}}} A_{i} \rightarrow A_{i+1}\rightarrow \cdots \rightarrow A_{n}\right)$$
\normalsize
for every $0\leq i\leq n$.
\end{example}

We need the next lemma in the sequel.

\begin{lemma} \label{4.8}
Let $R$ be a ring, and $\mathcal{P}$ a poset with a least element $e$. View $\mathcal{P}$ as a small category, and let $\zeta:\Delta^{0}\rightarrow \nrv(\mathcal{P})$ be the morphism of simplicial sets given by $\zeta_{m}(0,...,0)=(e\leq \cdots \leq e)$ for every $m\geq 0$. Then $\Nor\left(R^{(\zeta)}\right):\Nor\left(R^{(\Delta^{0})}\right) \rightarrow \Nor\left(R^{\left(\nrv(\mathcal{P})\right)}\right)$ is a homotopy equivalence of $R$-complexes. As a consequence, we have:
\begin{equation*}
 \label{eqn:damage piecewise}
 H_{i}\left(\Nor\left(R^{\left(\nrv(\mathcal{P})\right)}\right)\right)\cong
 \begin{dcases}
  R & \textrm{if } i=0 \\
  0 & \textrm{if } i>0
 \end{dcases}
\end{equation*}
\end{lemma}

\begin{proof}
Set $U=\nrv(\mathcal{P})$. Since $\Delta_{m}^{0}=\{(\underbrace{0,...,0}_{m+1 \textrm{ times}})\}$ for every $m\geq 0$, we can define a morphism $\zeta:\Delta^{0}\rightarrow U$ of simplicial sets by setting $\zeta_{m}(0,...,0)=(e\leq \cdots \leq e)$ for every $m\geq 0$. We show that the induced morphism
$$\textstyle \overline{R^{(\zeta)}}: \frac{C\left(R^{(\Delta^{0})}\right)}{D\left(R^{(\Delta^{0})}\right)} \rightarrow \frac{C\left(R^{(U)}\right)}{D\left(R^{(U)}\right)}$$
is a homotopy equivalence. Since $\Delta^{0}$ is a terminal object in $\mathpzc{s}\mathcal{S}\mathpzc{et}$, there exists a unique morphism $\varepsilon:U\rightarrow \Delta^{0}$ of simplicial sets. It is clear that $\varepsilon\zeta=1^{\Delta^{0}}$, so $R^{(\varepsilon)}R^{(\zeta)}=1^{R^{(\Delta^{0})}}$, whence
$$\overline{R^{(\varepsilon)}}\overline{R^{(\zeta)}}=1^{\frac{C\left(R^{(\Delta^{0})}\right)}{D\left(R^{(\Delta^{0})}\right)}}.$$

Now let $m\geq 0$. Then $R^{(U_{m})}$ is a free left $R$-module with a basis that is in 1-1 correspondence with the set $U_{m}$. We identify this basis with elements of $U_{m}$ in what follows. Define an $R$-homomorphism
$$t_{m}:C\left(R^{(U)}\right)_{m}=R^{(U_{m})}\rightarrow R^{(U_{m+1})}=C\left(R^{(U)}\right)_{m+1}$$
by setting
$$t_{m}(a_{0} \leq \cdots \leq a_{m}) = (e \leq a_{0} \leq \cdots \leq a_{m})$$
for every $(a_{0} \leq \cdots \leq a_{m})\in U_{m}$, and then extending linearly. We note that $(a_{0} \leq \cdots \leq a_{m})\in \im\left(s_{m-1,i}^{U}\right)$ for some $0\leq i\leq m-1$ if and only if $a_{i}=a_{i+1}$. This implies that if $(a_{0} \leq \cdots \leq a_{m})\in \im\left(s_{m-1,i}^{U}\right)$, then $(e \leq a_{0} \leq \cdots \leq a_{m})\in \im\left(s_{m,i+1}^{U}\right)$. It follows that $t_{m}\left(D\left(R^{(U)}\right)_{m}\right)\subseteq D\left(R^{(U)}\right)_{m+1}$, so $t_{m}$ induces an $R$-homomorphism
$$\textstyle \overline{t_{m}}: \frac{C\left(R^{(U)}\right)_{m}}{D\left(R^{(U)}\right)_{m}} \rightarrow \frac{C\left(R^{(U)}\right)_{m+1}}{D\left(R^{(U)}\right)_{m+1}}.$$
We also set $\overline{t_{m}}=0$ for every $m< 0$. We show that
$$\textstyle 1^{\frac{C\left(R^{(U)}\right)_{m}}{D\left(R^{(U)}\right)_{m}}} - \overline{R^{(\zeta_{m})}}\overline{R^{(\varepsilon_{m})}} = \partial_{m+1}^{\frac{C\left(R^{(U)}\right)}{D\left(R^{(U)}\right)}}\overline{t_{m}} + \overline{t_{m-1}}\partial_{m}^{\frac{C\left(R^{(U)}\right)}{D\left(R^{(U)}\right)}}$$
every $m\geq 0$. To this end, it suffices to establish the equality for the elements of $U_{m}$. Let $(a_{0} \leq \cdots \leq a_{m})\in U_{m}$. We consider two case: \\

Case I: $m=0$

In this case, we have:
$$(a_{0}) - \overline{R^{(\zeta_{0})}}\left(\overline{R^{(\varepsilon_{0})}}(a_{0})\right) = (a_{0})-\overline{R^{(\zeta_{0})}}(0)= (a_{0})-(e)$$
On the other hand, since $\overline{t_{-1}}=0$, we have:
\begin{equation*}
\begin{split}
 \partial_{1}^{\frac{C\left(R^{(U)}\right)}{D\left(R^{(U)}\right)}}\left(\overline{t_{0}}(a_{0})\right) + \overline{t_{-1}}\left(\partial_{0}^{\frac{C\left(R^{(U)}\right)}{D\left(R^{(U)}\right)}}(a_{0})\right)
 & = \partial_{1}^{\frac{C\left(R^{(U)}\right)}{D\left(R^{(U)}\right)}}\left((e \leq a_{0}) + D\left(R^{(U)}\right)_{1}\right) = \partial_{1}^{C\left(R^{(U)}\right)}(e \leq a_{0}) \\
 & = d_{1,0}^{R^{(U)}}(e \leq a_{0}) - d_{1,1}^{R^{(U)}}(e \leq a_{0}) = (a_{0})-(e)
\end{split}
\end{equation*}

\noindent
Consequently, we see that the desired equality holds in this case. \\

Case II: $m>0$

In this case, since $(e \leq \cdots \leq e)\in D\left(R^{(U)}\right)_{m}$, we have:
\footnotesize
\begin{equation*}
\begin{split}
 (a_{0} \leq \cdots \leq a_{m}) - \overline{R^{(\zeta_{m})}}\left(\overline{R^{(\varepsilon_{m})}}\left((a_{0} \leq \cdots \leq a_{m})+D\left(R^{(U)}\right)_{m}\right)\right)
 & = (a_{0} \leq \cdots \leq a_{m}) - \overline{R^{(\zeta_{m})}}\left((0,...,0)+D\left(R^{(\Delta^{0})}\right)_{m}\right) \\
 & = (a_{0} \leq \cdots \leq a_{m}) - (e \leq \cdots \leq e) + D\left(R^{(U)}\right)_{m} \\
 & = (a_{0} \leq \cdots \leq a_{m}) + D\left(R^{(U)}\right)_{m}
\end{split}
\end{equation*}
\normalsize

\noindent
On the other hand, we have:
\footnotesize
\begin{equation*}
\begin{split}
 \partial_{m+1}^{\frac{C\left(R^{(U)}\right)}{D\left(R^{(U)}\right)}}\left(\overline{t_{m}}\left((a_{0} \leq \cdots \leq a_{m}) + D\left(R^{(U)}\right)_{m}\right)\right)
 & = \partial_{m+1}^{\frac{C\left(R^{(U)}\right)}{D\left(R^{(U)}\right)}}\left((e \leq a_{0} \leq \cdots \leq a_{m}) + D\left(R^{(U)}\right)_{m+1}\right) \\
 & = \partial_{m+1}^{C\left(R^{(U)}\right)}(e \leq a_{0} \leq \cdots \leq a_{m}) + D\left(R^{(U)}\right)_{m} \\
 & = \sum_{i=0}^{m+1}(-1)^{i}d_{m+1,i}^{R^{(U)}}(e \leq a_{0} \leq \cdots \leq a_{m}) + D\left(R^{(U)}\right)_{m} \\
 & = d_{m+1,0}^{R^{(U)}}(e \leq a_{0} \leq \cdots \leq a_{m}) + \sum_{i=1}^{m+1}(-1)^{i}d_{m+1,i}^{R^{(U)}}(e \leq a_{0} \leq \cdots \leq a_{m}) + D\left(R^{(U)}\right)_{m} \\
 & = (a_{0} \leq \cdots \leq a_{m}) + \sum_{i=1}^{m+1}(-1)^{i}(e \leq a_{0} \leq \cdots \leq a_{i-2} \leq a_{i} \leq \cdots \leq a_{m}) + D\left(R^{(U)}\right)_{m}
\end{split}
\end{equation*}
\normalsize

\noindent
and
\footnotesize
\begin{equation*}
\begin{split}
 \overline{t_{m-1}}\left(\partial_{m}^{\frac{C\left(R^{(U)}\right)}{D\left(R^{(U)}\right)}}\left((a_{0} \leq \cdots \leq a_{m}) + D\left(R^{(U)}\right)_{m}\right)\right)
 & = \overline{t_{m-1}}\left(\partial_{m}^{C\left(R^{(U)}\right)}(a_{0} \leq \cdots \leq a_{m})+D\left(R^{(U)}\right)_{m-1}\right) \\
 & = \overline{t_{m-1}}\left(\sum_{i=0}^{m}(-1)^{i}d_{m,i}^{R^{(U)}}(a_{0} \leq \cdots \leq a_{m})+D\left(R^{(U)}\right)_{m-1}\right) \\
 & = \overline{t_{m-1}}\left(\sum_{i=0}^{m}(-1)^{i}(a_{0} \leq \cdots \leq a_{i-1} \leq a_{i+1} \leq \cdots \leq a_{m})+D\left(R^{(U)}\right)_{m-1}\right) \\
 & = \sum_{i=0}^{m}(-1)^{i}(e \leq a_{0} \leq \cdots \leq a_{i-1} \leq a_{i+1} \leq \cdots \leq a_{m})+D\left(R^{(U)}\right)_{m} \\
 & = - \sum_{i=1}^{m+1}(-1)^{i}(e \leq a_{0} \leq \cdots \leq a_{i-2} \leq a_{i} \leq \cdots \leq a_{m})+D\left(R^{(U)}\right)_{m}
\end{split}
\end{equation*}
\normalsize

\noindent
Adding them up, we get:
$$\partial_{m+1}^{\frac{C\left(R^{(U)}\right)}{D\left(R^{(U)}\right)}}\left(\overline{t_{m}}\left((a_{0} \leq \cdots \leq a_{m}) + D\left(R^{(U)}\right)_{m}\right)\right) + \overline{t_{m-1}}\left(\partial_{m}^{\frac{C\left(R^{(U)}\right)}{D\left(R^{(U)}\right)}}\left((a_{0} \leq \cdots \leq a_{m}) + D\left(R^{(U)}\right)_{m}\right)\right) = $$$$ (a_{0} \leq \cdots \leq a_{m}) + D\left(R^{(U)}\right)_{m}$$

\noindent
Consequently, we see that the desired equality holds in this case. This means that $\overline{R^{(\zeta)}}\overline{R^{(\varepsilon)}} \sim 1^{\frac{C\left(R^{(U)}\right)}{D\left(R^{(U)}\right)}}$. As a result, $\overline{R^{(\zeta)}}$ is a homotopy equivalence.

Now the commutative diagram
\begin{equation*}
  \begin{tikzcd}
  \Nor\left(R^{(\Delta^{0})}\right) \arrow{r}{\Nor\left(R^{(\zeta)}\right)} \arrow{d}[swap]{\cong}
  & [2em] \Nor\left(R^{(U)}\right) \arrow{d}{\cong}
  \\
  \frac{C\left(R^{(\Delta^{0})}\right)}{D\left(R^{(\Delta^{0})}\right)} \arrow{r}{\overline{R^{(\zeta)}}}
  & \frac{C\left(R^{(U)}\right)}{D\left(R^{(U)}\right)}
\end{tikzcd}
\end{equation*}

\noindent
implies that $\Nor\left(R^{(\zeta)}\right)$ is a homotopy equivalence, hence a quasi-isomorphism.

Next let $n\geq 0$. Recall that we have an identification
$$\Delta_{m}^{n}=\left\{(\alpha_{0},\alpha_{1},...,\alpha_{m})\in \mathbb{N}^{m+1} \suchthat 0\leq \alpha_{0} \leq\alpha_{1} \leq\cdots \leq \alpha_{m}\leq n \right\}$$
for every $m\geq 0$. The simplicial left $R$-module $R^{(\Delta^{n})}$ is in degree $m$ equal to the free left $R$-module $R^{(\Delta^{n}_{m})}$ with a basis $e_{1},...,e_{k}$ where $k=|\Delta_{m}^{n}|=\binom{n+m+1}{n}$. If an element of $\Delta^{n}_{m}$ is degenerate, then it belongs to the image of some degeneracy morphism $s_{m-1,i}^{\Delta^{n}}$ where $0\leq i \leq m-1$, so the corresponding basis element belongs to the image of the degeneracy morphism $R^{\left(s_{m-1,i}^{\Delta^{n}}\right)}=s_{m-1,i}^{R^{(\Delta^{n})}}$. In light of Remark \ref{4.6}, we have the isomorphism
$$\Nor\left(R^{(\Delta^{n})}\right)_{m}\cong \frac{R^{(\Delta_{m}^{n})}}{\sum_{i=0}^{m-1}\im\left(s_{m-1,i}^{R^{(\Delta^{n})}}\right)}$$
and we know that $\Nor\left(R^{(\Delta^{n})}\right)_{m}$ is a direct summand of $R^{(\Delta_{m}^{n})}$, so we conclude that $\Nor\left(R^{(\Delta^{n})}\right)_{m}$ is a free left $R$-module with a basis that corresponds to the non-degenerate elements of $\Delta^{n}_{m}$ for every $m\geq 0$. As a result, since $\Delta_{n}^{n}=\{(0,1,...,n)\}$, we get $\Nor\left(R^{(\Delta^{n})}\right)_{n}\cong R^{(\{(0,1,...,n)\})} = R$. Moreover, if $m>n$, then every element of $\Delta^{n}_{m}$ is degenerate, so $\Nor\left(R^{(\Delta^{n})}\right)_{m}=0$. In particular, we have $\Nor\left(R^{(\Delta^{0})}\right)_{0}=R^{(\Delta_{0}^{0})}=R^{(\{0\})}=R$ and $\Nor\left(R^{(\Delta^{0})}\right)_{m}=0$ for every $m\geq 1$. As a consequence, $\Nor\left(R^{(\Delta^{0})}\right)$ is the $R$-complex
$$0\rightarrow R \rightarrow 0$$
with $R$ in degree $0$. Therefore, we have:
\begin{equation*}
 \label{eqn:damage piecewise}
 H_{i}\left(\Nor\left(R^{(U)}\right)\right)\cong H_{i}\left(\Nor\left(R^{(\Delta^{0})}\right)\right)=
 \begin{dcases}
  R & \textrm{if } i=0 \\
  0 & \textrm{if } i>0
 \end{dcases}
\end{equation*}
\end{proof}

We next recall the Dold-Kan functor. We use the functorial point of view for simplicial modules in the next definition; see \cite[8.4.4]{We}.

\begin{definition} \label{4.9}
Let $R$ be a ring. The \textit{Dold-Kan functor} $\DK:\mathcal{C}_{\geq 0}(R) \rightarrow \mathpzc{s}\mathcal{M}\mathpzc{od}(R)$ is defined as follows:
\begin{enumerate}
\item[(i)] $\DK$ assigns to each connective $R$-complex $X$, the contravariant functor $\DK(X):\Delta \rightarrow \mathcal{M}\mathpzc{od}(R)$ defined as follows. If $[n] \in \Obj(\Delta)$, then we set
    $$\DK(X)\left([n]\right)=\textstyle\bigoplus_{\substack{{f\in \Mor_{\Delta}\left([n],[k]\right)} \\ {f \textrm{ surjective}}}}X_{k}$$
    which is a finite direct sum since there is no surjective map $[n]\rightarrow [k]$ if $k>n$. If $\eta \in \Mor_{\Delta}\left([n],[m]\right)$, then the $R$-homomorphism
    $$\DK(X)(\eta):\DK(X)\left([m]\right)=\textstyle\bigoplus_{\substack{{g\in \Mor_{\Delta}\left([m],[l]\right)} \\ {g \textrm{ surjective}}}}X_{l} \rightarrow \textstyle\bigoplus_{\substack{{f\in \Mor_{\Delta}\left([n],[k]\right)} \\ {f \textrm{ surjective}}}}X_{k}=\DK(X)\left([n]\right)$$
    is given by a matrix $\left(\DK(X)(\eta)_{f,g}:X_{l} \rightarrow X_{k}\right)$, where we consider the unique factorization of $g\eta$ as
\begin{equation*}
  \begin{tikzcd}
  {[n]} \arrow{r}{\eta} \arrow{d}[swap]{f}
  & {[m]} \arrow{d}{g}
  \\
  {[k]} \arrow{r}{\iota} & {[l]}
\end{tikzcd}
\end{equation*}

\noindent
in which $\iota$ is injective and $f$ is surjective (see Remark \ref{4.2}), and then we set:
\begin{equation*}
 \label{eqn:damage piecewise}
 \DK(X)(\eta)_{f,g}=
 \begin{dcases}
  1^{X_{l}} & \textrm{if } k=l \textrm{ and } \iota=1^{[l]} \\
  (-1)^{l}\partial_{l}^{X} & \textrm{if } k=l-1 \textrm{ and } \iota=\delta_{l,l} \\
  0 & \textrm{otherwise}
 \end{dcases}
\end{equation*}

\item[(ii)] $\DK$ assigns to each morphism $\phi:X \rightarrow Y$ of connective $R$-complexes, the natural transformation $\DK(\phi):\DK(X) \rightarrow \DK(Y)$, where for any $[n] \in \Obj(\Delta)$, the $R$-homomorphism
     $$\DK(X)(\phi)^{[n]}:\DK(X)\left([n]\right)=\textstyle\bigoplus_{\substack{{g\in \Mor_{\Delta}\left([n],[l]\right)} \\ {g \textrm{ surjective}}}}X_{l} \rightarrow \textstyle\bigoplus_{\substack{{f\in \Mor_{\Delta}\left([n],[k]\right)} \\ {f \textrm{ surjective}}}}Y_{k}=\DK(Y)\left([n]\right)$$
     is given by a matrix $\left(\DK(\phi)_{f,g}^{[n]}:X_{l} \rightarrow Y_{k}\right)$ where we have:
\begin{equation*}
 \label{eqn:damage piecewise}
 \DK(X)(\phi)_{f,g}^{[n]}=
 \begin{dcases}
  \phi_{k} & \textrm{if } f=g \\
  0 & \textrm{otherwise}
 \end{dcases}
\end{equation*}

\end{enumerate}
\end{definition}

In order to get a flavor of the Dold-Kan functor, we give the following example.

\begin{example} \label{4.10}
Let $R$ be a ring, and $X$ a connective $R$-complex. Let $\eta\in \Mor_{\Delta}\left([1],[2]\right)$ be given by $\eta(0)=0$ and $\eta(1)=1$. Consider the following $R$-homomorphism:
$$\DK(X)(\eta):\DK(X)\left([2]\right)=\textstyle\bigoplus_{\substack{{g\in \Mor_{\Delta}\left([2],[l]\right)} \\ {g \textrm{ surjective}}}}X_{l} \rightarrow \textstyle\bigoplus_{\substack{{f\in \Mor_{\Delta}\left([1],[k]\right)} \\ {f \textrm{ surjective}}}}X_{k}=\DK(X)\left([1]\right)$$
We first determine all the surjective maps $g$ in $\Mor_{\Delta}\left([2],[l]\right)$ for $0\leq l\leq 2$. We have:

Case I: $l=0$; $g_{1}:[2]\rightarrow [0]$ given by $g_{1}(0)=g_{1}(1)=g_{1}(2)=0$.

Case II: $l=1$; $g_{2}:[2]\rightarrow [1]$ given by $g_{2}(0)=0$, $g_{2}(1)=1$, and $g_{2}(2)=1$; and $g_{3}:[2]\rightarrow [1]$ given by $g_{3}(0)=0$, $g_{3}(1)=0$, and $g_{3}(2)=1$.

Case III: $l=2$; If $g_{4}:[2]\rightarrow [2]$ is surjective, then it must be injective as well, so $g_{4}=1^{[2]}$.

\noindent
We next determine all the surjective maps $f$ in $\Mor_{\Delta}\left([1],[k]\right)$ for $0\leq k\leq 1$. We have:

Case I: $k=0$; $f_{1}:[1]\rightarrow [0]$ given by $f_{1}(0)=f_{1}(1)=0$.

Case II: $k=1$; If $f_{2}:[1]\rightarrow [1]$ is surjective, then it must be injective as well, so $f_{2}=1^{[1]}$.

\noindent
By direct inspection, we observe that the only possible commutative diagrams are as follows:

\[
 \begin{tikzcd}
  {[1]} \arrow{r}{\eta} \arrow{d}[swap]{f_{1}} & {[2]} \arrow{d}{g_{1}} \\
  {[0]} \arrow{r}{1^{[0]}} & {[0]}
\end{tikzcd}
\quad \text{,} \quad
 \begin{tikzcd}
  {[1]} \arrow{r}{\eta} \arrow{d}[swap]{f_{2}} & {[2]} \arrow{d}{g_{2}} \\
  {[1]} \arrow{r}{1^{[1]}} & {[1]}
\end{tikzcd}
 \quad \text{,} \quad
 \begin{tikzcd}
  {[1]} \arrow{r}{\eta} \arrow{d}[swap]{f_{1}} & {[2]} \arrow{d}{g_{3}} \\
  {[0]} \arrow{r}{\delta_{1,1}} & {[1]}
\end{tikzcd}
\quad \text{,} \quad
 \begin{tikzcd}
  {[1]} \arrow{r}{\eta} \arrow{d}[swap]{f_{2}} & {[2]} \arrow{d}{g_{4}} \\
  {[1]} \arrow{r}{\delta_{2,2}} & {[2]}
\end{tikzcd}
\]

\noindent
It follows that
$$\DK(X)(\eta):\DK(X)\left([2]\right)= X_{0}\oplus X_{1} \oplus X_{1} \oplus X_{2} \rightarrow X_{0}\oplus X_{1} =\DK(X)\left([1]\right)$$
can be represented by the following matrix:
\begin{equation*}
\begin{pmatrix}
1^{X_{0}} & 0 & -\partial_{1}^{X} & 0 \\
0 & 1^{X_{1}} & 0 & \partial_{2}^{X}
\end{pmatrix}
\end{equation*}
\end{example}

The following proposition describes how the Dold-Kan functor interacts with surjections.

\begin{proposition} \label{4.11}
Let $R$ be a ring, and $X$ an $R$-complex. Let $\lambda_{f}^{X}:X_{k}\rightarrow \DK(X)\left([n]\right)$ be the canonical injection corresponding to the index $f\in \Mor_{\Delta}\left([n],[k]\right)$. If $\eta\in \Mor_{\Delta}\left([n],[m]\right)$, $g\in \Mor_{\Delta}\left([m],[l]\right)$ is surjective, and $x\in X_{l}$, then we have:
\begin{equation*}
 \label{eqn:damage piecewise}
 \DK(X)(\eta)\left(\lambda_{g}^{X}(x)\right)=
 \begin{dcases}
  \lambda_{g\eta}^{X}(x) & \textrm{if } g\eta \textrm{ is surjective} \\
  0 & \textrm{if } g\eta \textrm{ is not surjective but } m\in \im(\eta)
 \end{dcases}
\end{equation*}
\end{proposition}

\begin{proof}
Let $A$ denote the matrix representing $\DK(X)(\eta)$ as in Definition \ref{4.9}. First suppose that $g\eta$ is surjective. We focus on the column $g$ of $A$. By the uniqueness of the face-degeneracy factorization in Remark \ref{4.2}, we can only have the following commutative diagram:
\begin{equation*}
  \begin{tikzcd}
  {[n]} \arrow{r}{\eta} \arrow{d}[swap]{g\eta}
  & {[m]} \arrow{d}{g}
  \\
  {[l]} \arrow{r}{1^{[l]}} & {[l]}
\end{tikzcd}
\end{equation*}

\noindent
This means that the column $g$ of $A$ has $1^{X_{l}}$ in the position $(g\eta,g)$ and zero elsewhere. We note that $\lambda_{g}^{X}(x)$ has $x$ in its coordinate corresponding to $g$ and zero elsewhere. Considering $\lambda_{g}^{X}(x)$ as a column matrix, we see that the matrix product $A\lambda_{g}^{X}(x)$ is equal to the column matrix with only one non-zero coordinate $1^{X_{l}}(x)=x$ in position $g\eta$. It follows that $\DK(X)(\eta)\left(\lambda_{g}^{X}(x)\right)=A\lambda_{g}^{X}(x)=\lambda_{g\eta}^{X}(x)$.

Next suppose that $g\eta$ is not surjective but $m\in \im(\eta)$. We focus on the column $g$ of $A$. If there is a commutative diagram as
\begin{equation*}
  \begin{tikzcd}
  {[n]} \arrow{r}{\eta} \arrow{d}[swap]{f}
  & {[m]} \arrow{d}{g}
  \\
  {[l]} \arrow{r}{1^{[l]}} & {[l]}
\end{tikzcd}
\end{equation*}

\noindent
in which $f$ is surjective, then $g\eta=f$ is surjective which contradicts our assumption. Therefore, there is no such diagram. On the other hand, if there is a commutative diagram as
\begin{equation*}
  \begin{tikzcd}
  {[n]} \arrow{r}{\eta} \arrow{d}[swap]{f}
  & {[m]} \arrow{d}{g}
  \\
  {[l-1]} \arrow{r}{\delta_{l,l}} & {[l]}
\end{tikzcd}
\end{equation*}

\noindent
in which $f$ is surjective, then $g\eta=\delta_{l,l}f$. Let $g(m)=p\in [l]$. If $p\neq l$, then $p<l$. Since $g$ is surjective, there is an element $r\in [m]$ such that $g(r)=l$. But then $g(m)=p<l=g(r)$ which is a contradiction since $r\leq m$. Therefore, $g(m)=p=l$. Since $m\in \im(\eta)$, there is an element $s\in[n]$ such that $\eta(s)=m$. But then $\delta_{l,l}\left(f(s)\right)=g\left(\eta(s)\right)=g(m)=l$, so $l\in \im(\delta_{l,l})$ which is a contradiction. As a consequence, there is no such diagram either. This means that the column $g$ of $A$ is entirely zero. Considering $\lambda_{g}^{X}(x)$ as a column matrix, we see that the matrix product $A\lambda_{g}^{X}(x)$ is equal to the zero column matrix. It follows that $\DK(X)(\eta)\left(\lambda_{g}^{X}(x)\right)=A\lambda_{g}^{X}(x)=0$.
\end{proof}

\begin{corollary} \label{4.12}
Let $R$ be a ring, and $X$ an $R$-complex. Let $\lambda_{f}^{X}:X_{k}\rightarrow \DK(X)\left([n]\right)$ be the canonical injection corresponding to the index $f\in \Mor_{\Delta}\left([n],[k]\right)$. If $g\in \Mor_{\Delta}\left([n],[l]\right)$ is surjective and $x\in X_{l}$, then we have
\begin{equation*}
 \label{eqn:damage piecewise}
 d_{n,i}^{\DK(X)}\left(\lambda_{g}^{X}(x)\right)=
 \begin{dcases}
  \lambda_{g\delta_{n,i}}^{X}(x) & \textrm{if } 0\leq i \leq n \textrm{ and } g\delta_{n,i} \textrm{ is surjective} \\
  0 & \textrm{if } 0\leq i \leq n-1 \textrm{ and } g\delta_{n,i} \textrm{ is not surjective} \\
 (-1)^{l}\lambda_{\overline{g}}^{X}\left(\partial_{l}^{X}(x)\right) & \textrm{if } i=n \textrm{ and } g\delta_{n,n} \textrm{ is not surjective}
 \end{dcases}
\end{equation*}

\noindent
for every $0\leq i \leq n$, where $\overline{g}\in \Mor_{\Delta}\left([n-1],[l-1]\right)$ is induced by $g$.
\end{corollary}

\begin{proof}
Let $0\leq i \leq n$. We note that $\delta_{n,i}\in \Mor_{\Delta}\left([n-1],[n]\right)$ and $d_{n,i}^{\DK(X)}=\DK(X)(\delta_{n,i})$. Let $A_{n,i}$ denote the matrix representing $d_{n,i}^{\DK(X)}$ as in Definition \ref{4.9}. If $g\delta_{n,i}$ is surjective, then by Proposition \ref{4.11}, we have $d_{n,i}^{\DK(X)}\left(\lambda_{g}^{X}(x)\right)=\DK(X)(\delta_{n,i})\left(\lambda_{g}^{X}(x)\right)=\lambda_{g\delta_{n,i}}^{X}(x)$. If $0\leq i \leq n-1$, then $n\in \im(\delta_{n,i})$, so if in addition, $g\delta_{n,i}$ is not surjective, then Proposition \ref{4.11} implies that $d_{n,i}^{\DK(X)}\left(\lambda_{g}^{X}(x)\right)=\DK(X)(\delta_{n,i})\left(\lambda_{g}^{X}(x)\right)=0$. Now suppose that $g\delta_{n,n}$ is not surjective. We focus on the column $g$ of $A_{n,i}$. If there is a commutative diagram as
\begin{equation*}
  \begin{tikzcd}
  {[n-1]} \arrow{r}{\delta_{n,n}} \arrow{d}[swap]{f}
  & {[n]} \arrow{d}{g}
  \\
  {[l]} \arrow{r}{1^{[l]}} & {[l]}
\end{tikzcd}
\end{equation*}

\noindent
in which $f$ is surjective, then then $g\delta_{n,n}=f$ is surjective which contradicts our assumption. Therefore, there is no such diagram. On the other hand, as we saw in the proof of Proposition \ref{4.11}, since $g$ is surjective, we must have $g(n)=l$. But $g\delta_{n,n}$ is not surjective which means that if we omit $n$ from the domain of $g$, then it will no longer remain surjective. That is to say, $l$ is assumed by $g$ only at $n$. It follows that $g$ induces a surjective map $\overline{g}\in \Mor_{\Delta}\left([n-1],[l-1]\right)$ given by $\overline{g}(i)=g(i)$ for every $i\in [n-1]$. Therefore, we get the following commutative diagram:
\begin{equation*}
  \begin{tikzcd}
  {[n-1]} \arrow{r}{\delta_{n,n}} \arrow{d}[swap]{\overline{g}}
  & {[n]} \arrow{d}{g}
  \\
  {[l-1]} \arrow{r}{\delta_{l,l}} & {[l]}
\end{tikzcd}
\end{equation*}

\noindent
This implies that the column of the matrix $A_{n,i}$ corresponding to $g$ has all its entries zero except for the entry in the position $(\overline{g},g)$ which is equal to $(-1)^{l}\partial_{l}^{X}$. As a result, we have $d_{n,i}^{\DK(X)}\left(\lambda_{g}^{X}(x)\right)=A_{n,i}\lambda_{g}^{X}(x)= (-1)^{l}\lambda_{\overline{g}}^{X}\left(\partial_{l}^{X}(x)\right)$.
\end{proof}

\begin{corollary} \label{4.13}
Let $R$ be a ring, and $X$ an $R$-complex. Let $\lambda_{f}^{X}:X_{k}\rightarrow \DK(X)\left([n]\right)$ be the canonical injection corresponding to the index $f\in \Mor_{\Delta}\left([n],[k]\right)$. If $g\in \Mor_{\Delta}\left([n],[l]\right)$ is surjective and $x\in X_{l}$, then we have
\begin{equation*}
 \label{eqn:damage piecewise}
 s_{n,i}^{\DK(X)}\left(\lambda_{g}^{X}(x)\right)=\lambda_{g\sigma_{n,i}}^{X}(x)
\end{equation*}

\noindent
for every $0\leq i \leq n$.
\end{corollary}

\begin{proof}
Let $0\leq i \leq n$. We note that $\sigma_{n,i}\in \Mor_{\Delta}\left([n+1],[n]\right)$ and $s_{n,i}^{\DK(X)}=\DK(X)(\sigma_{n,i})$. Since $\sigma_{n,i}$ is surjective, we conclude that $g\sigma_{n,i}$ is surjective, so by Proposition \ref{4.11}, we have $s_{n,i}^{\DK(X)}\left(\lambda_{g}^{X}(x)\right)=\DK(X)(\sigma_{n,i})\left(\lambda_{g}^{X}(x)\right)=\lambda_{g\sigma_{n,i}}^{X}(x)$.
\end{proof}

Here is the celebrated Dold-Kan Correspondence.

\begin{theorem} \label{4.14}
Let $R$ be a ring. Then there is an adjoint equivalence of categories as follows:
$$(\DK,\Nor):\mathcal{C}_{\geq 0}(R) \leftrightarrows \mathpzc{s}\mathcal{M}\mathpzc{od}(R)$$
\end{theorem}

\begin{proof}
See \cite[Theorem 8.4.1 and Exercise 8.4.2]{We}.
\end{proof}

It is worth noting that under the Dold-Kan Correspondence, the connective chain complex concentrated in degree zero corresponds to the constant simplicial module.

The cofibrant generation property of the model structure of connective chain complexes in conjunction with the Dold-Kan Correspondence allow us to transfer the model structure of connective chain complexes to simplicial modules.

\begin{theorem} \label{4.15}
Let $R$ be a ring. Then $\mathpzc{s}\mathcal{M}\mathpzc{od}(R)$ is a cofibrantly generated model category with the model structure given as follows:
\begin{enumerate}
\item[(i)] Fibrations are morphisms $f:M\rightarrow N$ in $\mathpzc{s}\mathcal{M}\mathpzc{od}(R)$ such that $\Nor(f)_{i}:\Nor(M)_{i}\rightarrow \Nor(N)_{i}$ is surjective for every $i\geq 1$.
\item[(ii)] Cofibrations are morphisms $f:M\rightarrow N$ in $\mathpzc{s}\mathcal{M}\mathpzc{od}(R)$ such that $\Nor(f)_{i}:\Nor(M)_{i}\rightarrow \Nor(N)_{i}$ is injective and $\Coker\left(\Nor(f)_{i}\right)$ is a projective left $R$-module for every $i\geq 0$.
\item[(iii)] Weak equivalences are morphisms $f:M\rightarrow N$ in $\mathpzc{s}\mathcal{M}\mathpzc{od}(R)$ such that $\Nor(f):\Nor(M)\rightarrow \Nor(N)$ is a quasi-isomorphism in $\mathcal{C}_{\geq 0}(R)$.
\end{enumerate}
\end{theorem}

\begin{proof}
Follows from Corollary \ref{2.16}, Theorem \ref{3.1}, and Theorem \ref{4.14}.
\end{proof}

\section{Shuffle Product and a Variant of Eilenberg-Zilber Theorem}

In this section, we introduce the notion of shuffle product of connective chain complexes and prove a variant of the Eilenberg-Zilber Theorem. This will equip us with a formidable tool to overcome the acyclicity condition in Theorem \ref{2.9}.

We first review the concept of a shuffle. For any $p,q\geq 0$, let $S_{p+q}$ denote the symmetric group on the set $\{1,2,...,p+q\}$, and $\Sh(p,q)$ denote the set of $(p,q)$-shuffles in $S_{p+q}$, i.e.
$$\Sh(p,q)=\left\{\nu \in S_{p+q}\mid \nu(1)< \cdots <\nu(p) \textrm{ and } \nu(p+1)< \cdots <\nu(p+q)\right\}.$$
We note that $|\Sh(p,q)|=\binom{p+q}{p}$. For any $\nu \in \Sh(p,q)$, consider the set:
$$A_{\nu}=\left\{\left(\nu(\alpha),\nu(\beta)\right)\mid 1\leq \alpha \leq p\textrm{, } p+1 \leq \beta \leq p+q, \textrm{ and } \nu(\alpha)>\nu(\beta)\right\}$$
Then $\sgn(\nu)=(-1)^{|A_{\nu}|}$. Indeed, we fix the elements $\nu(p+1)< \cdots < \nu(p+q)$, and then for each $\alpha=p,p-1,...,1$, we compare $\nu(\alpha)$ with these elements to see how many moves to the right are needed in order to put $\nu(\alpha)$ in the right place.

We next fix some notations in the next remark.

\begin{remark} \label{5.1}
Let $n\geq 0$ and $0\leq k,l \leq n$. Moreover, let $f\in \Mor_{\Delta}\left([n],[k]\right)$ and $g\in \Mor_{\Delta}\left([n],[l]\right)$ be surjective. Define a map $\phi_{f,g}:[n]\rightarrow [k]\times [l]$ by setting $\phi_{f,g}(i)=\left(f(i),g(i)\right)$ for every $i\in [n]$. Suppose that $f$ and $g$ are not bijective. As in Remark \ref{4.2}, let $0\leq j_{1}< \cdots <j_{n-k}<n$ be the elements of $[n]$ for which $f(j)=f(j+1)$ so that $f=\sigma_{k,j_{1}} \cdots \sigma_{n-1,j_{n-k}}$, and $0\leq i_{1}< \cdots <i_{n-l}<n$ the elements of $[n]$ for which $g(i)=g(i+1)$ so that $g=\sigma_{l,i_{1}} \cdots \sigma_{n-1,i_{n-l}}$. If $\phi_{f,g}$ is not injective, then there are integers $m,m'\in [n]$ with $m<m'$ such that $\phi_{f,g}(m)=\phi_{f,g}(m')$, so we get $f(m)=f(m')$ and $g(m)=g(m')$. But $f(m)\leq f(m+1)\leq f(m')$ and $g(m)\leq g(m+1)\leq g(m')$, so $f(m)=f(m+1)$ and $g(m)=g(m+1)$. This means that $m\in \{i_{1},...,i_{n-l}\} \cap \{j_{1},...,j_{n-k}\}$. Conversely, if $m\in \{i_{1},...,i_{n-l}\} \cap \{j_{1},...,j_{n-k}\}$, then $\phi_{f,g}(m)=\phi_{f,g}(m+1)$, so $\phi_{f,g}$ is not injective. As a result, $\phi_{f,g}$ is injective if and only if $\{i_{1},...,i_{n-l}\} \cap \{j_{1},...,j_{n-k}\}=\emptyset$. If $f$ is bijective, then $f=1^{[n]}$ , so $\{j_{1},...,j_{n-k}\}=\emptyset$, and $\phi_{f,g}$ is injective. A similar result holds when $g$ is bijective. Including these two degenerate cases as well, we can state in general that $\phi_{f,g}$ is injective if and only if $\{i_{1},...,i_{n-l}\} \cap \{j_{1},...,j_{n-k}\}=\emptyset$.

Next let $n\geq 0$ and $0\leq k,l \leq n$ with $k+l=n$. Set:
$$\mathcal{S}_{n}^{k,l}= \left\{(f,g) \in \Mor_{\Delta}([n],[k])\times \Mor_{\Delta}([n],[l]) \suchthat \begin{tabular}{ccc}
$f$ and $g$ are surjective, and $\phi_{f,g}:[n]\rightarrow [k]\times[l]$ given \\
by $\phi_{f,g}(i)=(f(i),g(i))$ for every $i\in [n]$ is injective
  \end{tabular} \right\}$$
Let $(f,g)\in \mathcal{S}_{n}^{k,l}$. As before, let $0\leq j_{1}< \cdots <j_{l}<n$ be the elements of $[n]$ for which $f(j)=f(j+1)$ so that $f=\sigma_{k,j_{1}} \cdots \sigma_{n-1,j_{l}}$, and $0\leq i_{1}< \cdots <i_{k}<n$ the elements of $[n]$ for which $g(i)=g(i+1)$ so that $g=\sigma_{l,i_{1}} \cdots \sigma_{n-1,i_{k}}$. Define a shuffle $\nu_{f,g}\in \Sh(k,l)$ by setting $\nu_{f,g}(\alpha)=i_{\alpha}+1$ for every $1\leq \alpha \leq k$ and $\nu_{f,g}(\alpha)=j_{\alpha-k}+1$ for every $k+1\leq \alpha \leq n$. In other words, $\nu_{f,g}$ can be represented as follows:
\begin{equation*}
\nu_{f,g} =
\begin{pmatrix}
1 & \cdots & k & k+1 & \cdots & n \\
i_{1}+1 & \cdots & i_{k}+1 & j_{1}+1 & \cdots & j_{l}+1
\end{pmatrix}
\end{equation*}

\noindent
In the degenerate case where $f$ is bijective, we have $k=n$, so $l=0$, whence $g=0$, implying that $\{i_{1},...,i_{n}\}=\{0,...,n-1\}$. Therefore, $\nu_{f,g}=1^{[n]}$, so $\sgn(\nu_{f,g})=1$. The degenerate case where $g$ is bijective leads to a similar conclusion. Now define a map $\eta:\mathcal{S}_{n}^{k,l} \rightarrow \Sh(k,l)$ by setting $\eta(f,g)=\nu_{f,g}$ for every $(f,g)\in \mathcal{S}_{n}^{k,l}$. On the other hand, for any shuffle $\nu \in \Sh(k,l)$, define two maps as follows:
$$f_{n,\nu}=\sigma_{k,\nu(k+1)-1} \cdots \sigma_{n-1,\nu(n)-1}:[n] \rightarrow [k]$$
and
$$g_{n,\nu}=\sigma_{l,\nu(1)-1} \cdots \sigma_{n-1,\nu(k)-1}:[n] \rightarrow [l]$$
Then it is clear that $(f_{n,\nu},g_{n,\nu})\in \mathcal{S}_{n}^{k,l}$. Next define a map $\theta:\Sh(k,l) \rightarrow \mathcal{S}_{n}^{k,l}$ by setting $\theta(\nu)=(f_{n,\nu},g_{n,\nu})$ for every $\nu \in \Sh(k,l)$. It is easily seen that $\eta$ and $\theta$ are mutually inverse maps. Hence the two sets $ \mathcal{S}_{n}^{k,l}$ and $\Sh(k,l)$ are in 1-1 correspondence.
\end{remark}

We now define the shuffle product of connective chain complexes using the Dold-Kan functor. The terminology is justified by Remark \ref{5.1}.

\begin{definition} \label{5.2}
Let $R$ be a ring, $X$ a connective $R^{\op}$-complex, and $Y$ a connective $R$-complex. Then the \textit{shuffle product} $X\boxtimes_{R} Y$ is defined as follows. For any $n\geq 0$, set:
\small
$$\mathcal{S}_{n}= \left\{(f,g) \in \Mor_{\Delta}\left([n],[k]\right)\times \Mor_{\Delta}\left([n],[l]\right) \suchthat \begin{tabular}{ccc}
 $0 \leq k,l \leq n$, $f$ and $g$ are surjective, and $\phi_{f,g}:[n]\rightarrow [k]\times[l]$ \\
 given by $\phi_{f,g}(i)=\left(f(i),g(i)\right)$ for every $i\in [n]$ is injective
 \end{tabular} \right\}$$
\normalsize
Then we have
$$(X\boxtimes_{R}Y)_{n}:= \textstyle\bigoplus_{(f,g)\in \mathcal{S}_{n}}(X_{k}\otimes_{R}Y_{l})$$
for every $n\geq 0$. Moreover, if $\lambda_{f}^{X}:X_{k}\rightarrow \DK(X)\left([n]\right)$ is the canonical injection corresponding to the index $f \in \Mor_{\Delta}\left([n],[k]\right)$, $\lambda_{g}^{Y}:Y_{l}\rightarrow \DK(Y)\left([n]\right)$ is the canonical injection corresponding to the index $g \in \Mor_{\Delta}\left([n],[l]\right)$, and $\lambda_{f,g}^{X,Y}:X_{k}\otimes_{R}Y_{l}\rightarrow (X\boxtimes_{R}Y)_{n}$ is the canonical injection corresponding to the indices $f\in \Mor_{\Delta}\left([n],[k]\right)$ and $g \in \Mor_{\Delta}\left([n],[l]\right)$, then the differential
$$\partial_{n}^{X\boxtimes_{R}Y}:(X\boxtimes_{R}Y)_{n}\rightarrow (X\boxtimes_{R}Y)_{n-1}$$
is given by
$$\partial_{n}^{X\boxtimes_{R}Y}\left(\lambda_{f,g}^{X,Y}(x\otimes y)\right):=\textstyle \sum_{i=0}^{n}(-1)^{i} d_{n,i}^{\DK(X)}\left(\lambda_{f}^{X}(x)\right)\otimes d_{n,i}^{\DK(Y)}\left(\lambda_{g}^{Y}(y)\right)$$
for every $n\geq 0$, $(f,g)\in \mathcal{S}_{n}$, $x\in X_{k}$, and $y\in Y_{l}$.
\end{definition}

We need to first show that the shuffle product of two complexes is indeed a complex.

\begin{proposition} \label{5.3}
Let $R$ be a ring, $X$ a connective $R^{\op}$-complex, and $Y$ a connective $R$-complex. Then $X\boxtimes_{R} Y$ is a connective $\mathbb{Z}$-complex.
\end{proposition}

\begin{proof}
Let $n\geq 0$. Then we have:
\begin{equation*}
\begin{split}
 C\left(\DK(X)\otimes_{R}\DK(Y)\right)_{n} & = \left(\DK(X)\otimes_{R}\DK(Y)\right)_{n} = \DK(X)\left([n]\right)\otimes_{R}\DK(Y)\left([n]\right) \\
 & = \left(\textstyle\bigoplus_{\substack{{f\in \Mor_{\Delta}\left([n],[k]\right)} \\ {f \textrm{ surjective}}}} X_{k}\right)\otimes_{R}\left(\textstyle\bigoplus_{\substack{{g\in \Mor_{\Delta}\left([n],[l]\right)} \\ {g \textrm{ surjective}}}} Y_{l}\right) \\
 & \cong \textstyle\bigoplus_{\substack{{f\in \Mor_{\Delta}\left([n],[k]\right)} \\ {g\in \Mor_{\Delta}\left([n],[l]\right)} \\ {f \textrm{ and } g \textrm{ surjective}}}} (X_{k}\otimes_{R}Y_{l})
\end{split}
\end{equation*}

\noindent
We suppress the above isomorphism and consider it as an equality in what follows. Let $\lambda_{f}^{X}:X_{k}\rightarrow \DK(X)\left([n]\right)$ be the canonical injection corresponding to the index $f \in \Mor_{\Delta}\left([n],[k]\right)$, $\lambda_{g}^{Y}:Y_{l}\rightarrow \DK(Y)\left([n]\right)$ the canonical injection corresponding to the index $g \in \Mor_{\Delta}\left([n],[l]\right)$, and $\lambda_{f,g}^{X,Y}:X_{k}\otimes_{R}Y_{l}\rightarrow C\left(\DK(X)\otimes_{R}\DK(Y)\right)_{n}$ the canonical injection corresponding to the indices $f\in \Mor_{\Delta}\left([n],[k]\right)$ and $g \in \Mor_{\Delta}\left([n],[l]\right)$. Then it is clear that $\lambda_{f}^{X}(x)\otimes \lambda_{g}^{Y}(y)=\lambda_{f,g}^{X,Y}(x\otimes y)$ for every $x\in X_{k}$ and $y\in Y_{l}$, so we have $\lambda_{f}^{X}\otimes_{R}\lambda_{g}^{Y}=\lambda_{f,g}^{X,Y}$. As a result, we get
\begin{equation*}
\begin{split}
 \partial_{n}^{C\left(\DK(X)\otimes_{R}\DK(Y)\right)}\left(\lambda_{f,g}^{X,Y}(x\otimes y)\right) & = \textstyle\sum_{i=0}^{n}(-1)^{i}d_{n,i}^{\DK(X)\otimes_{R}\DK(Y)}\left(\lambda_{f,g}^{X,Y}(x\otimes y)\right) \\
 & = \textstyle\sum_{i=0}^{n}(-1)^{i}\left(d_{n,i}^{\DK(X)}\otimes_{R}d_{n,i}^{\DK(Y)}\right)\left(\lambda_{f}^{X}(x)\otimes \lambda_{g}^{Y}(y)\right) \\
 & = \textstyle\sum_{i=0}^{n}(-1)^{i}d_{n,i}^{\DK(X)}\left(\lambda_{f}^{X}(x)\right) \otimes d_{n,i}^{\DK(Y)}\left(\lambda_{g}^{Y}(y)\right)
\end{split}
\end{equation*}

\noindent
for every $\lambda_{f,g}^{X,Y}(x\otimes y)\in C\left(\DK(X)\otimes_{R}\DK(Y)\right)_{n}$. We show that $X\boxtimes_{R} Y$ is a subcomplex of $C\left(\DK(X)\otimes_{R}\DK(Y)\right)$. Let $n\geq 0$, and let $\mathcal{S}_{n}$ be as in Definition \ref{5.2}. Then we can write:
\begin{equation*}
\begin{split}
 C\left(\DK(X)\otimes_{R}\DK(Y)\right)_{n} & = \left(\textstyle\bigoplus_{(f,g)\in \mathcal{S}_{n}}(X_{k}\otimes_{R}Y_{l})\right) \oplus \left(\textstyle\bigoplus_{(f,g)\notin \mathcal{S}_{n}}(X_{k}\otimes_{R}Y_{l})\right) \\
 & = (X\boxtimes_{R} Y)_{n} \oplus \left(\textstyle\bigoplus_{(f,g)\notin \mathcal{S}_{n}}(X_{k}\otimes_{R}Y_{l})\right)
\end{split}
\end{equation*}

\noindent
Next let $(f,g)\in \mathcal{S}_{n}$. We show that $\partial_{n}^{C\left(\DK(X)\otimes_{R}\DK(Y)\right)}\left(\lambda_{f,g}^{X,Y}(x\otimes y)\right) \in (X\boxtimes_{R}Y)_{n-1}$ for every $x\in X_{k}$ and $y\in Y_{l}$. To this end, fix $0\leq i\leq n$, and consider the following cases:

Case I: $f\delta_{n,i}$ is not surjective but $g\delta_{n,i}$ is surjective

If $0\leq i\leq n-1$, then by Corollary \ref{4.12}, $d_{n,i}^{\DK(X)}\left(\lambda_{f}^{X}(x)\right)=0$, so $d_{n,i}^{\DK(X)}\left(\lambda_{f}^{X}(x)\right) \otimes d_{n,i}^{\DK(Y)}\left(\lambda_{g}^{Y}(y)\right)=0 \in (X\boxtimes_{R}Y)_{n-1}$. Hence we can assume that $i=n$. In this case, as in the proof of Corollary \ref{4.12}, $f$ induces a surjective map $\overline{f} \in \Mor_{\Delta}\left([n-1],[k-1]\right)$ given by $\overline{f}(i)=f(i)$ for every $i\in [n-1]$. Accordingly, using Corollary \ref{4.12}, we have:
$$d_{n,n}^{\DK(X)}\left(\lambda_{f}^{X}(x)\right) \otimes d_{n,n}^{\DK(Y)}\left(\lambda_{g}^{Y}(y)\right)= (-1)^{k}\lambda_{\overline{f}}^{X}\left(\partial_{k}^{X}(x)\right)\otimes \lambda_{g\delta_{n,n}}^{Y}(y)= (-1)^{k}\lambda_{\overline{f},g\delta_{n,n}}^{X,Y}\left(\partial_{k}^{X}(x)\otimes y\right)$$
If $\left(\overline{f},g\delta_{n,n}\right)\notin \mathcal{S}_{n-1}$, then $\phi_{\overline{f},g\delta_{n,n}}$ is not injective, so there is an $0\leq m\leq n-2$ such that $\phi_{\overline{f},g\delta_{n,n}}(m)=\phi_{\overline{f},g\delta_{n,n}}(m+1)$. But we note that:
$$\phi_{\overline{f},g\delta_{n,n}}(m)=\left(\overline{f}(m),(g\delta_{n,n})(m)\right)=\left(f(m),g(m)\right)=\phi_{f,g}(m)$$
and
$$\phi_{\overline{f},g\delta_{n,n}}(m+1)=\left(\overline{f}(m+1),(g\delta_{n,n})(m+1)\right)=\left(f(m+1),g(m+1)\right)=\phi_{f,g}(m+1)$$
It follows that $\phi_{f,g}$ is not injective, so $(f,g) \notin \mathcal{S}_{n}$ which is a contradiction. Therefore, $\left(\overline{f},g\delta_{n,n}\right)\in \mathcal{S}_{n-1}$, so $\lambda_{\overline{f},g\delta_{n,n}}^{X,Y}\left(\partial_{k}^{X}(x)\otimes y\right)\in (X\boxtimes_{R}Y)_{n-1}$.

Case II: $f\delta_{n,i}$ is surjective but $g\delta_{n,i}$ is not surjective

This case is similar to Case I.

Case III: $f\delta_{n,i}$ and $g\delta_{n,i}$ are both surjective

In this case, using Corollary \ref{4.12}, we have:
$$d_{n,i}^{\DK(X)}\left(\lambda_{f}^{X}(x)\right) \otimes d_{n,i}^{\DK(Y)}\left(\lambda_{g}^{Y}(y)\right)= \lambda_{f\delta_{n,i}}^{X}(x)\otimes \lambda_{g\delta_{n,i}}^{Y}(y)= \lambda_{f\delta_{n,i},g\delta_{n,i}}^{X,Y}\left(x\otimes y\right)$$
If $\left(f\delta_{n,i},g\delta_{n,i}\right)\notin \mathcal{S}_{n-1}$, then $\phi_{f\delta_{n,i},g\delta_{n,i}}$ is not injective, so there is an $0\leq m\leq n-2$ such that $\phi_{f\delta_{n,i},g\delta_{n,i}}(m)=\phi_{f\delta_{n,i},g\delta_{n,i}}(m+1)$. But we note that:
\begin{equation*}
 \label{eqn:damage piecewise}
 \phi_{f\delta_{n,i},g\delta_{n,i}}(m)=\left((f\delta_{n,i})(m),(g\delta_{n,i})(m)\right)=
 \begin{dcases}
  \left(f(m+1),g(m+1)\right)=\phi_{f,g}(m+1) & \textrm{if } i \leq m \\
  \left(f(m),g(m)\right)=\phi_{f,g}(m) & \textrm{if } i = m+1 \\
  \left(f(m),g(m)\right)=\phi_{f,g}(m) & \textrm{if } m+1 < i
 \end{dcases}
\end{equation*}

\noindent
and
\begin{equation*}
 \label{eqn:damage piecewise}
 \phi_{f\delta_{n,i},g\delta_{n,i}}(m+1)=\left((f\delta_{n,i})(m+1),(g\delta_{n,i})(m+1)\right)=
 \begin{dcases}
  \left(f(m+2),g(m+2)\right)=\phi_{f,g}(m+2) & \textrm{if } i \leq m \\
  \left(f(m+2),g(m+2)\right)=\phi_{f,g}(m+2) & \textrm{if } i = m+1 \\
  \left(f(m+1),g(m+1)\right)=\phi_{f,g}(m+1) & \textrm{if } m+1 < i
 \end{dcases}
\end{equation*}

\noindent
In any case, it follows that $\phi_{f,g}$ is not injective, so $(f,g)\notin \mathcal{S}_{n}$ which is a contradiction. Therefore, $\left(f\delta_{n,i},g\delta_{n,i}\right)\in \mathcal{S}_{n-1}$, so $\lambda_{f\delta_{n,i},g\delta_{n,i}}^{X,Y}\left(x\otimes y\right)\in (X\boxtimes_{R}Y)_{n-1}$.

Case IV: $f\delta_{n,i}$ and $g\delta_{n,i}$ are neither surjective

If $0\leq i\leq n-1$, then by Corollary \ref{4.12}, $d_{n,i}^{\DK(X)}\left(\lambda_{f}^{X}(x)\right)=0=d_{n,i}^{\DK(Y)}\left(\lambda_{g}^{Y}(y)\right)$, so $d_{n,i}^{\DK(X)}\left(\lambda_{f}^{X}(x)\right) \otimes d_{n,i}^{\DK(Y)}\left(\lambda_{g}^{Y}(y)\right)=0 \in (X\boxtimes_{R}Y)_{n-1}$. Hence we can assume that $i=n$. In this case, as in the proof of Corollary \ref{4.12}, $f$ induces a surjective map $\overline{f} \in \Mor_{\Delta}\left([n-1],[k-1]\right)$ given by $\overline{f}(i)=f(i)$ for every $i\in [n-1]$, and $g$ induces a surjective map $\overline{g} \in \Mor_{\Delta}\left([n-1],[l-1]\right)$ given by $\overline{g}(i)=g(i)$ for every $i\in [n-1]$. Accordingly, using Corollary \ref{4.12}, we have:
$$d_{n,n}^{\DK(X)}\left(\lambda_{f}^{X}(x)\right) \otimes d_{n,n}^{\DK(Y)}\left(\lambda_{g}^{Y}(y)\right)= (-1)^{k}\lambda_{\overline{f}}^{X}\left(\partial_{k}^{X}(x)\right)\otimes (-1)^{l}\lambda_{\overline{g}}^{Y}\left(\partial_{l}^{Y}(y)\right)= (-1)^{k+l}\lambda_{\overline{f},\overline{g}}^{X,Y}\left(\partial_{k}^{X}(x)\otimes \partial_{l}^{Y}(y)\right)$$
If $\left(\overline{f},\overline{g}\right)\notin \mathcal{S}_{n-1}$, then $\phi_{\overline{f},\overline{g}}$ is not injective, so there is an $0\leq m\leq n-2$ such that $\phi_{\overline{f},\overline{g}}(m)=\phi_{\overline{f},\overline{g}}(m+1)$. But we note that:
$$\phi_{\overline{f},\overline{g}}(m)=\left(\overline{f}(m),\overline{g}(m)\right)=\left(f(m),g(m)\right)=\phi_{f,g}(m)$$
and
$$\phi_{\overline{f},\overline{g}}(m+1)=\left(\overline{f}(m+1),\overline{g}(m+1)\right)=\left(f(m+1),g(m+1)\right)=\phi_{f,g}(m+1)$$
It follows that $\phi_{f,g}$ is not injective, so $(f,g) \notin \mathcal{S}_{n}$ which is a contradiction. Therefore, $\left(\overline{f},\overline{g}\right)\in \mathcal{S}_{n-1}$, so $\lambda_{\overline{f},\overline{g}}^{X,Y}\left(\partial_{k}^{X}(x)\otimes \partial_{l}^{Y}(y)\right)\in (X\boxtimes_{R}Y)_{n-1}$.

All in all, we conclude that
$$\partial_{n}^{C\left(\DK(X)\otimes_{R}\DK(Y)\right)}\left(\lambda_{f,g}^{X,Y}(x\otimes y)\right) = \textstyle\sum_{i=0}^{n}(-1)^{i}d_{n,i}^{\DK(X)}\left(\lambda_{f}^{X}(x)\right) \otimes d_{n,i}^{\DK(Y)}\left(\lambda_{g}^{Y}(y)\right) \in (X\boxtimes_{R}Y)_{n-1}$$
for every $x\in X_{k}$ and $y\in Y_{l}$. It now follows that $X\boxtimes_{R} Y$ is a subcomplex of $C\left(\DK(X)\otimes_{R}\DK(Y)\right)$.
\end{proof}

We next establish the functoriality of the shuffle product.

\begin{proposition} \label{5.4}
Let $R$ be a ring. Then the following assertions hold:
\begin{enumerate}
\item[(i)] If $X$ is a connective $R^{\op}$-complex, then $X\boxtimes_{R}-:\mathcal{C}_{\geq 0}(R)\rightarrow \mathcal{C}_{\geq 0}(\mathbb{Z})$ defines a covariant functor.
\item[(ii)] If $Y$ is a connective $R$-complex, then $-\boxtimes_{R}Y:\mathcal{C}_{\geq 0}(R^{\op})\rightarrow \mathcal{C}_{\geq 0}(\mathbb{Z})$ defines a covariant functor.
\end{enumerate}
\end{proposition}

\begin{proof}
(i): Let $n\geq 0$, and let $\mathcal{S}_{n}$ be as in Definition \ref{5.2}. Let $X$ be a connective $R^{\op}$-complex. If $Y$ is a connective $R$-complex, then by Proposition \ref{5.3}, $X\boxtimes_{R}Y$ is a connective $\mathbb{Z}$-complex. On the other hand, if $\theta:Y\rightarrow Y'$ is a morphism of connective $R$-complexes, then we define the morphism $X\boxtimes_{R}\theta:X\boxtimes_{R}Y \rightarrow X\boxtimes_{R}Y'$ of $\mathbb{Z}$-complexes as follows. Let $n\geq 0$. Then we have $(X\boxtimes_{R}Y)_{n}= \textstyle\bigoplus_{(f,g)\in \mathcal{S}_{n}}(X_{k}\otimes_{R}Y_{l})$ and $(X\boxtimes_{R}Y')_{n}= \textstyle\bigoplus_{(f,g)\in \mathcal{S}_{n}}(X_{k}\otimes_{R}Y_{l}')$. We set
$$(X\boxtimes_{R}\theta)_{n}\left(\lambda_{f,g}^{X,Y}(x\otimes y)\right):= \lambda_{f,g}^{X,Y'}\left(x\otimes \theta_{l}(y)\right)$$
for every $(f,g)\in \mathcal{S}_{n}$, $x\in X_{k}$, and $y\in Y_{l}$. Extending it linearly to $(X\boxtimes_{R}Y)_{n}$, we see that $(X\boxtimes_{R}\theta)_{n}$ is a $\mathbb{Z}$-homomorphism. Moreover, the following diagram is commutative:
\begin{equation*}
\begin{tikzcd}[column sep=3.5em,row sep=2em]
  (X\boxtimes_{R} Y)_{n} \arrow{r}{\partial_{n}^{X\boxtimes_{R}Y}} \arrow{d}[swap]{(X\boxtimes_{R}\theta)_{n}}
  & (X\boxtimes_{R} Y)_{n-1} \arrow{d}{(X\boxtimes_{R}\theta)_{n-1}}
  \\
  (X\boxtimes_{R} Y')_{n} \arrow{r}{\partial_{n}^{X\boxtimes_{R}Y'}}
  & (X\boxtimes_{R} Y')_{n-1}
\end{tikzcd}
\end{equation*}

\noindent
Indeed, we have for every $(f,g)\in \mathcal{S}_{n}$, $x\in X_{k}$, and $y\in Y_{l}$:
\begin{equation*}
\begin{split}
 \partial_{n}^{X\boxtimes_{R}Y'}\left((X\boxtimes_{R}\theta)_{n}\left(\lambda_{f,g}^{X,Y}(x\otimes y)\right)\right) & = \partial_{n}^{X\boxtimes_{R}Y'}\left(\lambda_{f,g}^{X,Y'}\left(x\otimes \theta_{l}(y)\right)\right) \\
 & = \textstyle\sum_{i=0}^{n}(-1)^{i} d_{n,i}^{\DK(X)}\left(\lambda_{f}^{X}(x)\right)\otimes d_{n,i}^{\DK(Y')}\left(\lambda_{g}^{Y'}\left(\theta_{l}(y)\right)\right)
\end{split}
\end{equation*}

\noindent
and
\begin{equation*}
\begin{split}
 (X\boxtimes_{R}\theta)_{n-1}\left(\partial_{n}^{X\boxtimes_{R}Y}\left(\lambda_{f,g}^{X,Y}(x\otimes y)\right)\right) & = (X\boxtimes_{R}\theta)_{n-1}\left(\textstyle\sum_{i=0}^{n}(-1)^{i} d_{n,i}^{\DK(X)}\left(\lambda_{f}^{X}(x)\right)\otimes d_{n,i}^{\DK(Y)}\left(\lambda_{g}^{Y}(y)\right)\right) \\
 & = \textstyle\sum_{i=0}^{n}(-1)^{i} (X\boxtimes_{R}\theta)_{n-1}\left(d_{n,i}^{\DK(X)}\left(\lambda_{f}^{X}(x)\right)\otimes d_{n,i}^{\DK(Y)}\left(\lambda_{g}^{Y}(y)\right)\right)
\end{split}
\end{equation*}

\noindent
Let $0\leq i\leq n$, and consider the following cases from Proposition \ref{5.3}:

Case I: $f\delta_{n,i}$ is not surjective but $g\delta_{n,i}$ is surjective

If $0\leq i\leq n-1$, then $d_{n,i}^{\DK(X)}\left(\lambda_{f}^{X}(x)\right)=0$, so we have:
$$d_{n,i}^{\DK(X)}\left(\lambda_{f}^{X}(x)\right) \otimes d_{n,i}^{\DK(Y')}\left(\lambda_{g}^{Y'}\left(\theta_{l}(y)\right)\right) = 0 = d_{n,i}^{\DK(X)}\left(\lambda_{f}^{X}(x)\right) \otimes d_{n,i}^{\DK(Y)}\left(\lambda_{g}^{Y}(y)\right)$$
If $i=n$, then we have:
\begin{equation*}
\begin{split}
 d_{n,n}^{\DK(X)}\left(\lambda_{f}^{X}(x)\right) \otimes d_{n,n}^{\DK(Y')}\left(\lambda_{g}^{Y'}\left(\theta_{l}(y)\right)\right) & = (-1)^{k}\lambda_{\overline{f},g\delta_{n,n}}^{X,Y'}\left(\partial_{k}^{X}(x)\otimes \theta_{l}(y)\right) \\
 & = (X\boxtimes_{R}\theta)_{n-1}\left((-1)^{k}\lambda_{\overline{f},g\delta_{n,n}}^{X,Y}\left(\partial_{k}^{X}(x)\otimes y\right)\right) \\
 & = (X\boxtimes_{R}\theta)_{n-1}\left(d_{n,n}^{\DK(X)}\left(\lambda_{f}^{X}(x)\right) \otimes d_{n,n}^{\DK(Y)}\left(\lambda_{g}^{Y}(y)\right)\right)
\end{split}
\end{equation*}

Case II: $f\delta_{n,i}$ is surjective but $g\delta_{n,i}$ is not surjective

This case is similar to Case I.

Case III: $f\delta_{n,i}$ and $g\delta_{n,i}$ are both surjective

In this case, we have:
\begin{equation*}
\begin{split}
 d_{n,i}^{\DK(X)}\left(\lambda_{f}^{X}(x)\right) \otimes d_{n,i}^{\DK(Y')}\left(\lambda_{g}^{Y'}\left(\theta_{l}(y)\right)\right) & = \lambda_{f\delta_{n,i},g\delta_{n,i}}^{X,Y'}\left(x \otimes \theta_{l}(y)\right) = (X\boxtimes_{R}\theta)_{n-1}\left(\lambda_{f\delta_{n,i},g\delta_{n,i}}^{X,Y}\left(x \otimes y\right)\right) \\
 & = (X\boxtimes_{R}\theta)_{n-1}\left(d_{n,i}^{\DK(X)}\left(\lambda_{f}^{X}(x)\right) \otimes d_{n,i}^{\DK(Y)}\left(\lambda_{g}^{Y}(y)\right)\right)
\end{split}
\end{equation*}

Case IV: $f\delta_{n,i}$ and $g\delta_{n,i}$ are neither surjective

If $0\leq i\leq n-1$, then $d_{n,i}^{\DK(X)}\left(\lambda_{f}^{X}(x)\right) = d_{n,i}^{\DK(Y)}\left(\lambda_{g}^{Y}(y)\right) = d_{n,i}^{\DK(Y')}\left(\lambda_{g}^{Y'}\left(\theta_{l}(y)\right)\right)=0$, so we have:
$$d_{n,i}^{\DK(X)}\left(\lambda_{f}^{X}(x)\right) \otimes d_{n,i}^{\DK(Y')}\left(\lambda_{g}^{Y'}\left(\theta_{l}(y)\right)\right) = 0 = d_{n,i}^{\DK(X)}\left(\lambda_{f}^{X}(x)\right) \otimes d_{n,i}^{\DK(Y)}\left(\lambda_{g}^{Y}(y)\right)$$
If $i=n$, then using the fact that $\theta:Y\rightarrow Y'$ is a morphism of $R$-complexes, we get:
\begin{equation*}
\begin{split}
 d_{n,n}^{\DK(X)}\left(\lambda_{f}^{X}(x)\right) \otimes d_{n,n}^{\DK(Y')}\left(\lambda_{g}^{Y'}\left(\theta_{l}(y)\right)\right) & = (-1)^{k+l}\lambda_{\overline{f},\overline{g}}^{X,Y'}\left(\partial_{k}^{X}(x)\otimes \partial_{l}^{Y'}\left(\theta_{l}(y)\right)\right) \\
 & = (-1)^{k+l}\lambda_{\overline{f},\overline{g}}^{X,Y'}\left(\partial_{k}^{X}(x)\otimes \theta_{l-1}\left(\partial_{l}^{Y}(y)\right)\right) \\
 & = (X\boxtimes_{R}\theta)_{n-1}\left((-1)^{k+l}\lambda_{\overline{f},\overline{g}}^{X,Y}\left(\partial_{k}^{X}(x)\otimes \partial_{l}^{Y}(y)\right)\right) \\
 & = (X\boxtimes_{R}\theta)_{n-1}\left(d_{n,n}^{\DK(X)}\left(\lambda_{f}^{X}(x)\right) \otimes d_{n,n}^{\DK(Y)}\left(\lambda_{g}^{Y}(y)\right)\right)
\end{split}
\end{equation*}

\noindent
We observe from the above analysis that in any case we have:
$$\partial_{n}^{X\boxtimes_{R}Y'}\left((X\boxtimes_{R}\theta)_{n}\left(\lambda_{f,g}^{X,Y}(x\otimes y)\right)\right) = (X\boxtimes_{R}\theta)_{n-1}\left(\partial_{n}^{X\boxtimes_{R}Y}\left(\lambda_{f,g}^{X,Y}(x\otimes y)\right)\right)$$
Therefore, it follows that $\partial_{n}^{X\boxtimes_{R}Y'}(X\boxtimes_{R}\theta)_{n} = (X\boxtimes_{R}\theta)_{n-1}\partial_{n}^{X\boxtimes_{R}Y}$. As a result, $X\boxtimes_{R}\theta:X\boxtimes_{R}Y \rightarrow X\boxtimes_{R}Y'$ is a morphism of $\mathbb{Z}$-complexes. Finally, we note that $X\boxtimes_{R}-:\mathcal{C}_{\geq 0}(R)\rightarrow \mathcal{C}_{\geq 0}(\mathbb{Z})$ preserves the composition of morphisms as well as the identity morphisms, so it defines a covariant functor.

(ii): Similar to (i).
\end{proof}

We have the following special cases of the shuffle product.

\begin{proposition} \label{5.5}
Let $R$ be a ring, $X$ a connective $R^{\op}$-complex, $Y$ a connective $R$-complex, $M$ a right $R$-module, and $N$ a left $R$-module. Then the following assertions hold:
\begin{enumerate}
\item[(i)] $X\boxtimes_{R}N=X\otimes_{R}N$.
\item[(ii)] $M\boxtimes_{R}Y=M\otimes_{R}Y$.
\end{enumerate}
\end{proposition}

\begin{proof}
(i): Let $n\geq 0$, and let $\mathcal{S}_{n}$ be as in Definition \ref{5.2}. Then we have $(X\boxtimes_{R}N)_{n}=\bigoplus_{(f,g)\in \mathcal{S}_{n}}(X_{k}\otimes_{R}N_{l})$. But $N_{l}=0$ for every $l\neq 0$, so we should only consider those terms in the latter direct sum that correspond to $l=0$. However, if $(f,g)\in \mathcal{S}_{n}$ and $l=0$, then $g=0$. But $\phi_{f,g}$ is injective, so $f$ must be injective. As $f$ is already surjective, we conclude that $k=n$ and $f=1^{[n]}$. As a result, we only have one possibly non-zero term, i.e. $(X\boxtimes_{R}N)_{n}=X_{n}\otimes_{R}N=(X\otimes_{R}N)_{n}$. On the other hand, since $1^{[n]}\delta_{n,i}=\delta_{n,i}$ is not surjective, Corollary \ref{4.12} implies that $d_{n,i}^{\DK(X)}\left(\lambda_{1^{[n]}}^{X}(x)\right)=0$ for every $0\leq i\leq n-1$ and $x\in X_{n}$. Furthermore, $\DK(N)$ is a constant simplicial module, so $d_{n,i}^{\DK(N)}=1^{N}$ for every $0\leq i\leq n$. Therefore, using Corollary \ref{4.12}, we have for every $x\in X_{n}$ and $y\in N$:
\small
\begin{equation*}
\begin{split}
\partial_{n}^{X\boxtimes_{R}N}(x\otimes y) & = \textstyle\sum_{i=0}^{n}(-1)^{i}d_{n,i}^{\DK(X)}\left(\lambda_{1^{[n]}}^{X}(x)\right)\otimes d_{n,i}^{\DK(N)}\left(\lambda_{0}^{N}(y)\right) = (-1)^{n}d_{n,n}^{\DK(X)}\left(\lambda_{1^{[n]}}^{X}(x)\right)\otimes d_{n,n}^{\DK(N)}\left(\lambda_{0}^{N}(y)\right) \\
 & = (-1)^{n}(-1)^{n}\lambda_{1^{[n-1]}}^{X}\left(\partial_{n}^{X}(x)\right)\otimes_{R}\lambda_{0}^{N}(y) = \partial_{n}^{X}(x)\otimes y = \partial_{n}^{X\otimes_{R}N}(x\otimes y) \\
\end{split}
\end{equation*}
\normalsize
It follows that $\partial_{n}^{X\boxtimes_{R}N}=\partial_{n}^{X\otimes_{R}N}$. Thus $X\boxtimes_{R}N=X\otimes_{R}N$.

(ii): Similar to (i).
\end{proof}

The next proposition describes the shuffle product in terms of the normalization and Dold-Kan functors. This in turn provides some useful corollaries.

\begin{proposition} \label{5.6}
Let $R$ be a ring, $X$ a connective $R^{\op}$-complex, and $Y$ a connective $R$-complex. Then the following natural isomorphism of $\mathbb{Z}$-complexes holds:
$$X\boxtimes_{R}Y \cong \Nor\left(\DK(X)\otimes_{R}\DK(Y)\right)$$
\end{proposition}

\begin{proof}
Let $n\geq 0$, and let $\mathcal{S}_{n}$ be as in Definition \ref{5.2}. As in the proof of Proposition \ref{5.3}, we can write:
\begin{equation} \label{eq:5.5.1}
C\left(\DK(X)\otimes_{R}\DK(Y)\right)_{n} = (X\boxtimes_{R} Y)_{n} \oplus \left(\textstyle\bigoplus_{(f,g)\notin \mathcal{S}_{n}}(X_{k}\otimes_{R}Y_{l})\right)
\end{equation}

Let $f\in \Mor_{\Delta}\left([n],[k]\right)$ and $g\in \Mor_{\Delta}\left([n],[l]\right)$ be surjective. Let $0\leq j_{1}< \cdots <j_{n-k}<n$ be the elements of $[n]$ for which $f(j)=f(j+1)$ so that $f=\sigma_{k,j_{1}} \cdots \sigma_{n-1,j_{n-k}}$, and $0\leq i_{1}< \cdots <i_{n-l}<n$ the elements of $[n]$ for which $g(i)=g(i+1)$ so that $g=\sigma_{l,i_{1}} \cdots \sigma_{n-1,i_{n-l}}$. If $(f,g)\notin \mathcal{S}_{n}$, then $\phi_{f,g}$ is not injective, so there is an element $m\in \{i_{1},...,i_{n-l}\} \cap \{j_{1},...,j_{n-k}\}$. This implies that $f=\sigma_{k,j_{1}} \cdots \sigma_{p,m} \cdots \sigma_{n-1,j_{n-k}}$ and $g=\sigma_{l,i_{1}} \cdots \sigma_{q,m} \cdots \sigma_{n-1,i_{n-l}}$ where $k\leq p\leq n-1$ and $l\leq q\leq n-1$. Using the face-degeneracy relations of Remark \ref{4.2}, we can successively swap the face map $\sigma_{p,m}$ with the terms appearing on its right to get $f=\sigma_{k,j_{1}} \cdots \sigma_{n-2,j_{n-k}-1}\sigma_{n-1,m}$. Similarly, we get $g=\sigma_{l,i_{1}} \cdots \sigma_{n-2,i_{n-l}-1}\sigma_{n-1,m}$. Applying the contravariant functor $\DK(X)$, we get $\DK(X)(f)=s_{n-1,m}^{\DK(X)}s_{n-2,j_{n-k}-1}^{\DK(X)} \cdots s_{k,j_{1}}^{\DK(X)}$ and $\DK(Y)(g)=s_{n-1,m}^{\DK(Y)}s_{n-2,i_{n-l}-1}^{\DK(Y)} \cdots s_{l,i_{1}}^{\DK(Y)}$. Moreover, since $f$ and $g$ are surjective, we can invoke Proposition \ref{4.11} to write: $$\lambda_{f}^{X}=\DK(X)(f)\lambda_{1^{[k]}}^{X}=s_{n-1,m}^{\DK(X)}s_{n-2,j_{n-k}-1}^{\DK(X)} \cdots s_{k,j_{1}}^{\DK(X)}\lambda_{1^{[k]}}^{X}$$
and
$$\lambda_{g}^{Y}=\DK(Y)(g)\lambda_{1^{[l]}}^{Y}=s_{n-1,m}^{\DK(Y)}s_{n-2,i_{n-l}-1}^{\DK(Y)} \cdots s_{l,i_{1}}^{\DK(Y)}\lambda_{1^{[l]}}^{Y}$$
Therefore, we obtain:
\begin{equation*}
\begin{split}
 \lambda_{f,g}^{X,Y} & = \lambda_{f}^{X}\otimes_{R}\lambda_{g}^{Y} \\
 & = s_{n-1,m}^{\DK(X)}s_{n-2,j_{n-k}-1}^{\DK(X)} \cdots s_{k,j_{1}}^{\DK(X)}\lambda_{1^{[k]}}^{X} \otimes_{R} s_{n-1,m}^{\DK(Y)}s_{n-2,i_{n-l}-1}^{\DK(Y)} \cdots s_{l,i_{1}}^{\DK(Y)}\lambda_{1^{[l]}}^{Y} \\
 & = \left(s_{n-1,m}^{\DK(X)}\otimes_{R}s_{n-1,m}^{\DK(Y)}\right)\left(s_{n-2,j_{n-k}-1}^{\DK(X)}\cdots s_{k,j_{1}}^{\DK(X)}\lambda_{1^{[k]}}^{X} \otimes_{R} s_{n-2,i_{n-l}-1}^{\DK(Y)}\cdots s_{l,i_{1}}^{\DK(Y)}\lambda_{1^{[l]}}^{Y}\right) \\
 & = \left(s_{n-1,m}^{\DK(X)\otimes_{R}\DK(Y)}\right)\left(s_{n-2,j_{n-k}-1}^{\DK(X)}\cdots s_{k,j_{1}}^{\DK(X)}\lambda_{1^{[k]}}^{X} \otimes_{R} s_{n-2,i_{n-l}-1}^{\DK(Y)}\cdots s_{l,i_{1}}^{\DK(Y)}\lambda_{1^{[l]}}^{Y}\right)
\end{split}
\end{equation*}

\noindent
This shows that:
$$\im\left(\lambda_{f,g}^{X,Y}\right) \subseteq \im\left(s_{n-1,m}^{\DK(X)\otimes_{R}\DK(Y)}\right) \subseteq \textstyle\sum_{i=0}^{n-1}\im\left(s_{n-1,i}^{\DK(X)\otimes_{R}\DK(Y)}\right)$$
Thus $\lambda_{f,g}^{X,Y}(x\otimes y)\in \sum_{i=0}^{n-1}\im\left(s_{n-1,i}^{\DK(X)\otimes_{R}\DK(Y)}\right)$ for every $x\in X_{k}$ and $y\in Y_{l}$. It follows that:
\begin{equation*}
\textstyle\bigoplus_{(f,g)\notin \mathcal{S}_{n}}(X_{k}\otimes_{R}Y_{l})\subseteq \textstyle\sum_{i=0}^{n-1}\im\left(s_{n-1,i}^{\DK(X)\otimes_{R}\DK(Y)}\right)
\end{equation*}

On the other hand, let $f'\in \Mor_{\Delta}\left([n-1],[k]\right)$ and $g'\in \Mor_{\Delta}\left([n-1],[l]\right)$ be surjective. Using Corollary \ref{4.13}, we have for every $0\leq i\leq n-1$, $x\in X_{k}$, and $y\in Y_{l}$:
\begin{equation*}
\begin{split}
 s_{n-1,i}^{\DK(X)\otimes_{R}\DK(Y)}\left(\lambda_{f',g'}^{X,Y}(x\otimes y)\right) & = \left(s_{n-1,i}^{\DK(X)}\otimes_{R}s_{n-1,i}^{\DK(Y)}\right)\left(\left(\lambda_{f'}^{X}\otimes_{R}\lambda_{g'}^{Y}\right)(x\otimes y)\right) \\
 & = \left(s_{n-1,i}^{\DK(X)}\lambda_{f'}^{X}\otimes_{R}s_{n-1,i}^{\DK(Y)}\lambda_{g'}^{Y}\right)\left(x\otimes y\right) \\
 & = \left(\lambda_{f'\sigma_{n-1,i}}^{X}\otimes_{R}\lambda_{g'\sigma_{n-1,i}}^{Y}\right)\left(x\otimes y\right) \\
 & = \lambda_{f'\sigma_{n-1,i},g'\sigma_{n-1,i}}^{X,Y}\left(x\otimes y\right)
\end{split}
\end{equation*}

\noindent
Set $f=f'\sigma_{n-1,i}\in \Mor_{\Delta}\left([n],[k]\right)$ and $g=g'\sigma_{n-1,i} \in \Mor_{\Delta}\left([n],[l]\right)$. As $\sigma_{n-1,i}(i)=i=\sigma_{n-1,i}(i+1)$, we see that $\phi_{f,g}(i)=\left(f(i),g(i)\right)=\left(f(i+1),g(i+1)\right)=\phi_{f,g}(i+1)$. Thus $\phi_{f,g}$ is not injective, so $(f,g) \notin \mathcal{S}_{n}$. It follows that:
\begin{equation*}
\textstyle\sum_{i=0}^{n-1}\im\left(s_{n-1,i}^{\DK(X)\otimes_{R}\DK(Y)}\right) \subseteq \textstyle\bigoplus_{(f,g)\notin \mathcal{S}_{n}}(X_{k}\otimes_{R}Y_{l})
\end{equation*}

Therefore, we have:
\begin{equation} \label{eq:5.5.2}
\textstyle\bigoplus_{(f,g)\notin \mathcal{S}_{n}}(X_{k}\otimes_{R}Y_{l}) = \textstyle\sum_{i=0}^{n-1}\im\left(s_{n-1,i}^{\DK(X)\otimes_{R}\DK(Y)}\right) = D\left(\DK(X)\otimes_{R}\DK(Y)\right)_{n}
\end{equation}

\noindent
Hence we obtain from \eqref{eq:5.5.1} and \eqref{eq:5.5.2} that:
$$C\left(\DK(X)\otimes_{R}\DK(Y)\right)_{n} = (X\boxtimes_{R} Y)_{n} \oplus D\left(\DK(X)\otimes_{R}\DK(Y)\right)_{n}$$
In light of the proof of Proposition \ref{5.3} and Remark \ref{4.6}, we observe that $X\boxtimes_{R} Y$ and $D\left(\DK(X)\otimes_{R}\DK(Y)\right)$ are subcomplexes of $C\left(\DK(X)\otimes_{R}\DK(Y)\right)$, so we deduce that:
\begin{equation} \label{eq:5.5.3}
C\left(\DK(X)\otimes_{R}\DK(Y)\right) = (X\boxtimes_{R} Y) \oplus D\left(\DK(X)\otimes_{R}\DK(Y)\right)
\end{equation}

On the other hand, we know from Remark \ref{4.6} that:
\begin{equation} \label{eq:5.5.4}
C\left(\DK(X)\otimes_{R}\DK(Y)\right) = \Nor\left(\DK(X)\otimes_{R}\DK(Y)\right) \oplus D\left(\DK(X)\otimes_{R}\DK(Y)\right)
\end{equation}

\noindent
Therefore, it follows from \eqref{eq:5.5.3} and \eqref{eq:5.5.4} that $X\boxtimes_{R}Y \cong \Nor\left(\DK(X)\otimes_{R}\DK(Y)\right)$.
\end{proof}

\begin{corollary} \label{5.7}
Let $R$ be a ring, $M$ a simplicial right $R$-module, and $N$ a simplicial left $R$-module. Then the following natural isomorphism of $\mathbb{Z}$-complexes holds:
$$\Nor(M\otimes_{R}N)\cong \Nor(M)\boxtimes_{R}\Nor(N)$$
\end{corollary}

\begin{proof}
In view of Theorem \ref{4.14} and Proposition \ref{5.6}, we have the following natural isomorphisms:
$$\Nor(M\otimes_{R}N)\cong \Nor\left(\DK\left(\Nor(M)\right)\otimes_{R}\DK\left(\Nor(N)\right)\right)\cong \Nor(M)\boxtimes_{R}\Nor(N)$$
\end{proof}

\begin{corollary} \label{5.8}
Let $R$ be a ring, $M$ a simplicial right $R$-module, $N$ a simplicial left $R$-module, $K$ a right $R$-module, and $L$ a left $R$-module. Then the following natural isomorphisms of $\mathbb{Z}$-complexes hold:
\begin{enumerate}
\item[(i)] $\Nor(M\otimes_{R}L) \cong \Nor(M)\otimes_{R}L$.
\item[(ii)] $\Nor(K\otimes_{R}N) \cong K\otimes_{R}\Nor(N)$.
\end{enumerate}
\end{corollary}

\begin{proof}
(i): If we consider $L$ as a constant simplicial left $R$-module, then $\Nor(L)=L$ is the $R$-complex with only $L$ concentrated in degree zero. Therefore, in light of Corollary \ref{5.7} and Proposition \ref{5.5}, we get the following natural isomorphism:
$$\Nor(M\otimes_{R}L) \cong \Nor(M)\boxtimes_{R}\Nor(L) = \Nor(M)\boxtimes_{R}L = \Nor(M)\otimes_{R}L$$

(ii): Similar to (i).
\end{proof}

Now we are ready to establish a variant of the Eilenberg-Zilber Theorem; see \cite[Chapter 8, Section 8.5]{We} for the original version.

\begin{theorem} \label{5.9}
Let $R$ be a ring, $X$ a connective $R^{\op}$-complex, and $Y$ a connective $R$-complex. Then there exists a natural quasi-isomorphism $\nabla^{X,Y}:X\otimes_{R}Y\rightarrow X\boxtimes_{R}Y$ of $\mathbb{Z}$-complexes.
\end{theorem}

\begin{proof}
Let $n\geq 0$. Then $(X\otimes_{R}Y)_{n}=\bigoplus_{k+l=n}(X_{k}\otimes_{R}Y_{l})$. For any $0\leq k,l\leq n$, set:
$$\mathcal{S}_{n}^{k,l}= \left\{(f,g) \in \Mor_{\Delta}([n],[k])\times \Mor_{\Delta}([n],[l]) \suchthat \begin{tabular}{ccc}
$f$ and $g$ are surjective, and $\phi_{f,g}:[n]\rightarrow [k]\times[l]$ given \\
by $\phi_{f,g}(i)=(f(i),g(i))$ for every $i\in [n]$ is injective
  \end{tabular} \right\}$$
Moreover, set $\mathcal{S}_{n}=\bigcup_{0\leq k,l\leq n}\mathcal{S}_{n}^{k,l}$. Then $(X\boxtimes_{R}Y)_{n}=\bigoplus_{(f,g)\in \mathcal{S}_{n}}(X_{k}\otimes_{R}Y_{l})$. Now let $k,l\geq 0$ be such that $k+l=n$. Let $\iota_{k,l}^{X,Y}:X_{k}\otimes_{R}Y_{l} \rightarrow (X\otimes_{R}Y)_{n}$ be the canonical injection. Also, let $\lambda_{f}^{X}:X_{k}\rightarrow \DK(X)([n])$ be the canonical injection corresponding to the index $f\in \Mor_{\Delta}\left([n],[k]\right)$, $\lambda_{g}^{Y}:Y_{l}\rightarrow \DK(Y)([n])$ the canonical injection corresponding to the index $g\in \Mor_{\Delta}\left([n],[l]\right)$, and $\lambda_{f,g}^{X,Y}:X_{k}\otimes_{R}Y_{l}\rightarrow (X\boxtimes_{R}Y)_{n}$ the canonical injection corresponding to the indices $f\in \Mor_{\Delta}\left([n],[k]\right)$ and $g\in \Mor_{\Delta}\left([n],[l]\right)$. Moreover, let $\nu_{f,g}\in \Sh(k,l)$ denote the shuffle determined by a pair $(f,g)\in \mathcal{S}_{n}^{k,l}$ as in Remark \ref{5.1}. Define a map $\psi_{k,l}:X_{k}\times Y_{l}\rightarrow (X\boxtimes_{R}Y)_{n}$ by setting
$$\psi_{k,l}(x,y)=\textstyle\sum_{(f,g)\in\mathcal{S}_{n}^{k,l}}\sgn(\nu_{f,g})\lambda_{f,g}^{X,Y}(x\otimes y)$$
for every $(x,y)\in X_{k}\times Y_{l}$. It is clear that $\psi_{k,l}$ is an $R$-balanced map. Therefore, the universal property of tensor product yields a unique $\mathbb{Z}$-homomorphism $\overline{\psi_{k,l}}:X_{k}\otimes_{R} Y_{l}\rightarrow (X\boxtimes_{R}Y)_{n}$ that makes the following diagram commutative:
\begin{equation*}
  \begin{tikzcd}
  X_{k}\times Y_{l} \arrow{r}{\psi_{k,l}} \arrow{d}[swap]{\omega_{X_{k},Y_{l}}}
  & (X\boxtimes_{R}Y)_{n}
  \\
  X_{k}\otimes_{R} Y_{l} \arrow{ru}[swap]{\overline{\psi_{k,l}}}
\end{tikzcd}
\end{equation*}

\noindent
Then we have
$$\overline{\psi_{k,l}}(x\otimes y)=\textstyle\sum_{(f,g)\in\mathcal{S}_{n}^{k,l}}\sgn(\nu_{f,g})\lambda_{f,g}^{X,Y}(x\otimes y)$$
for every $x\in X_{k}$ and $y\in Y_{l}$. By the universal property of direct sum, there is a unique $\mathbb{Z}$-homomorphism $\nabla_{n}^{X,Y}:(X\otimes_{R}Y)_{n}\rightarrow (X\boxtimes_{R}Y)_{n}$ that makes the following diagram commutative for every $0\leq k,l\leq n$ with $k+l=n$:
\begin{equation*}
  \begin{tikzcd}
  X_{k}\otimes_{R} Y_{l} \arrow{r}{\iota_{k,l}^{X,Y}} \arrow{d}[swap]{\overline{\psi_{k,l}}}
  & (X\otimes_{R}Y)_{n} \arrow{dl}{\nabla_{n}^{X,Y}}
  \\
  (X\boxtimes_{R}Y)_{n}
\end{tikzcd}
\end{equation*}

\noindent
Then we have
$$\nabla_{n}^{X,Y}\left(\textstyle\sum_{k+l=n}\iota_{k,l}^{X,Y}(x_{k}\otimes y_{l})\right)= \textstyle\sum_{k+l=n}\textstyle\sum_{(f,g)\in\mathcal{S}_{n}^{k,l}}\sgn(\nu_{f,g})\lambda_{f,g}^{X,Y}(x_{k}\otimes y_{l})$$
for every $\sum_{k+l=n}\iota_{k,l}^{X,Y}(x_{k}\otimes y_{l})\in (X\otimes_{R}Y)_{n}$.

Now set $\nabla^{X,Y}=\left(\nabla_{n}^{X,Y}\right)_{n\geq 0}:X\otimes_{R}Y \rightarrow X\boxtimes_{R}Y$. We show that $\nabla^{X,Y}$ is a morphism of $\mathbb{Z}$-complexes. Let $n\geq 0$. Then the following diagram is commutative:
\begin{equation*}
\begin{tikzcd}[column sep=3.5em,row sep=2em]
  (X\otimes_{R} Y)_{n} \arrow{r}{\partial_{n}^{X\otimes_{R}Y}} \arrow{d}[swap]{\nabla_{n}^{X,Y}}
  & (X\otimes_{R} Y)_{n-1} \arrow{d}{\nabla_{n-1}^{X,Y}}
  \\
  (X\boxtimes_{R}Y)_{n} \arrow{r}{\partial_{n}^{X\boxtimes_{R}Y}}
  & (X\boxtimes_{R}Y)_{n-1}
\end{tikzcd}
\end{equation*}

\noindent
Indeed, we have for every $x\in X_{k}$ and $y\in Y_{l}$ with $k+l=n$:
\begin{equation*}
\begin{split}
 \partial_{n}^{X\boxtimes_{R}Y}\left(\nabla_{n}^{X,Y}\left(\iota_{k,l}^{X,Y}(x\otimes y)\right)\right) & = \partial_{n}^{X\boxtimes_{R}Y}\left(\textstyle\sum_{(f,g)\in\mathcal{S}_{n}^{k,l}}\sgn(\nu_{f,g})\lambda_{f,g}^{X,Y}(x\otimes y)\right) \\
 & = \textstyle\sum_{(f,g)\in\mathcal{S}_{n}^{k,l}}\sgn(\nu_{f,g})\partial_{n}^{X\boxtimes_{R}Y}\left(\lambda_{f,g}^{X,Y}(x\otimes y)\right) \\
 & = \textstyle\sum_{i=0}^{n}(-1)^{i}\textstyle\sum_{(f,g)\in\mathcal{S}_{n}^{k,l}}\sgn(\nu_{f,g}) d_{n,i}^{\DK(X)}\left(\lambda_{f}^{X}(x)\right)\otimes d_{n,i}^{\DK(Y)}\left(\lambda_{g}^{Y}(y)\right) \\
\end{split}
\end{equation*}

Let $0\leq i\leq n$ and $(f,g)\in \mathcal{S}_{n}^{k,l}$, and consider the following cases:

Case I: $f\delta_{n,i}$ is not surjective but $g\delta_{n,i}$ is surjective

If $0\leq i\leq n-1$, then by Corollary \ref{4.12}, $d_{n,i}^{\DK(X)}\left(\lambda_{f}^{X}(x)\right)=0$, so $d_{n,i}^{\DK(X)}\left(\lambda_{f}^{X}(x)\right)\otimes d_{n,i}^{\DK(Y)}\left(\lambda_{g}^{Y}(y)\right)=0$. Hence we can assume that $i=n$. In this case, by Corollary \ref{4.12}, $f$ induces a surjective map $\overline{f} \in \Mor_{\Delta}\left([n-1],[k-1]\right)$ given by $\overline{f}(i)=f(i)$ for every $i\in [n-1]$. Moreover, we note that $g\delta_{n,n}=g\mid_{[n-1]}:[n-1]\rightarrow [l]$. If $0\leq j_{1}<\cdots<j_{l}<n$ are the elements of $[n]$ for which $f(j)=f(j+1)$, and $0\leq i_{1}<\cdots<i_{k}<n$ are the elements of $[n]$ for which $g(i)=g(i+1)$, then we have:
\begin{equation*}
\nu_{f,g} =
\begin{pmatrix}
1 & \cdots & k & k+1 & \cdots & n \\
i_{1}+1 & \cdots & i_{k}+1 & j_{1}+1 & \cdots & j_{l}+1
\end{pmatrix}
\end{equation*}

\noindent
Set:
$$A=\left\{(i_{\alpha}+1,j_{\beta}+1)\mid 1\leq \alpha \leq k \textrm{, } 1\leq \beta \leq l\textrm{, and } i_{\alpha}>j_{\beta}\right\}$$
If $|A|=r$, then $\sgn(\nu_{f,g})=(-1)^{r}$. As $f$ is surjective but $f\delta_{n,n}$ is not surjective, we see that $j_{l}<n-1$ since otherwise, $f(n-1)=f(n)=k$ which would imply that $f\delta_{n,n}=f\mid_{[n-1]}:[n-1]\rightarrow [k]$ is surjective. Also, since $g$ and $g\delta_{n,n}$ are both surjective, we notice that $i_{k}=n-1$ as $g(n-1)=l=g(n)$. As a result, $0\leq j_{1}<\cdots<j_{l}<n-1$ are the elements of $[n-1]$ for which $\overline{f}(j)=\overline{f}(j+1)$, and $0\leq i_{1}<\cdots<i_{k-1}<n-1$ are the elements of $[n-1]$ for which $(g\delta_{n,n})(i)=(g\delta_{n,n})(i+1)$, so we have:
\begin{equation*}
\nu_{\overline{f},g\delta_{n,n}} =
\begin{pmatrix}
1 & \cdots & k-1 & k & \cdots & n-1 \\
i_{1}+1 & \cdots & i_{k-1}+1 & j_{1}+1 & \cdots & j_{l}+1
\end{pmatrix}
\end{equation*}

\noindent
Set:
$$B=\left\{(i_{\alpha}+1,j_{\beta}+1)\mid 1\leq \alpha \leq k-1 \textrm{, } 1\leq \beta \leq l\textrm{, and } i_{\alpha}>j_{\beta}\right\}$$
If $|B|=s$, then $\sgn\left(\nu_{\overline{f},g\delta_{n,n}}\right)=(-1)^{s}$. Since $i_{k}=n-1>j_{l}>\cdots>j_{1}$, we conclude that
$$A=B \cupdot \{(i_{k}+1,j_{\beta}+1)\mid 1\leq \beta \leq l\},$$
which in turn implies that $r=s+l$. Therefore, we have $\sgn(\nu_{f,g})=(-1)^{r}=(-1)^{l}(-1)^{s}=(-1)^{l}\sgn\left(\nu_{\overline{f},g\delta_{n,n}}\right)$. Thus using Corollary \ref{4.12}, we get:
\begin{equation*}
\begin{split}
(-1)^{n}\sgn(\nu_{f,g})d_{n,n}^{\DK(X)}\left(\lambda_{f}^{X}(x)\right)\otimes d_{n,n}^{\DK(Y)}\left(\lambda_{g}^{Y}(y)\right) & = (-1)^{n}(-1)^{l}\sgn\left(\nu_{\overline{f},g\delta_{n,n}}\right)(-1)^{k}\lambda_{\overline{f}}^{X}\left(\partial_{k}^{X}(x)\right)\otimes \lambda_{g\delta_{n,n}}^{Y}(y) \\
 & = \sgn\left(\nu_{\overline{f},g\delta_{n,n}}\right)\lambda_{\overline{f},g\delta_{n,n}}^{X,Y}\left(\partial_{k}^{X}(x)\otimes y\right)
\end{split}
\end{equation*}

Case II: $f\delta_{n,i}$ is surjective but $g\delta_{n,i}$ is not surjective

If $0\leq i\leq n-1$, then by Corollary \ref{4.12}, $d_{n,i}^{\DK(Y)}\left(\lambda_{g}^{Y}(y)\right)=0$, so $d_{n,i}^{\DK(X)}\left(\lambda_{f}^{X}(x)\right)\otimes d_{n,i}^{\DK(Y)}\left(\lambda_{g}^{Y}(y)\right)=0$. Hence we can assume that $i=n$. In this case, we have $f\delta_{n,n}=f\mid_{[n-1]}:[n-1]\rightarrow [k]$. Moreover, by Corollary \ref{4.12}, $g$ induces a surjective map $\overline{g}\in \Mor_{\Delta}\left([n-1],[l-1]\right)$ given by $\overline{g}(i)=g(i)$ for every $i\in [n-1]$. If $0\leq j_{1}<\cdots<j_{l}<n$ are the elements of $[n]$ for which $f(j)=f(j+1)$, and $0\leq i_{1}<\cdots<i_{k}<n$ are the elements of $[n]$ for which $g(i)=g(i+1)$, then we have:
\begin{equation*}
\nu_{f,g} =
\begin{pmatrix}
1 & \cdots & k & k+1 & \cdots & n \\
i_{1}+1 & \cdots & i_{k}+1 & j_{1}+1 & \cdots & j_{l}+1
\end{pmatrix}
\end{equation*}

\noindent
Set:
$$A=\left\{(i_{\alpha}+1,j_{\beta}+1)\mid 1\leq \alpha \leq k \textrm{, } 1\leq \beta \leq l\textrm{, and } i_{\alpha}>j_{\beta}\right\}$$
If $|A|=r$, then $\sgn(\nu_{f,g})=(-1)^{r}$. As $f$ and $f\delta_{n,n}$ are both surjective, we notice that $j_{l}=n-1$ as $f(n-1)=k=f(n)$. Also, since $g$ is surjective but $g\delta_{n,n}$ is not surjective, we see that $i_{k}<n-1$ since otherwise, $g(n-1)=g(n)=l$ which would imply that $g\delta_{n,n}=g\mid_{[n-1]}:[n-1]\rightarrow [l]$ is surjective. As a result, $0\leq j_{1}<\cdots<j_{l-1}<n-1$ are the elements of $[n-1]$ for which $(f\delta_{n,n})(j)=(f\delta_{n,n})(j+1)$, and $0\leq i_{1}<\cdots<i_{k}<n-1$ are the elements of $[n-1]$ for which $\overline{g}(i)=\overline{g}(i+1)$, so we have:
\begin{equation*}
\nu_{f\delta_{n,n},\overline{g}} =
\begin{pmatrix}
1 & \cdots & k & k+1 & \cdots & n-1 \\
i_{1}+1 & \cdots & i_{k}+1 & j_{1}+1 & \cdots & j_{l-1}+1
\end{pmatrix}
\end{equation*}

\noindent
Set:
$$B=\left\{(i_{\alpha}+1,j_{\beta}+1)\mid 1\leq \alpha \leq k \textrm{, } 1\leq \beta \leq l-1 \textrm{, and } i_{\alpha}>j_{\beta}\right\}$$
If $|B|=s$, then $\sgn\left(\nu_{f\delta_{n,n},\overline{g}}\right)=(-1)^{s}$. Since $j_{l}=n-1>i_{k}>\cdots>i_{1}$, we deduce that $A=B$, which in turn implies that $r=s$. Therefore, we have $\sgn(\nu_{f,g})=(-1)^{r}=(-1)^{s}=\sgn\left(\nu_{f\delta_{n,n},\overline{g}}\right)$. Thus using Corollary \ref{4.12}, we get:
\begin{equation*}
\begin{split}
(-1)^{n}\sgn(\nu_{f,g})d_{n,n}^{\DK(X)}\left(\lambda_{f}^{X}(x)\right)\otimes d_{n,n}^{\DK(Y)}\left(\lambda_{g}^{Y}(y)\right) & = (-1)^{n}\sgn\left(\nu_{f\delta_{n,n},\overline{g}}\right)\lambda_{f\delta_{n,n}}^{X}(x)\otimes (-1)^{l}\lambda_{\overline{g}}^{Y}\left(\partial_{l}^{Y}(y)\right) \\
 & = (-1)^{k}\sgn\left(\nu_{f\delta_{n,n},\overline{g}}\right)\lambda_{f\delta_{n,n},\overline{g}}^{X,Y}\left(x\otimes \partial_{l}^{Y}(y)\right)
\end{split}
\end{equation*}

Case III: $f\delta_{n,i}$ and $g\delta_{n,i}$ are both surjective

Let $0\leq j_{1}<\cdots<j_{l}<n$ be the elements of $[n]$ for which $f(j)=f(j+1)$, and $0\leq i_{1}<\cdots<i_{k}<n$ the elements of $[n]$ for which $g(i)=g(i+1)$. If $i=0$, then $f\delta_{n,0}$ is surjective which means that if $0$ is omitted from the domain of $f$, then it will remain surjective. This implies that $f(0)=0=f(1)$, so $j_{1}=0$. Similarly, we see that $i_{1}=0$. This is a contradiction as $\{i_{1},...,i_{k}\} \cap \{j_{1},...,j_{l}\}=\emptyset$. If $i=n$, then $f\delta_{n,n}$ is surjective which means that if $n$ is omitted from the domain of $f$, then it will remain surjective. This yields that $f(n-1)=k=f(n)$, so $j_{l}=n-1$. Similarly, we see that $i_{k}=n-1$. This is a contradiction as $\{i_{1},...,i_{k}\} \cap \{j_{1},...,j_{l}\}=\emptyset$. Hence we can assume that $1\leq i\leq n-1$. By Corollary \ref{4.12}, we have:
$$\sgn(\nu_{f,g})d_{n,i}^{\DK(X)}\left(\lambda_{f}^{X}(x)\right)\otimes d_{n,i}^{\DK(Y)}\left(\lambda_{g}^{Y}(y)\right) = \sgn(\nu_{f,g})\lambda_{f\delta_{n,i}}^{X}(x)\otimes \lambda_{g\delta_{n,i}}^{Y}(y) = \sgn(\nu_{f,g})\lambda_{f\delta_{n,i},g\delta_{n,i}}^{X,Y}(x\otimes y)$$
Since $f\delta_{n,i}$ is surjective, we observe that if $i$ is omitted from the domain of $f$, then it will remain surjective. This means that $f(i-1)=f(i)$ or $f(i)=f(i+1)$, so $i-1\in \{j_{1},...,j_{l}\}$ or $i\in \{j_{1},...,j_{l}\}$. Similarly, we have $i-1\in \{i_{1},...,i_{k}\}$ or $i\in \{i_{1},...,i_{k}\}$. Since $\{i_{1},...,i_{k}\} \cap \{j_{1},...,j_{l}\}=\emptyset$, we should have either $i-1\in \{j_{1},...,j_{l}\}$ and $i\in \{i_{1},...,i_{k}\}$, or $i\in \{j_{1},...,j_{l}\}$ and $i-1\in \{i_{1},...,i_{k}\}$. Without loss of generality, suppose that $i-1\in \{j_{1},...,j_{l}\}$ and $i\in \{i_{1},...,i_{k}\}$. Define a map $f':[n]\rightarrow [k]$ by setting $f'(t)=f(t)$ for every $t\neq i$ and $f'(i)=f(i+1)$. Since $f'(i-1)=f(i-1)=f(i)<f(i+1)=f'(i)=f'(i+1)$, it follows that $f'$ is order-preserving. That is, $f'\in \Mor_{\Delta}\left([n],[k]\right)$. Moreover, $f'(i-1)=f(i-1)=f(i)$, so $f(i)\in \im(f')$, whence we see that $f'$ is surjective. Also, since $f$ and $f'$ have the same values except at $i$, we have $f\delta_{n,i}=f'\delta_{n,i}$. Similarly, define a map $g':[n]\rightarrow [l]$ by setting $g'(t)=g(t)$ for every $t\neq i$ and $g'(i)=g(i-1)$. Since $g'(i-1)=g'(i)=g(i-1)<g(i)=g(i+1)=g'(i+1)$, it follows that $g'$ is order-preserving. That is, $g'\in \Mor_{\Delta}\left([n],[l]\right)$. Moreover, $g'(i+1)=g(i+1)=g(i)$, so $g(i)\in \im(g')$, thereby we notice that $g'$ is surjective. Also, since $g$ and $g'$ have the same values except at $i$, we have $g\delta_{n,i}=g'\delta_{n,i}$. Now if $(f',g') \notin \mathcal{S}_{n}^{k,l}$, then $\phi_{f',g'}$ is not injective, so there is an $0\leq m\leq n-1$ such that $\phi_{f',g'}(m)=\phi_{f',g'}(m+1)$. But we note that:
\begin{equation*}
 \label{eqn:damage piecewise}
 \phi_{f',g'}(m)=\left(f'(m),g'(m)\right)=
 \begin{dcases}
  \left(f(m),g(m)\right)=\phi_{f,g}(m) & \textrm{if } m < i-1 \\
  \left(f(m),g(m)\right) & \textrm{if } m = i-1 \\
  \left(f(m+1),g(m-1)\right) & \textrm{if } m = i \\
  \left(f(m),g(m)\right)=\phi_{f,g}(m) & \textrm{if } m > i
 \end{dcases}
\end{equation*}

\noindent
and
\begin{equation*}
 \label{eqn:damage piecewise}
 \phi_{f',g'}(m+1)=\left(f'(m+1),g'(m+1)\right)=
 \begin{dcases}
  \left(f(m+1),g(m+1)\right)=\phi_{f,g}(m+1) & \textrm{if } m < i-1 \\
  \left(f(m+2),g(m)\right) & \textrm{if } m = i-1 \\
  \left(f(m+1),g(m+1)\right) & \textrm{if } m = i \\
  \left(f(m+1),g(m+1)\right)=\phi_{f,g}(m+1) & \textrm{if } m > i
 \end{dcases}
\end{equation*}

\noindent
Taking into account that $\phi_{f,g}$ is injective, we notice that in any case, $\phi_{f',g'}(m)\neq \phi_{f',g'}(m+1)$ which is a contradiction. Hence $(f',g') \in \mathcal{S}_{n}^{k,l}$. We finally have:
\begin{equation*}
\nu_{f,g} =
\begin{pmatrix}
1 & \cdots & p & \cdots & k & k+1 & \cdots & q & \cdots & n \\
i_{1}+1 & \cdots & i & \cdots & i_{k}+1 & j_{1}+1 & \cdots & i-1 & \cdots & j_{l}+1
\end{pmatrix}
\end{equation*}

\noindent
Furthermore, we have:
\begin{equation*}
\nu_{f',g'} =
\begin{pmatrix}
1 & \cdots & p & \cdots & k & k+1 & \cdots & q & \cdots & n \\
i_{1}+1 & \cdots & i-1 & \cdots & i_{k}+1 & j_{1}+1 & \cdots & i & \cdots & j_{l}+1
\end{pmatrix}
\end{equation*}

\noindent
It is clear from the above displays that $\sgn(\nu_{f',g'})=-\sgn(\nu_{f,g})$. As a result, we have:
$$\sgn(\nu_{f',g'})d_{n,i}^{\DK(X)}\left(\lambda_{f'}^{X}(x)\right)\otimes d_{n,i}^{\DK(Y)}\left(\lambda_{g'}^{Y}(y)\right) = \sgn(\nu_{f',g'})\lambda_{f'\delta_{n,i}}^{X}(x)\otimes \lambda_{g'\delta_{n,i}}^{Y}(y) = \sgn(\nu_{f',g'})\lambda_{f'\delta_{n,i},g'\delta_{n,i}}^{X,Y}(x\otimes y)=$$$$-sgn(\nu_{f,g})\lambda_{f\delta_{n,i},g\delta_{n,i}}^{X,Y}(x\otimes y)$$
Consequently, fixing $1\leq i\leq n-1$, for any pair $(f,g)\in \mathcal{S}_{n}^{k,l}$, there is a dual pair $(f',g')\in \mathcal{S}_{n}^{k,l}$ whose corresponding term cancels that of $(f,g)$. Therefore, the sum of the terms arising in this way vanishes.

Case IV: $f\delta_{n,i}$ and $g\delta_{n,i}$ are neither surjective

Let $0\leq j_{1}<\cdots<j_{l}<n$ be the elements of $[n]$ for which $f(j)=f(j+1)$, and $0\leq i_{1}<\cdots<i_{k}<n$ the elements of $[n]$ for which $g(i)=g(i+1)$. Since $f\delta_{n,i}$ is not surjective, we conclude that if $i$ is omitted from the domain of $f$, then it will no longer remain surjective. This implies that $f(i)$ is not assumed by any element of $[n]$ other than $i$. If $i=0$, then $f(0)\neq f(1)$, so $0 \notin \{j_{1},...,j_{l}\}$. Similarly, $0 \notin \{i_{1},...,i_{k}\}$. But then $\{i_{1},...,i_{k},j_{1},...,j_{l}\}\subseteq \{0,1,...,n-1\}\setminus \{0\}$ and $\{i_{1},...,i_{k}\} \cap \{j_{1},...,j_{l}\}=\emptyset$, so $n=k+l\leq n-1$ which is a contradiction. If $1\leq i\leq n-1$, then $f(i-1)\neq f(i)\neq f(i+1)$, so $i-1,i \notin \{j_{1},...,j_{l}\}$. Similarly, $i-1,i \notin \{i_{1},...,i_{k}\}$. But then $\{i_{1},...,i_{k},j_{1},...,j_{l}\}\subseteq \{0,1,...,n-1\}\setminus \{i-1,i\}$, so $n=k+l\leq n-2$ which is a contradiction. If $i=n$, then $f(n-1)\neq f(n)$, so $n-1 \notin \{j_{1},...,j_{l}\}$. Similarly, $n-1 \notin \{i_{1},...,i_{k}\}$. But then $\{i_{1},...,i_{k},j_{1},...,j_{l}\}\subseteq \{0,1,...,n-1\}\setminus \{n-1\}$, so $n=k+l\leq n-1$ which is a contradiction. Therefore, this case never happens.

Now in light of the above analysis, we have:
\small
$$\textstyle\sum_{i=0}^{n}(-1)^{i}\textstyle\sum_{(f,g)\in \mathcal{S}_{n}^{k,l}}\sgn(\nu_{f,g})d_{n,i}^{\DK(X)}\left(\lambda_{f}^{X}(x)\right)\otimes d_{n,i}^{\DK(Y)}\left(\lambda_{g}^{Y}(y)\right) = $$$$\textstyle\sum_{\left(\overline{f},g\delta_{n,n}\right)\in \mathcal{S}_{n-1}^{k-1,l}}\sgn\left(\nu_{\overline{f},g\delta_{n,n}}\right)\lambda_{\overline{f},g\delta_{n,n}}^{X,Y}\left(\partial_{k}^{X}(x)\otimes y\right) + (-1)^{k}\textstyle\sum_{\left(f\delta_{n,n},\overline{g}\right) \in \mathcal{S}_{n-1}^{k,l-1}}\sgn\left(\nu_{f\delta_{n,n},\overline{g}}\right)\lambda_{f\delta_{n,n},\overline{g}}^{X,Y}\left(x\otimes \partial_{l}^{Y}(y)\right)$$
\normalsize
On the other hand, we have:
\footnotesize
\begin{equation*}
\begin{split}
 \nabla_{n-1}^{X,Y}\left(\partial_{n}^{X\otimes_{R}Y}\left(\iota_{k,l}^{X,Y}(x\otimes y)\right)\right) & = \nabla_{n-1}^{X,Y}\left(\iota_{k-1,l}^{X,Y}\left(\partial_{k}^{X}(x)\otimes y\right)\right) + (-1)^{k}\nabla_{n-1}^{X,Y}\left(\iota_{k,l-1}^{X,Y}\left(x\otimes \partial_{l}^{Y}(y)\right)\right)\\
 & = \textstyle\sum_{(f,g)\in \mathcal{S}_{n-1}^{k-1,l}}\sgn(\nu_{f,g})\lambda_{f,g}^{X,Y}\left(\partial_{k}^{X}(x)\otimes y\right) + (-1)^{k} \textstyle\sum_{(f,g)\in \mathcal{S}_{n-1}^{k,l-1}} \sgn(\nu_{f,g})\lambda_{f,g}^{X,Y}\left(x\otimes \partial_{l}^{Y}(y)\right)
\end{split}
\end{equation*}
\normalsize

\noindent
After reindexing properly, we observe that $\partial_{n}^{X\boxtimes_{R}Y}\nabla_{n}^{X,Y}\iota_{k,l}^{X,Y}=\nabla_{n-1}^{X,Y}\partial_{n}^{X\otimes_{R}Y}\iota_{k,l}^{X,Y}$ for every $k,l\geq 0$ with $k+l=n$, so we infer that $\partial_{n}^{X\boxtimes_{R}Y}\nabla_{n}^{X,Y}=\nabla_{n-1}^{X,Y}\partial_{n}^{X\otimes_{R}Y}$. As a consequence, $\nabla^{X,Y}: X\otimes_{R}Y \rightarrow X\boxtimes_{R}Y$ is a morphism of $\mathbb{Z}$-complexes.

We next establish the naturality of $\nabla^{X,Y}$. To this end, let $\eta:X\rightarrow X'$ be a morphism of $R^{\op}$-complexes. Then the following diagram is commutative:
\begin{equation*}
\begin{tikzcd}[column sep=3em,row sep=2em]
  X\otimes_{R} Y \arrow{r}{\eta\otimes_{R}Y} \arrow{d}[swap]{\nabla^{X,Y}}
  & X'\otimes_{R}Y \arrow{d}{\nabla^{X',Y}}
  \\
  X\boxtimes_{R}Y \arrow{r}{\eta\boxtimes_{R}Y}
  & X'\boxtimes_{R}Y
\end{tikzcd}
\end{equation*}

\noindent
Indeed, we have for every $x\in X_{k}$ and $y\in Y_{l}$ with $k+l=n$:
\begin{equation*}
\begin{split}
 (\eta\boxtimes_{R}Y)_{n}\left(\nabla_{n}^{X,Y}\left(\iota_{k,l}^{X,Y}(x\otimes y)\right)\right) & = (\eta\boxtimes_{R}Y)_{n}\left(\textstyle\sum_{(f,g)\in \mathcal{S}_{n}^{k,l}}\sgn(\nu_{f,g})\lambda_{f,g}^{X,Y}(x\otimes y)\right) \\
 & = \textstyle\sum_{(f,g)\in \mathcal{S}_{n}^{k,l}}\sgn(\nu_{f,g})(\eta\boxtimes_{R}Y)_{n}\left(\lambda_{f,g}^{X,Y}(x\otimes y)\right) \\
 & = \textstyle\sum_{(f,g)\in \mathcal{S}_{n}^{k,l}}\sgn(\nu_{f,g})\lambda_{f,g}^{X',Y}\left(\eta_{k}(x)\otimes y\right) \\
 & = \nabla_{n}^{X',Y}\left(\iota_{k,l}^{X',Y}\left(\eta_{k}(x)\otimes y\right)\right) \\
 & = \nabla_{n}^{X',Y}\left((\eta\otimes_{R}Y)_{n}\left(\iota_{k,l}^{X,Y}(x\otimes y)\right)\right)
\end{split}
\end{equation*}

\noindent
As a result, $(\eta\boxtimes_{R}Y)_{n}\nabla_{n}^{X,Y}\iota_{k,l}^{X,Y} = \nabla_{n}^{X',Y}(\eta\otimes_{R}Y)_{n}\iota_{k,l}^{X,Y}$ for every $k,l\geq 0$ with $k+l=n$, so we infer that $(\eta\boxtimes_{R}Y)_{n}\nabla_{n}^{X,Y} = \nabla_{n}^{X',Y}(\eta\otimes_{R}Y)_{n}$ for every $n\geq 0$, whence $(\eta\boxtimes_{R}Y)\nabla^{X,Y} = \nabla^{X',Y}(\eta\otimes_{R}Y)$. If $\theta:Y\rightarrow Y'$ is a morphism of $R$-complexes, then a similar argument shows that the following diagram is commutative:
\begin{equation*}
\begin{tikzcd}[column sep=3em,row sep=2em]
  X\otimes_{R} Y \arrow{r}{X\otimes_{R}\theta} \arrow{d}[swap]{\nabla^{X,Y}}
  & X\otimes_{R}Y' \arrow{d}{\nabla^{X,Y'}}
  \\
  X\boxtimes_{R}Y \arrow{r}{X\boxtimes_{R}\theta}
  & X\boxtimes_{R}Y'
\end{tikzcd}
\end{equation*}

\noindent
As a consequence, $\nabla^{X,Y}$ is natural in $X$ and $Y$.

We finally show that $\nabla^{X,Y}$ is a quasi-isomorphism. The author has learned the ideas of the following delicate argument from \cite[proof of Theorem 2.5.7.14]{Lu}. We first consider the special case in which $X$ and $Y$ are left-bounded. Thus we assume that $X_{i}=0$ for every $i>n$, and $Y_{i}=0$ for every $i>m$. We argue by induction on $n$ to show that $\nabla^{X,Y}$ is a quasi-isomorphism. If $n=0$, then $X$ has only one non-zero term in degree $0$, so we can identify $X$ with the left $R$-module $X_{0}$. Consequently, $X\otimes_{R}Y=X\boxtimes_{R}Y$ by Proposition \ref{5.5}, and $\nabla_{i}^{X,Y}:(X\otimes_{R}Y)_{i}\rightarrow (X\boxtimes_{R}Y)_{i}$ is the identity map for every $i\geq 0$. Therefore, $\nabla^{X,Y}$ is a quasi-isomorphism. Now let $n\geq 1$, and assume the result for $n-1$. Consider the degreewise split short exact sequence
$$0\rightarrow X_{n-1\sqsubset}\rightarrow X\rightarrow \Sigma^{n}X_{n}\rightarrow 0$$
of $R$-complexes in which $\sqsubset$ designates the hard truncation from left; see \cite[A.1.14]{Ch}. Then we obtain the following commutative diagram with exact rows:
\begin{equation*}
\begin{tikzcd}
  0 \arrow{r} & X_{n-1\sqsubset}\otimes_{R}Y \arrow{r} \arrow{d}{\nabla^{X_{n-1\sqsubset},Y}}
  & X\otimes_{R}Y \arrow{r} \arrow{d}{\nabla^{X,Y}} & \Sigma^{n}X_{n}\otimes_{R}Y \arrow{r} \arrow{d}{\nabla^{\Sigma^{n}X_{n},Y}} & 0
  \\
  0 \arrow{r} & X_{n-1\sqsubset}\boxtimes_{R}Y \arrow{r} & X\boxtimes_{R}Y \arrow{r} \arrow{r} & \Sigma^{n}X_{n}\boxtimes_{R}Y \arrow{r} & 0
\end{tikzcd}
\end{equation*}

\noindent
By the induction hypothesis, $\nabla^{X_{n-1\sqsubset},Y}$ is a quasi-isomorphism. Therefore, the above diagram implies that $\nabla^{X,Y}$ is a quasi-isomorphism if and only if $\nabla^{\Sigma^{n}X_{n},Y}$ is a quasi-isomorphism. We note that $X_{n}\otimes_{R}\Nor\left(R^{(\Delta^{n})}\right)$ is a left-bounded $R$-complex with
$$\left(X_{n}\otimes_{R}\Nor\left(R^{(\Delta^{n})}\right)\right)_{n} = X_{n}\otimes_{R}\Nor\left(R^{(\Delta^{n})}\right)_{n} = X_{n}\otimes_{R}R \cong X_{n},$$
so by the above argument, $\nabla^{X_{n}\otimes_{R}\Nor\left(R^{(\Delta^{n})}\right),Y}$ is a quasi-isomorphism if and only if $\nabla^{\Sigma^{n}X_{n},Y}$ is a quasi-isomorphism. It follows that $\nabla^{X,Y}$ is a quasi-isomorphism if and only if $\nabla^{X_{n}\otimes_{R}\Nor\left(R^{(\Delta^{n})}\right),Y}$ is a quasi-isomorphism. Now we argue by induction on $m$ to show that $\nabla^{X_{n}\otimes_{R}\Nor\left(R^{(\Delta^{n})}\right),Y}$ is a quasi-isomorphism. If $m=0$, then $Y$ has only one non-zero term in degree $0$, so we can identify $Y$ with the left $R$-module $Y_{0}$, and thus a similar argument as above shows that $\nabla^{X_{n}\otimes_{R}\Nor\left(R^{(\Delta^{n})}\right),Y}$ is a quasi-isomorphism. Let $m\geq 1$, and assume the result for $m-1$. Consider the degreewise split short exact sequence
$$0\rightarrow Y_{m-1\sqsubset}\rightarrow Y\rightarrow \Sigma^{m}Y_{m}\rightarrow 0$$
of $R$-complexes. Then we achieve the following commutative diagram with exact rows:
\small
\begin{equation*}
\begin{tikzcd}[column sep=1em,row sep=3em]
  0 \arrow{r} & \left(X_{n}\otimes_{R}\Nor\left(R^{(\Delta^{n})}\right)\right)\otimes_{R}Y_{m-1\sqsubset} \arrow{r} \arrow{d}{\nabla^{X_{n}\otimes_{R}\Nor\left(R^{(\Delta^{n})}\right),Y_{m-1\sqsubset}}}
  & \left(X_{n}\otimes_{R}\Nor\left(R^{(\Delta^{n})}\right)\right)\otimes_{R}Y \arrow{r} \arrow{d}{\nabla^{X_{n}\otimes_{R}\Nor\left(R^{(\Delta^{n})}\right),Y}}
  & \left(X_{n}\otimes_{R}\Nor\left(R^{(\Delta^{n})}\right)\right)\otimes_{R}\Sigma^{m}Y_{m} \arrow{r} \arrow{d}{\nabla^{X_{n}\otimes_{R}\Nor\left(R^{(\Delta^{n})}\right),\Sigma^{m}Y_{m}}} & 0
  \\
 0 \arrow{r} & \left(X_{n}\otimes_{R}\Nor\left(R^{(\Delta^{n})}\right)\right)\boxtimes_{R}Y_{m-1\sqsubset} \arrow{r}
  & \left(X_{n}\otimes_{R}\Nor\left(R^{(\Delta^{n})}\right)\right)\boxtimes_{R}Y \arrow{r}
  & \left(X_{n}\otimes_{R}\Nor\left(R^{(\Delta^{n})}\right)\right)\boxtimes_{R}\Sigma^{m}Y_{m} \arrow{r} & 0
\end{tikzcd}
\end{equation*}
\normalsize

\noindent
By the induction hypothesis, $\nabla^{X_{n}\otimes_{R}\Nor\left(R^{(\Delta^{n})}\right),Y_{m-1\sqsubset}}$ is a quasi-isomorphism. Therefore, the above diagram implies that $\nabla^{X_{n}\otimes_{R}\Nor\left(R^{(\Delta^{n})}\right),Y}$ is a quasi-isomorphism if and only if $\nabla^{X_{n}\otimes_{R}\Nor\left(R^{(\Delta^{n})}\right),\Sigma^{m}Y_{m}}$ is a quasi-isomorphism. We note that $\Nor\left(R^{(\Delta^{m})}\right)\otimes_{R}Y_{m}$ is a left-bounded $R$-complex with
$$\left(\Nor\left(R^{(\Delta^{m})}\right)\otimes_{R}Y_{m}\right)_{m} =\Nor\left(R^{(\Delta^{m})}\right)_{m}\otimes_{R}Y_{m} =R\otimes_{R}Y_{m} \cong Y_{m},$$
so by the above argument, $\nabla^{X_{n}\otimes_{R}\Nor\left(R^{(\Delta^{n})}\right),\Nor\left(R^{(\Delta^{m})}\right)\otimes_{R}Y_{m}}$ is a quasi-isomorphism if and only if $\nabla^{X_{n}\otimes_{R}\Nor\left(R^{(\Delta^{n})}\right),\Sigma^{m}Y_{m}}$ is a quasi-isomorphism. It follows that $\nabla^{X_{n}\otimes_{R}\Nor\left(R^{(\Delta^{n})}\right),Y}$ is a quasi-isomorphism if and only if $\nabla^{X_{n}\otimes_{R}\Nor\left(R^{(\Delta^{n})}\right),\Nor\left(R^{(\Delta^{m})}\right)\otimes_{R}Y_{m}}$ is a quasi-isomorphism. Altogether, we infer that $\nabla^{X,Y}$ is a quasi-isomorphism if and only if $\nabla^{X_{n}\otimes_{R}\Nor\left(R^{(\Delta^{n})}\right),\Nor\left(R^{(\Delta^{m})}\right)\otimes_{R}Y_{m}}$ is a quasi-isomorphism.

Now let $n,m\geq 0$. Viewing $[n]$ as a small category, we have $\nrv([n])=\Delta^{n}$. Since $[n]$ has a least element $0$, we can define a morphism $\zeta^{\Delta^{n}}:\Delta^{0}\rightarrow \Delta^{n}$ of simplicial sets by setting $\zeta_{k}^{\Delta^{n}}(0,...,0)=(0,...,0)$ for every $k\geq 0$. Then by Lemma \ref{4.8}, $\Nor\left(R^{\left(\zeta^{\Delta^{n}}\right)}\right):\Nor\left(R^{(\Delta^{0})}\right)\rightarrow \Nor\left(R^{(\Delta^{n})}\right)$ is a homotopy equivalence. Similarly, $\nrv([m])=\Delta^{m}$, $\zeta^{\Delta^{m}}:\Delta^{0}\rightarrow \Delta^{m}$ is given by $\zeta_{k}^{\Delta^{m}}(0,...,0)=(0,...,0)$ for every $k\geq 0$, and $\Nor\left(R^{\left(\zeta^{\Delta^{m}}\right)}\right):\Nor\left(R^{(\Delta^{0})}\right)\rightarrow \Nor\left(R^{(\Delta^{m})}\right)$ is a homotopy equivalence. Also, $[n]\times[m]$ is a poset with componentwise order, so $\nrv([n]\times[m])=\Delta^{n}\times \Delta^{m}$, $\zeta^{\Delta^{n}\times \Delta^{m}}:\Delta^{0}\rightarrow \Delta^{n}\times \Delta^{m}$ is given by $\zeta_{k}^{\Delta^{n}\times \Delta^{m}}(0,...,0)=\left((0,...,0),(0,...,0)\right)$ for every $k\geq 0$, and $\Nor\left(R^{\left(\zeta^{\Delta^{n}\times \Delta^{m}}\right)}\right):\Nor\left(R^{(\Delta^{0})}\right)\rightarrow \Nor\left(R^{(\Delta^{n}\times \Delta^{m})}\right)$ is a homotopy equivalence. Let
$$\eta:\Nor\left(R^{(\Delta^{0})}\right)\otimes_{R}\Nor\left(R^{(\Delta^{0})}\right)\rightarrow \Nor\left(R^{(\Delta^{0})}\right)$$
be the isomorphism given by $\eta_{0}(a\otimes b)=ab$ for every $a,b\in R$, and $\eta_{i}=0$ for every $i\neq 0$. Furthermore, let $\chi$ be the composition of the following morphisms of $\mathbb{Z}$-complexes:
$$\Nor\left(R^{(\Delta^{n})}\right)\otimes_{R}\Nor\left(R^{(\Delta^{m})}\right) \xrightarrow{\nabla^{\Nor\left(R^{(\Delta^{n})}\right),\Nor\left(R^{(\Delta^{m})}\right)}} \Nor\left(R^{(\Delta^{n})}\right)\boxtimes_{R}\Nor\left(R^{(\Delta^{m})}\right) $$$$ \xrightarrow{\cong}\Nor\left(R^{(\Delta^{n})}\otimes_{R}R^{(\Delta^{m})}\right) \xrightarrow{\cong}\Nor\left(R^{(\Delta^{n}\times \Delta^{m})}\right)$$

\noindent
Then the following diagram is commutative:
\begin{equation*}
\begin{tikzcd}[column sep=2em,row sep=4em]
  \Nor\left(R^{(\Delta^{0})}\right)\otimes_{R}\Nor\left(R^{(\Delta^{0})}\right) \arrow{r}{\eta} \arrow{d}[swap]{\Nor\left(R^{\left(\zeta^{\Delta^{n}}\right)}\right)\otimes_{R}\Nor\left(R^{\left(\zeta^{\Delta^{m}}\right)}\right)}
  & \Nor\left(R^{(\Delta^{0})}\right) \arrow{d}{\Nor\left(R^{\left(\zeta^{\Delta^{n}\times\Delta^{m}}\right)}\right)}
  \\
   \Nor\left(R^{(\Delta^{n})}\right)\otimes_{R}\Nor\left(R^{(\Delta^{m})}\right) \arrow{r}{\chi}
  & \Nor\left(R^{(\Delta^{n}\times \Delta^{m})}\right)
\end{tikzcd}
\end{equation*}

\noindent
To see this, we note that $\Nor\left(R^{(\Delta^{0})}\right)$ has only one non-zero term in degree $0$, so we must have $\Nor\left(R^{\left(\zeta^{\Delta^{n}}\right)}\right)_{i}= \Nor\left(R^{\left(\zeta^{\Delta^{m}}\right)}\right)_{i}= \Nor\left(R^{\left(\zeta^{\Delta^{n}\times \Delta^{m}}\right)}\right)_{i}=0$ for every $i\neq 0$. Therefore, it suffices to check the commutativity of the above diagram for $i=0$. If $i=0$, then we have the following commutative diagram:
\begin{equation*}
\begin{tikzcd}[column sep=2em,row sep=3em]
  R\otimes_{R}R \arrow{r}{\eta_{0}} \arrow{d}[swap]{R^{\left(\zeta_{0}^{\Delta^{n}}\right)}\otimes_{R}R^{\left(\zeta_{0}^{\Delta^{m}}\right)}}
  & R \arrow{d}{R^{\left(\zeta_{0}^{\Delta^{n}\times\Delta^{m}}\right)}}
  \\
   R^{n}\otimes_{R}R^{m} \arrow{r}{\chi_{0}}
  & R^{nm}
\end{tikzcd}
\end{equation*}

\noindent
Indeed, we have for every $a,b\in R$:
\begin{equation*}
\begin{split}
 \chi_{0}\left(\left(R^{\left(\zeta_{0}^{\Delta^{n}}\right)}\otimes_{R}R^{\left(\zeta_{0}^{\Delta^{m}}\right)}\right)(a\otimes b)\right) & = \chi_{0}\left(R^{\left(\zeta_{0}^{\Delta^{n}}\right)}(a)\otimes R^{\left(\zeta_{0}^{\Delta^{m}}\right)}(b)\right) = \chi_{0}\left((a,0,...,0)\otimes(b,0,...,0)\right) \\
 & = (ab,0,...,0) = R^{\left(\zeta_{0}^{\Delta^{n}\times\Delta^{m}}\right)}(ab) = R^{\left(\zeta_{0}^{\Delta^{n}\times\Delta^{m}}\right)}\left(\eta_{0}(a\otimes b)\right)
\end{split}
\end{equation*}

\noindent
Now $\eta$ is an isomorphism, and $\Nor\left(R^{\left(\zeta^{\Delta^{n}}\right)}\right)\otimes_{R}\Nor\left(R^{\left(\zeta^{\Delta^{m}}\right)}\right)$ and $\Nor\left(R^{\left(\zeta^{\Delta^{n}\times\Delta^{m}}\right)}\right)$ are homotopy equivalences, so the above diagram implies that $\chi$ is a homotopy equivalence. By the definition of $\chi$, we conclude that $\nabla^{\Nor\left(R^{(\Delta^{n})}\right),\Nor\left(R^{(\Delta^{m})}\right)}$ is a homotopy equivalence. Consider the following commutative diagram:
\footnotesize
\begin{equation*}
\begin{tikzcd}[column sep=1.5em,row sep=4em]
  X_{n}\otimes_{R}\left(\Nor\left(R^{(\Delta^{n})}\right)\otimes_{R}\Nor\left(R^{(\Delta^{m})}\right)\right)\otimes_{R}Y_{m} \arrow{d}[swap]{X_{n}\otimes_{R}\nabla^{\Nor\left(R^{(\Delta^{n})}\right),\Nor\left(R^{(\Delta^{m})}\right)}\otimes_{R}Y_{m}} \arrow{r}{\cong}
  & \left(X_{n}\otimes_{R}\Nor\left(R^{(\Delta^{n})}\right)\right)\otimes_{R}\left(\Nor\left(R^{(\Delta^{m})}\right)\otimes_{R}Y_{m}\right) \arrow{d}{\nabla^{X_{n}\otimes_{R}\Nor\left(R^{(\Delta^{n})}\right),\Nor\left(R^{(\Delta^{m})}\right)\otimes_{R}Y_{m}}}
  \\
  X_{n}\otimes_{R}\left(\Nor\left(R^{(\Delta^{n})}\right)\boxtimes_{R}\Nor\left(R^{(\Delta^{m})}\right)\right)\otimes_{R}Y_{m} \arrow{r}{\cong}
  & \left(X_{n}\otimes_{R}\Nor\left(R^{(\Delta^{n})}\right)\right)\boxtimes_{R}\left(\Nor\left(R^{(\Delta^{m})}\right)\otimes_{R}Y_{m}\right)
\end{tikzcd}
\end{equation*}
\normalsize

\noindent
Since tensor product functor preserves homotopy equivalences, $X_{n}\otimes_{R}\nabla^{\Nor\left(R^{(\Delta^{n})}\right),\Nor\left(R^{(\Delta^{m})}\right)}\otimes_{R}Y_{m}$ is a homotopy equivalence, so we deduce from the above diagram that $\nabla^{X_{n}\otimes_{R}\Nor\left(R^{(\Delta^{n})}\right),\Nor\left(R^{(\Delta^{m})}\right)\otimes_{R}Y_{m}}$ is a homotopy equivalence, so in particular, it is a quasi-isomorphism. Therefore, the above discussion shows that $\nabla^{X,Y}$ is a quasi-isomorphism.

Next suppose that $X$ and $Y$ are arbitrary. Then we have the following commutative diagram:
\begin{equation*}
\begin{tikzcd}[column sep=8.5em,row sep=2em]
  \underset{(n,m)\in \mathbb{Z}\times \mathbb{Z}}{\varinjlim}\left(X_{n\sqsubset}\otimes_{R}Y_{m\sqsubset}\right) \arrow{r}{\underset{(n,m)\in \mathbb{Z}\times \mathbb{Z}}{\varinjlim}\nabla^{X_{n\sqsubset}\otimes_{R}Y_{m\sqsubset}}} \arrow{d}[swap]{\cong}
  & \underset{(n,m)\in \mathbb{Z}\times \mathbb{Z}}{\varinjlim}\left(X_{n\sqsubset} \boxtimes_{R}Y_{m\sqsubset}\right) \arrow{d}{\cong}
  \\
  \left(\underset{n\in \mathbb{Z}}{\varinjlim}X_{n\sqsubset}\right) \otimes_{R} \left(\underset{m\in \mathbb{Z}}{\varinjlim}Y_{m\sqsubset}\right)  \arrow{r}{\nabla^{\underset{n\in \mathbb{Z}}{\varinjlim}X_{n\sqsubset},\underset{m\in \mathbb{Z}}{\varinjlim}Y_{m\sqsubset}}} \arrow{d}[swap]{\cong}
  & \left(\underset{n\in \mathbb{Z}}{\varinjlim}X_{n\sqsubset}\right) \boxtimes_{R} \left(\underset{m\in \mathbb{Z}}{\varinjlim}Y_{m\sqsubset}\right) \arrow{d}{\cong}
  \\
  X \otimes_{R}Y  \arrow{r}{\nabla^{X,Y}}
  & X \boxtimes_{R}Y
\end{tikzcd}
\end{equation*}

\noindent
By the special case, $\nabla^{X_{n\sqsubset}\otimes_{R}Y_{m\sqsubset}}$ is a quasi-isomorphism for every $n,m \geq 0$, so $\underset{(n,m)\in \mathbb{Z}\times \mathbb{Z}}{\varinjlim}\nabla^{X_{n\sqsubset}\otimes_{R}Y_{m\sqsubset}}$ is a quasi-isomorphism. Then the above diagram shows that $\nabla^{X,Y}$ is a quasi-isomorphism.
\end{proof}

\begin{corollary} \label{5.10}
Let $R$ be a ring, $X$ a connective $R^{\op}$-complex, and $Y$ a connective $R$-complex. Then the following natural $\mathbb{Z}$-isomorphism holds for every $i\geq 0$:
$$H_{i}(X\boxtimes_{R}Y) \cong H_{i}(X\otimes_{R}Y)$$
\end{corollary}

\begin{proof}
Follows from Theorem \ref{5.9}.
\end{proof}

\begin{corollary} \label{5.11}
Let $R$ be a ring. Then the following assertions hold:
\begin{enumerate}
\item[(i)] If $X$ is a connective $R^{\op}$-complex of flat modules, then $X\boxtimes_{R}-:\mathcal{C}_{\geq 0}(R)\rightarrow \mathcal{C}_{\geq 0}(\mathbb{Z})$ preserves quasi-isomorphisms.
\item[(ii)] If $Y$ is a connective $R$-complex of flat modules, then $-\boxtimes_{R}Y:\mathcal{C}_{\geq 0}(R^{\op})\rightarrow \mathcal{C}_{\geq 0}(\mathbb{Z})$ preserves quasi-isomorphisms.
\end{enumerate}
\end{corollary}

\begin{proof}
(i): Since $X$ is a right-bounded $R^{\op}$-complex of flat modules, \cite[Corollary A.2]{CT} or \cite[Example 5.4.8]{CFH} implies that $X\otimes_{R}-:\mathcal{C}_{\geq 0}(R)\rightarrow \mathcal{C}_{\geq 0}(\mathbb{Z})$ preserves acyclicity, so by \cite[Examples 1.1.F, (3)]{AF}, it also preserves quasi-isomorphisms. On the other hand, by Theorem \ref{5.9}, there is a natural quasi-isomorphism $\nabla^{X,Y}:X\otimes_{R}Y\rightarrow X\boxtimes_{R}Y$ of $\mathbb{Z}$-complexes. Let $g:Y\rightarrow Y'$ be a quasi-isomorphism of $R$-complexes. Consider the following commutative diagram:
\begin{equation*}
\begin{tikzcd}
  X\otimes_{R}Y \arrow{r}{X\otimes_{R}g} \arrow{d}[swap]{\nabla^{X,Y}} & [1em] X\otimes_{R}Y' \arrow{d}{\nabla^{X,Y'}}
  \\
  X\boxtimes_{R}Y \arrow{r}{X\boxtimes_{R}g} & X\boxtimes_{R}Y'
\end{tikzcd}
\end{equation*}

\noindent
As $X\otimes_{R}g$, $\nabla^{X,Y}$, and $\nabla^{X,Y'}$ are quasi-isomorphisms, the above diagram implies that $X\boxtimes_{R}g$ is a quasi-isomorphism.

(ii): Similar to (i).
\end{proof}

\section{Model Structure on Simplicial Commutative Algebras}

In this section, we deploy Theorem \ref{2.9} to transfer the model structure on connective chain complexes to simplicial commutative algebras. In order to overcome the acyclicity condition, we need to construct certain cylinder and path objects for simplicial commutative algebras. To this end, we need Lemmas \ref{6.1} and \ref{6.3} below. Variants of these lemmas are incorporated as Axiom SM7 into Quillen's theory of model categories to define a simplicially enriched model category; see \cite[Definition 2 on page 2.2]{Qu2} or \cite[Proposition 11.5 and Axiom 3.1]{GJ}. He proves that the category of simplicial sets satisfy these lemmas. However, the proofs that are available in the literature are based on the model structure of simplicial sets which we have totally avoided in this article. Instead, we leverage the shuffle product of connective chain complexes developed in the previous section to prove these lemmas for simplicial modules.

\begin{lemma} \label{6.1}
Let $R$ be a commutative ring, $\rho:M\rightarrow N$ a fibration in $\mathpzc{s}\mathcal{M}\mathpzc{od}(R)$, and $\iota:U\rightarrow V$ a morphism of simplicial sets. Consider the commutative diagram
\begin{equation*}
\begin{tikzcd}[column sep=5.5em,row sep=3em]
  \Map_{\mathpzc{s}\mathcal{S}\mathpzc{et}}(V,M) \arrow{r}{\Map_{\mathpzc{s}\mathcal{S}\mathpzc{et}}(V,\rho)} \arrow{d}[swap]{\Map_{\mathpzc{s}\mathcal{S}\mathpzc{et}}(\iota,M)}
  & \Map_{\mathpzc{s}\mathcal{S}\mathpzc{et}}(V,N) \arrow{d}{\Map_{\mathpzc{s}\mathcal{S}\mathpzc{et}}(\iota,N)}
  \\
  \Map_{\mathpzc{s}\mathcal{S}\mathpzc{et}}(U,M) \arrow{r}{\Map_{\mathpzc{s}\mathcal{S}\mathpzc{et}}(U,\rho)}
  & \Map_{\mathpzc{s}\mathcal{S}\mathpzc{et}}(U,N)
\end{tikzcd}
\end{equation*}

\noindent
and form the following pullback diagram:
\begin{equation*}
  \begin{tikzcd}
  \Map_{\mathpzc{s}\mathcal{S}\mathpzc{et}}(V,M) \arrow{dr}{\eta} \arrow[bend right, swap]{ddr}{\Map_{\mathpzc{s}\mathcal{S}\mathpzc{et}}(\iota,M)}
  \arrow[bend left]{drr}{\Map_{\mathpzc{s}\mathcal{S}\mathpzc{et}}(V,\rho)} & &
  \\
  & \Map_{\mathpzc{s}\mathcal{S}\mathpzc{et}}(U,M)\bigsqcap_{\Map_{\mathpzc{s}\mathcal{S}\mathpzc{et}}(U,N)} \Map_{\mathpzc{s}\mathcal{S}\mathpzc{et}}(V,N) \arrow{r}{r} \arrow{d}[swap]{s} & \Map_{\mathpzc{s}\mathcal{S}\mathpzc{et}}(V,N) \arrow{d}{\Map_{\mathpzc{s}\mathcal{S}\mathpzc{et}}(\iota,N)}
  \\ [1em]
  & \Map_{\mathpzc{s}\mathcal{S}\mathpzc{et}}(U,M) \arrow{r}{\Map_{\mathpzc{s}\mathcal{S}\mathpzc{et}}(U,\rho)} & \Map_{\mathpzc{s}\mathcal{S}\mathpzc{et}}(U,N)
\end{tikzcd}
\end{equation*}

\noindent
If $R^{(\iota)}:R^{(U)}\rightarrow R^{(V)}$ is a cofibration in $\mathpzc{s}\mathcal{M}\mathpzc{od}(R)$, then $\eta$ is a fibration in $\mathpzc{s}\mathcal{M}\mathpzc{od}(R)$ which is trivial if either $\rho$ or $R^{(\iota)}$ is.
\end{lemma}

\begin{proof}
Let $\varrho:K\rightarrow L$ be a morphism of simplicial $R$-modules. Consider a commutative diagram as follows:
\begin{equation*}
\begin{tikzcd}
  K \arrow{r}{p} \arrow{d}[swap]{\varrho}
  & \Map_{\mathpzc{s}\mathcal{S}\mathpzc{et}}(V,M) \arrow{d}{\eta}
  \\
  L \arrow{r}{q}
  & \Map_{\mathpzc{s}\mathcal{S}\mathpzc{et}}(U,M)\bigsqcap_{\Map_{\mathpzc{s}\mathcal{S}\mathpzc{et}}(U,N)} \Map_{\mathpzc{s}\mathcal{S}\mathpzc{et}}(V,N)
\end{tikzcd}
\end{equation*}

\noindent
Considering the commutative diagram
\begin{equation*}
\begin{tikzcd}
  K^{(U)} \arrow{r}{\varrho^{(U)}} \arrow{d}[swap]{K^{(\iota)}}
  & L^{(U)} \arrow{d}{L^{(\iota)}}
  \\
  K^{(V)} \arrow{r}{\varrho^{(V)}}
  & L^{(V)}
\end{tikzcd}
\end{equation*}

\noindent
there is a unique morphism $\zeta:K^{(V)}\bigsqcup_{K^{(U)}}L^{(U)}\rightarrow L^{(V)}$ that makes the following pushout diagram commutative:
\begin{equation*}
  \begin{tikzcd}
  K^{(U)} \arrow{r}{\varrho^{(U)}} \arrow{d}[swap]{K^{(\iota)}} & L^{(U)} \arrow{d} \arrow[bend left]{ddr}{L^{(\iota)}} &
  \\
  K^{(V)} \arrow{r} \arrow[bend right, swap]{drr}{\varrho^{(V)}} & K^{(V)}\bigsqcup_{K^{(U)}}L^{(U)} \arrow{dr}{\zeta} &
  \\
  & & L^{(V)}
\end{tikzcd}
\end{equation*}

\noindent
On the other hand, considering the natural bijections
$$\phi_{KVM}:\Mor_{\mathpzc{s}\mathcal{M}\mathpzc{od}(R)}\left(K,\Map_{\mathpzc{s}\mathcal{S}\mathpzc{et}}(V,M)\right)\rightarrow \Mor_{\mathpzc{s}\mathcal{M}\mathpzc{od}(R)}\left(K^{(V)},M\right)$$
and
$$\phi_{LUM}:\Mor_{\mathpzc{s}\mathcal{M}\mathpzc{od}(R)}\left(L,\Map_{\mathpzc{s}\mathcal{S}\mathpzc{et}}(U,M)\right)\rightarrow \Mor_{\mathpzc{s}\mathcal{M}\mathpzc{od}(R)}\left(L^{(U)},M\right)$$
from Remark \ref{4.4}, we get the following commutative diagram:
\begin{equation*}
\begin{tikzcd}[column sep=4em,row sep=2em]
  K^{(U)} \arrow{r}{\varrho^{(U)}} \arrow{d}[swap]{K^{(\iota)}}
  & L^{(U)} \arrow{d}{\phi_{LUM}(sq)}
  \\
  K^{(V)} \arrow{r}{\phi_{KVM}(p)}
  & M
\end{tikzcd}
\end{equation*}

\noindent
Hence there is a unique morphism $\xi:K^{(V)}\bigsqcup_{K^{(U)}}L^{(U)}\rightarrow M$ that makes the following pushout diagram commutative:
\begin{equation*}
  \begin{tikzcd}
  K^{(U)} \arrow{r}{\varrho^{(U)}} \arrow{d}[swap]{K^{(\iota)}} & L^{(U)} \arrow{d} \arrow[bend left]{ddr}{\phi_{LUM}(sq)} &
  \\
  K^{(V)} \arrow{r} \arrow[bend right, swap]{drr}{\phi_{KVM}(p)} & K^{(V)}\bigsqcup_{K^{(U)}}L^{(U)} \arrow{dr}{\xi} &
  \\
  & & M
\end{tikzcd}
\end{equation*}

\noindent
Then considering the natural bijection
$$\phi_{LVN}:\Mor_{\mathpzc{s}\mathcal{M}\mathpzc{od}(R)}\left(L,\Map_{\mathpzc{s}\mathcal{S}\mathpzc{et}}(V,N)\right)\rightarrow \Mor_{\mathpzc{s}\mathcal{M}\mathpzc{od}(R)}\left(L^{(V)},N\right)$$
we get the following commutative diagram:
\begin{equation*}
\begin{tikzcd}
  K^{(V)}\bigsqcup_{K^{(U)}}L^{(U)} \arrow{r}{\xi} \arrow{d}[swap]{\zeta}
  & M \arrow{d}{\rho}
  \\
  L^{(V)} \arrow{r}{\phi_{LVN}(rq)}
  & N
\end{tikzcd}
\end{equation*}

\noindent
If $\tau:L^{(V)}\rightarrow M$ is a morphism of simplicial $R$-modules that makes the diagram
\begin{equation*}
\begin{tikzcd}
  K^{(V)}\bigsqcup_{K^{(U)}}L^{(U)} \arrow{r}{\xi} \arrow{d}[swap]{\zeta}
  & M \arrow{d}{\rho}
  \\
  L^{(V)} \arrow{r}{\phi_{LVN}(rq)} \arrow{ru}{\tau}
  & N
\end{tikzcd}
\end{equation*}

\noindent
commutative, then considering the natural bijection
$$\phi_{LVM}: \Mor_{\mathpzc{s}\mathcal{M}\mathpzc{od}(R)}\left(L,\Map_{\mathpzc{s}\mathcal{S}\mathpzc{et}}(V,M)\right) \rightarrow \Mor_{\mathpzc{s}\mathcal{M}\mathpzc{od}(R)}\left(L^{(V)},M\right)$$
we see that the following diagram is commutative:
\begin{equation*}
\begin{tikzcd}[column sep=3em,row sep=3.5em]
  K \arrow{r}{p} \arrow{d}[swap]{\varrho}
  & \Map_{\mathpzc{s}\mathcal{S}\mathpzc{et}}(V,M) \arrow{d}{\eta}
  \\
  L \arrow{r}{q} \arrow{ru}{\phi_{LVM}^{-1}(\tau)}
  & \Map_{\mathpzc{s}\mathcal{S}\mathpzc{et}}(U,M)\bigsqcap_{\Map_{\mathpzc{s}\mathcal{S}\mathpzc{et}}(U,N)} \Map_{\mathpzc{s}\mathcal{S}\mathpzc{et}}(V,N)
\end{tikzcd}
\end{equation*}

\noindent
Therefore, in order to show that $\eta$ has the right lifting property against $\varrho$, it suffices to show that $\rho$ has the right lifting property against $\zeta$. Indeed, if $\rho$ is a fibration in $\mathpzc{s}\mathcal{M}\mathpzc{od}(R)$, then it suffices to show that $\zeta$ is a trivial cofibration in $\mathpzc{s}\mathcal{M}\mathpzc{od}(R)$, and if $\rho$ is a trivial fibration in $\mathpzc{s}\mathcal{M}\mathpzc{od}(R)$, then it suffices to show that $\zeta$ is a cofibration in $\mathpzc{s}\mathcal{M}\mathpzc{od}(R)$.

Denote the morphism $\Nor\left(R^{(\iota)}\right):\Nor\left(R^{(U)}\right)\rightarrow \Nor\left(R^{(V)}\right)$ in $\mathcal{C}_{\geq 0}(R)$ by $\mu:Z\rightarrow W$ for convenience. If $X$ is an $R$-complex and $T$ is a simplicial set, then in view of Remark \ref{4.4}, Corollary \ref{5.7}, and Theorem \ref{4.14}, we have the following natural isomorphisms:
$$\Nor\left(\DK(X)^{(T)}\right)\cong \Nor\left(R^{(T)}\otimes_{R}\DK(X)\right)\cong \Nor\left(R^{(T)}\right)\boxtimes_{R}\Nor\left(\DK(X)\right) \cong \Nor\left(R^{(T)}\right)\boxtimes_{R}X$$
As a consequence, assuming that $\varrho:K\rightarrow L$ is of the form $\DK(h):\DK(X)\rightarrow \DK(Y)$ for some morphism $h:X\rightarrow Y$ of connective $R$-complexes, then we get a commutative diagram as follows:
\begin{equation*}
\begin{tikzcd}
  \Nor\left(K^{(V)}\bigsqcup_{K^{(U)}}L^{(U)}\right) \arrow{r}{\Nor(\zeta)} \arrow{d}[swap]{\cong}
  & \Nor\left(L^{(V)}\right) \ar[equal]{d}
  \\
   \Nor\left(K^{(V)}\right)\bigsqcup_{\Nor\left(K^{(U)}\right)}\Nor\left(L^{(U)}\right) \arrow{r} \ar[equal]{d}
  & \Nor\left(L^{(V)}\right) \ar[equal]{d}
  \\
   \Nor\left(\DK(X)^{(V)}\right)\bigsqcup_{\Nor\left(\DK(X)^{(U)}\right)}\Nor\left(\DK(Y)^{(U)}\right) \arrow{r} \arrow{d}[swap]{\cong}
  & \Nor\left(\DK(Y)^{(V)}\right) \arrow{d}{\cong}
  \\
   \left(\Nor\left(R^{(V)}\right)\boxtimes_{R}X\right)\bigsqcup_{\left(\Nor\left(R^{(U)}\right)\boxtimes_{R}X\right)} \left(\Nor\left(R^{(U)}\right)\boxtimes_{R}Y\right) \arrow{r} \ar[equal]{d}
  & \Nor\left(R^{(V)}\right)\boxtimes_{R}Y \ar[equal]{d}
   \\
   \left(W\boxtimes_{R}X\right)\bigsqcup_{\left(Z\boxtimes_{R}X\right)} \left(Z\boxtimes_{R}Y\right) \arrow{r}{\zeta'}
  & W\boxtimes_{R}Y
\end{tikzcd}
\end{equation*}

\noindent
Thus in order to show that $\zeta$ is a (trivial) cofibration in $\mathpzc{s}\mathcal{M}\mathpzc{od}(R)$, i.e. $\Nor(\zeta)$ is a (trivial) cofibration in $\mathcal{C}_{\geq 0}(R)$, it suffices to show that $\zeta'$ is a (trivial) cofibration in $\mathcal{C}_{\geq 0}(R)$.

By Theorem \ref{3.1}, $\mathcal{C}_{\geq 0}(R)$ is a cofibrantly generated model category with the corresponding sets
$$\mathcal{X}=\{\iota^{n}:S(n-1)\rightarrow D(n)\suchthat n\geq 1\} \cup \{\lambda:0\rightarrow S(0)\}$$
where $\iota^{n}$ is the inclusion morphism for every $n\geq 1$, and
$$\mathcal{Y}=\{\kappa^{n}:0\rightarrow D(n)\suchthat n\geq 1\}.$$
In light of the proof of Corollary \ref{2.16} and Theorem \ref{4.15}, we see that $\mathpzc{s}\mathcal{M}\mathpzc{od}(R)$ is a cofibrantly generated model category with the corresponding sets $\DK(\mathcal{X})$ and $\DK(\mathcal{Y})$. In particular, $\RLP\left(\DK(\mathcal{X})\right)$ is the class of trivial fibrations and $\RLP\left(\DK(\mathcal{Y})\right)$ is the class of fibrations in $\mathpzc{s}\mathcal{M}\mathpzc{od}(R)$. As a result, in order to show that $\eta$ is a fibration in $\mathpzc{s}\mathcal{M}\mathpzc{od}(R)$, it suffices to show that $\eta$ has the right lifting property against morphisms in $\DK(\mathcal{Y})$, i.e. to choose $\varrho:K\rightarrow L$ from $\DK(\mathcal{Y})$ and solve the corresponding lifting problem. Similarly, to show that $\eta$ is a trivial fibration in $\mathpzc{s}\mathcal{M}\mathpzc{od}(R)$, it suffices to show that $\eta$ has the right lifting property against morphisms in $\DK(\mathcal{X})$, i.e. to choose $\varrho:K\rightarrow L$ from $\DK(\mathcal{X})$ and solve the corresponding lifting problem.

Let $X$ be a connective $R$-complex. We analyze the shuffle products $X\boxtimes_{R}D(n)$ and $X\boxtimes_{R}S(n-1)$ for every $n\geq 1$ more closely. For any $m\geq 0$, set:
\small
$$\mathcal{S}_{m}= \left\{(f,g) \in \Mor_{\Delta}\left([m],[k]\right)\times \Mor_{\Delta}\left([m],[l]\right) \suchthat \begin{tabular}{ccc}
 $0 \leq k,l \leq m$, $f$ and $g$ are surjective, and $\phi_{f,g}:[m]\rightarrow [k]\times[l]$ \\
 given by $\phi_{f,g}(i)=\left(f(i),g(i)\right)$ for every $i\in [m]$ is injective
 \end{tabular} \right\}$$
\normalsize
Let $n\geq 1$ and $m\geq 0$. If $l\neq n,n-1$, then $D(n)_{l}=0$. Moreover, $D(n)_{n}=D(n)_{n-1}=R$. Let $\mathcal{S}_{m}'$ be the subset of $\mathcal{S}_{m}$ in which we assume $l=n$ or $n-1$. Then we have the following natural isomorphisms:
$$\left(X\boxtimes_{R}D(n)\right)_{m}=\textstyle\bigoplus_{(f,g)\in \mathcal{S}_{m}}\left(X_{k}\otimes_{R}D(n)_{l}\right)\cong \textstyle\bigoplus_{(f,g)\in \mathcal{S}_{m}'}\left(X_{k}\otimes_{R}R\right)\cong \textstyle\bigoplus_{(f,g)\in \mathcal{S}_{m}'}X_{k}$$
If $l\neq n-1$, then $S(n-1)_{l}=0$. Moreover, $S(n-1)_{n-1}=R$. Let $\mathcal{S}_{m}''$ be the subset of $\mathcal{S}_{m}$ in which we assume $l=n-1$. Then we have the following natural isomorphisms:
$$\left(X\boxtimes_{R}S(n)\right)_{m}=\textstyle\bigoplus_{(f,g)\in \mathcal{S}_{m}}\left(X_{k}\otimes_{R}S(n)_{l}\right)\cong \textstyle\bigoplus_{(f,g)\in \mathcal{S}_{m}''}\left(X_{k}\otimes_{R}R\right)\cong \textstyle\bigoplus_{(f,g)\in \mathcal{S}_{m}''}X_{k}$$
It is clear that $\mathcal{S}_{m}''\subseteq \mathcal{S}_{m}'$. Accordingly, we let $\nu_{m}:\mathcal{S}_{m}''\rightarrow \mathcal{S}_{m}'$ be the inclusion map.

Now suppose that $R^{(\iota)}:R^{(U)}\rightarrow R^{(V)}$ is a cofibration in $\mathpzc{s}\mathcal{M}\mathpzc{od}(R)$, i.e. $\mu:Z\rightarrow W$ is a cofibration in $\mathcal{C}_{\geq 0}(R)$. We will show that $\eta$ is a fibration in $\mathpzc{s}\mathcal{M}\mathpzc{od}(R)$. To this end, let $n\geq 1$, and let $\varrho:K\rightarrow L$ be the morphism $\DK(\kappa^{n}):\DK(0)\rightarrow \DK\left(D(n)\right)$. As we observed before, it suffices to show that
$$\zeta':\left(W\boxtimes_{R}0\right)\textstyle\bigsqcup_{\left(Z\boxtimes_{R}0\right)}\left(Z\boxtimes_{R}D(n)\right)\rightarrow W\boxtimes_{R}D(n)$$
is a trivial cofibration. We have the following commutative diagram:
\begin{equation*}
\begin{tikzcd}
 \left(W\boxtimes_{R}0\right)\bigsqcup_{\left(Z\boxtimes_{R}0\right)}\left(Z\boxtimes_{R}D(n)\right) \arrow{r}{\zeta'} \arrow{d}[swap]{\cong}
  & W\boxtimes_{R}D(n) \ar[equal]{d}
  \\
   0\bigsqcup_{0}\left(Z\boxtimes_{R}D(n)\right) \arrow{d}[swap]{\cong}
  & W\boxtimes_{R}D(n) \ar[equal]{d}
  \\
  Z\boxtimes_{R}D(n) \arrow{r}{\mu\boxtimes_{R}D(n)}
  & W\boxtimes_{R}D(n)
\end{tikzcd}
\end{equation*}

\noindent
As $\mu$ is a cofibration, we have the short exact sequence $0\rightarrow Z\xrightarrow{\mu} W\rightarrow \Coker(\mu)\rightarrow 0$ of $R$-complexes in which $\Coker(\mu)$ is an $R$-complex of projective modules. In particular, for any $k\geq 0$, we have the short exact sequence $0\rightarrow Z_{k}\xrightarrow{\mu_{k}} W_{k}\rightarrow \Coker(\mu_{k})\rightarrow 0$ of $R$-modules in which $\Coker(\mu_{k})$ is projective. Now consider the following sequence:
$$0\rightarrow Z\boxtimes_{R}D(n)\xrightarrow{\mu\boxtimes_{R}D(n)} W\boxtimes_{R}D(n)\rightarrow \Coker(\mu)\boxtimes_{R}D(n)\rightarrow 0$$
For each $m\geq 0$, we have the following commutative diagram:
\begin{equation*}
\begin{tikzcd}
 0 \arrow{r} & \left(Z\boxtimes_{R}D(n)\right)_{m} \arrow{r}{\left(\mu\boxtimes_{R}D(n)\right)_{m}} \arrow{d}{\cong} & [3.5em]
 \left(W\boxtimes_{R}D(n)\right)_{m} \arrow{r} \arrow{d}{\cong} &
 \left(\Coker(\mu)\boxtimes_{R}D(n)\right)_{m} \arrow{r} \arrow{d}{\cong} & 0
  \\
 0 \arrow{r} & \bigoplus_{(f,g)\in \mathcal{S}_{m}'}Z_{k} \arrow{r}{\bigoplus_{(f,g)\in \mathcal{S}_{m}'}\mu_{k}} &
 \bigoplus_{(f,g)\in \mathcal{S}_{m}'}W_{k} \arrow{r} &
 \bigoplus_{(f,g)\in \mathcal{S}_{m}'}\Coker(\mu_{k}) \arrow{r} & 0
\end{tikzcd}
\end{equation*}

\noindent
Now the second row is exact and $\bigoplus_{(f,g)\in \mathcal{S}_{m}'}\Coker(\mu_{k})$ is projective, so it follows that the first row is exact and $\left(\Coker(\mu)\boxtimes_{R}D(n)\right)_{m}$ is projective. As a result, the previous sequence is exact and $\Coker(\mu)\boxtimes_{R}D(n)$ is an $R$-complex of projective modules. This means that $\mu\boxtimes_{R}D(n)$ is a cofibration, so $\zeta'$ is a cofibration. On the other hand, we have the following commutative diagram for every $i\geq 0$:
\begin{equation*}
\begin{tikzcd}
 H_{i}\left(Z\boxtimes_{R}D(n)\right) \arrow{r}{H_{i}\left(\mu\boxtimes_{R}D(n)\right)} \arrow{d}[swap]{\cong}
  & [3.5em] H_{i}\left(W\boxtimes_{R}D(n)\right) \arrow{d}{\cong}
  \\
  H_{i}\left(Z\otimes_{R}D(n)\right) \arrow{r}{H_{i}\left(\mu\otimes_{R}D(n)\right)}
  & H_{i}\left(W\otimes_{R}D(n)\right)
\end{tikzcd}
\end{equation*}

\noindent
We note that $D(n)$ is an exact $R$-complex whose terms and boundaries are all flat, so by the K\"{u}nneth isomorphisms in \cite[Corollary 10.84]{Ro}, $Z\otimes_{R}D(n)$ and $W\otimes_{R}D(n)$ are exact. As a result, for any $i\geq 0$, we notice that
$$0=H_{i}\left(Z\otimes_{R}D(n)\right)\xrightarrow{H_{i}\left(\mu\otimes_{R}D(n)\right)} H_{i}\left(W\otimes_{R}D(n)\right)=0$$
is an isomorphism, so the above diagram implies that $H_{i}\left(\mu\boxtimes_{R}D(n)\right)$ is an isomorphism. Thus $\mu\boxtimes_{R}D(n)$ is a quasi-isomorphism, so $\zeta'$ is a quasi-isomorphism. Altogether, $\zeta'$ is a trivial cofibration.

Next suppose that either $R^{(\iota)}:R^{(U)}\rightarrow R^{(V)}$ is a trivial cofibration or $\rho:M\rightarrow N$ is a trivial fibration in $\mathpzc{s}\mathcal{M}\mathpzc{od}(R)$. We will show that $\eta$ is a trivial fibration in $\mathpzc{s}\mathcal{M}\mathpzc{od}(R)$. To this end, let $n\geq 1$, and let $\varrho:K\rightarrow L$ be the morphism $\DK(\iota^{n}):\DK\left(S(n-1)\right)\rightarrow \DK\left(D(n)\right)$. We show that
$$\zeta':\left(W\boxtimes_{R}S(n-1)\right)\textstyle\bigsqcup_{\left(Z\boxtimes_{R}S(n-1)\right)}\left(Z\boxtimes_{R}D(n)\right)\rightarrow W\boxtimes_{R}D(n)$$
is a cofibration. Consider the following pushout diagram:
\begin{equation*}
  \begin{tikzcd}
  Z\boxtimes_{R}S(n-1) \arrow{r}{Z\boxtimes_{R}\iota^{n}} \arrow{d}[swap]{\mu\boxtimes_{R}S(n-1)} & Z\boxtimes_{R}D(n) \arrow{d}{\varpi} \arrow[bend left]{ddr}{\mu\boxtimes_{R}D(n)} &
  \\ [1em]
  W\boxtimes_{R}S(n-1) \arrow{r} \arrow[bend right, swap]{drr}{W\boxtimes_{R}\iota^{n}} & \left(W\boxtimes_{R}S(n-1)\right)\bigsqcup_{\left(Z\boxtimes_{R}S(n-1)\right)}\left(Z\boxtimes_{R}D(n)\right) \arrow{dr}{\zeta'} &
  \\
  & & W\boxtimes_{R}D(n)
\end{tikzcd}
\end{equation*}

\noindent
In each degree $m\geq 0$, the above diagram is naturally isomorphic to the following diagram:
\begin{equation*}
  \begin{tikzcd}
   \bigoplus_{(f,g)\in \mathcal{S}_{m}''}Z_{k} \arrow{r}{\overline{\nu_{m}}} \arrow{d}[swap]{\bigoplus_{(f,g)\in \mathcal{S}_{m}''}\mu_{k}} & \bigoplus_{(f,g)\in \mathcal{S}_{m}'}Z_{k} \arrow{d}{p_{m}} \arrow[bend left]{ddr}{\bigoplus_{(f,g)\in \mathcal{S}_{m}'}\mu_{k}} &
  \\ [1em]
  \bigoplus_{(f,g)\in \mathcal{S}_{m}''}W_{k} \arrow{r} \arrow[bend right, swap]{drr}{\overline{\nu_{m}}} & \left(\bigoplus_{(f,g)\in \mathcal{S}_{m}''}W_{k}\right) \bigsqcup_{\left(\bigoplus_{(f,g)\in \mathcal{S}_{m}''}Z_{k}\right)} \left(\bigoplus_{(f,g)\in \mathcal{S}_{m}'}Z_{k}\right) \arrow{dr}{q_{m}} &
  \\
  & & \bigoplus_{(f,g)\in \mathcal{S}_{m}'}W_{k}
\end{tikzcd}
\end{equation*}

\noindent
Let $m\geq 0$. By definition, we have:
$$\left(\textstyle\bigoplus_{(f,g)\in \mathcal{S}_{m}''}W_{k}\right) \textstyle\bigsqcup_{\left(\textstyle\bigoplus_{(f,g)\in \mathcal{S}_{m}''}Z_{k}\right)} \left(\textstyle\bigoplus_{(f,g)\in \mathcal{S}_{m}'}Z_{k}\right)=\frac{\left(\textstyle\bigoplus_{(f,g)\in \mathcal{S}_{m}''}W_{k}\right) \oplus \left(\textstyle\bigoplus_{(f,g)\in \mathcal{S}_{m}'}Z_{k}\right)}{\left\{\left(\left(\textstyle\bigoplus_{(f,g)\in \mathcal{S}_{m}''}\mu_{k}\right)(z),-\overline{\nu_{m}}(z)\right) \suchthat z\in \bigoplus_{(f,g)\in \mathcal{S}_{m}''}Z_{k} \right\}}$$
Let $x=\left(x_{f,g}\right)_{(f,g)\in \mathcal{S}_{m}''}\in \bigoplus_{(f,g)\in \mathcal{S}_{m}''}W_{k}$ and $y=\left(y_{f,g}\right)_{(f,g)\in \mathcal{S}_{m}'}\in \bigoplus_{(f,g)\in \mathcal{S}_{m}'}Z_{k}$. For any $(f,g)\in \mathcal{S}_{m}'$, define:
\begin{equation*}
 \label{eqn:damage piecewise}
 x_{f,g}':=
 \begin{dcases}
  x_{f,g} & \textrm{if } (f,g)\in \mathcal{S}_{m}'' \\
  0 & \textrm{if } (f,g)\in \mathcal{S}_{m}' \backslash \mathcal{S}_{m}''
 \end{dcases}
\end{equation*}

\noindent
Then $\overline{\nu_{m}}(x)=\left(x_{f,g}'\right)_{(f,g)\in \mathcal{S}_{m}'}$. Thus we have:
$$q_{m}\left(\overline{(x,y)}\right) = \overline{\nu_{m}}(x) + \left(\textstyle\bigoplus_{(f,g)\in \mathcal{S}_{m}'}\mu_{k}\right)(y) = \left(x_{f,g}'+\mu_{k}\left(y_{f,g}\right)\right)_{(f,g)\in \mathcal{S}_{m}'}$$
Now if $q_{m}\left(\overline{(x,y)}\right)=0$, then $x_{f,g}'=-\mu_{k}\left(y_{f,g}\right)$ for every $(f,g)\in \mathcal{S}_{m}'$. Therefore, $\mu_{k}\left(y_{f,g}\right)=-x_{f,g}'=-x_{f,g}$ for every $(f,g)\in \mathcal{S}_{m}''$. Moreover, $\mu_{k}\left(y_{f,g}\right)=-x_{f,g}'=0$ for every $(f,g)\in \mathcal{S}_{m}'\backslash \mathcal{S}_{m}''$. But $\mu$ is a cofibration, so $\mu_{k}$ is injective, implying that $y_{f,g}=0$ for every $(f,g)\in \mathcal{S}_{m}'\backslash \mathcal{S}_{m}''$. Set $z=\left(-y_{f,g}\right)_{(f,g)\in \mathcal{S}_{m}''}\in \bigoplus_{(f,g)\in \mathcal{S}_{m}''}Z_{k}$. Then $\overline{\nu_{m}}(z)=\left(-y_{f,g}'\right)_{(f,g)\in \mathcal{S}_{m}'}= \left(-y_{f,g}\right)_{(f,g)\in \mathcal{S}_{m}'}$ where $y_{f,g}'$ is defined similar to $x_{f,g}'$. Thus we have $x=\left(x_{f,g}\right)_{(f,g)\in \mathcal{S}_{m}''} = \left(-\mu_{k}\left(y_{f,g}\right)\right)_{(f,g)\in \mathcal{S}_{m}''} = \left(\bigoplus_{(f,g)\in \mathcal{S}_{m}''}\mu_{k}\right)(z)$ and $y=\left(y_{f,g}\right)_{(f,g)\in \mathcal{S}_{m}'} = -\overline{\nu_{m}}(z)$. This shows that $\overline{(x,y)}=0$. Therefore, $q_{m}$ is injective. On the other hand, from the above diagram, we have $q_{m}p_{m}= \bigoplus_{(f,g)\in \mathcal{S}_{m}'}\mu_{k}$, so $\im \left(\bigoplus_{(f,g)\in \mathcal{S}_{m}'}\mu_{k}\right) \subseteq \im(q_{m})$. Therefore, we get an $R$-epimorphism $\psi_{m}:\Coker\left(\bigoplus_{(f,g)\in \mathcal{S}_{m}'}\mu_{k}\right) \rightarrow \Coker(q_{m})$ given by $\psi_{m}\left(x+ \im\left(\bigoplus_{(f,g)\in \mathcal{S}_{m}'}\mu_{k}\right)\right) = x+ \im(q_{m})$ for every $x\in \bigoplus_{(f,g)\in \mathcal{S}_{m}'}W_{k}$. We next define an $R$-homomorphism $\chi_{m}: \bigoplus_{(f,g)\in \mathcal{S}_{m}'}W_{k} \rightarrow \Coker\left(\bigoplus_{(f,g)\in \mathcal{S}_{m}'}\mu_{k}\right)$ as follows. Let $w=\left(w_{f,g}\right)_{(f,g)\in \mathcal{S}_{m}'}\in \bigoplus_{(f,g)\in \mathcal{S}_{m}'}W_{k}$. For any $(f,g)\in \mathcal{S}_{m}'$, define:
\begin{equation*}
 \label{eqn:damage piecewise}
 w_{f,g}'':=
 \begin{dcases}
  w_{f,g} & \textrm{if } (f,g)\in \mathcal{S}_{m}' \backslash \mathcal{S}_{m}'' \\
  0 & \textrm{if } (f,g)\in \mathcal{S}_{m}''
 \end{dcases}
\end{equation*}

\noindent
Set $\chi_{m}(w):=\left(w_{f,g}''\right)_{(f,g)\in \mathcal{S}_{m}'}+ \im\left(\bigoplus_{(f,g)\in \mathcal{S}_{m}'}\mu_{k}\right)$. If $w\in \im(q_{m})$, then there are elements $x=\left(x_{f,g}\right)_{(f,g)\in \mathcal{S}_{m}''} \in \bigoplus_{(f,g)\in \mathcal{S}_{m}''}W_{k}$ and $y=\left(y_{f,g}\right)_{(f,g)\in \mathcal{S}_{m}'} \in \bigoplus_{(f,g)\in \mathcal{S}_{m}'}Z_{k}$ such that $w=q_{m}\left(\overline{(x,y)}\right)=\left(x_{f,g}' + \mu_{k}\left(y_{f,g}\right)\right)_{(f,g)\in \mathcal{S}_{m}'}$. We note that $x_{f,g}'=0$ for every $(f,g)\in \mathcal{S}_{m}'\backslash \mathcal{S}_{m}''$, so by the above construction, $\left(x_{f,g}'\right)''=0$ for every $(f,g)\in \mathcal{S}_{m}'$. Therefore, we get:
\begin{equation*}
\begin{split}
 \chi_{m}(w) & = \left(\left(x_{f,g}'\right)''+\mu_{k}\left(y_{f,g}''\right)\right)_{(f,g)\in \mathcal{S}_{m}'} + \im\left(\textstyle\bigoplus_{(f,g)\in \mathcal{S}_{m}'}\mu_{k}\right) \\
 & = \left(\mu_{k}\left(y_{f,g}''\right)\right)_{(f,g)\in \mathcal{S}_{m}'} + \im\left(\textstyle\bigoplus_{(f,g)\in \mathcal{S}_{m}'}\mu_{k}\right)\\
 & = \left(\textstyle\bigoplus_{(f,g)\in \mathcal{S}_{m}'}\mu_{k}\right)\left(\left(y_{f,g}''\right)_{(f,g)\in \mathcal{S}_{m}'}\right) + \im\left(\textstyle\bigoplus_{(f,g)\in \mathcal{S}_{m}'}\mu_{k}\right) \\
 & = 0
\end{split}
\end{equation*}

\noindent
As a result, $\chi_{m}\left(\im(q_{m})\right)=0$, so $\chi_{m}$ induces an $R$-homomorphism $\overline{\chi_{m}}:\Coker(q_{m})\rightarrow \Coker\left(\bigoplus_{(f,g)\in\mathcal{S}_{m}'}\mu_{k}\right)$ given by $\overline{\chi_{m}}\left(w+\im(q_{m})\right)=\left(w_{f,g}''\right)_{(f,g)\in \mathcal{S}_{m}'}+ \im\left(\bigoplus_{(f,g)\in \mathcal{S}_{m}'}\mu_{k}\right)$. Now we have for every $w=\left(w_{f,g}\right)_{(f,g)\in \mathcal{S}_{m}'}\in \bigoplus_{(f,g)\in \mathcal{S}_{m}'}W_{k}$:
$$\psi_{m}\left(\overline{\chi_{m}}\left(w+\im(q_{m})\right)\right)= \psi_{m}\left(\left(w_{f,g}''\right)_{(f,g)\in \mathcal{S}_{m}'}+ \im\left(\textstyle\bigoplus_{(f,g)\in \mathcal{S}_{m}'}\mu_{k}\right)\right)= \left(w_{f,g}''\right)_{(f,g)\in \mathcal{S}_{m}'}+\im(q_{m})$$
Let $\widetilde{\nu_{m}}:\bigoplus_{(f,g)\in \mathcal{S}_{m}'}W_{k} \rightarrow \bigoplus_{(f,g)\in \mathcal{S}_{m}''}W_{k}$ be the induced $R$-homomorphism which projects a tuple onto its components from $\mathcal{S}_{m}''$. Setting $x= \widetilde{\nu_{m}}\left(\left(w_{f,g}-w_{f,g}''\right)_{(f,g)\in \mathcal{S}_{m}'}\right) \in \bigoplus_{(f,g)\in \mathcal{S}_{m}''}W_{k}$, we see that:
$$w-\left(w_{f,g}''\right)_{(f,g)\in \mathcal{S}_{m}'} = \left(w_{f,g}-w_{f,g}''\right)_{(f,g)\in \mathcal{S}_{m}'} = \overline{\nu_{m}}(x)=q_{m}\left(\overline{(x,0)}\right)$$
Therefore, $w-\left(w_{f,g}''\right)_{(f,g)\in \mathcal{S}_{m}'} \in \im(q_{m})$, so we get:
$$\psi_{m}\left(\overline{\chi_{m}}\left(w+\im(q_{m})\right)\right)=\left(w_{f,g}''\right)_{(f,g)\in \mathcal{S}_{m}'}+\im(q_{m}) = w+\im(q_{m})$$
This shows that $\psi_{m}\overline{\chi_{m}}=1^{\Coker(q_{m})}$. Thus the short exact sequence
$$0\rightarrow \Ker(\psi_{m})\rightarrow \Coker\left(\textstyle\bigoplus_{(f,g)\in \mathcal{S}_{m}'}\mu_{k}\right)\xrightarrow{\psi_{m}} \Coker(q_{m})\rightarrow 0$$
splits, so in particular, $\Coker(q_{m})$ is isomorphic to a direct summand of $\Coker\left(\bigoplus_{(f,g)\in \mathcal{S}_{m}'}\mu_{k}\right)$. But $\Coker\left(\bigoplus_{(f,g)\in \mathcal{S}_{m}'}\mu_{k}\right) \cong \bigoplus_{(f,g)\in \mathcal{S}_{m}'}\Coker(\mu_{k})$ is projective as $\mu$ is a cofibration. As a consequence, $\Coker(q_{m})$ is projective. This shows that $\zeta'$ is a cofibration.

If $\rho:M\rightarrow N$ is a trivial fibration in $\mathpzc{s}\mathcal{M}\mathpzc{od}(R)$, then we are done. If $R^{(\iota)}:R^{(U)}\rightarrow R^{(V)}$ is a trivial cofibration in $\mathpzc{s}\mathcal{M}\mathpzc{od}(R)$, i.e. $\mu:Z\rightarrow W$ is a trivial cofibration in $\mathcal{C}_{\geq 0}(R)$, then we need to further show that $\zeta'$ is a weak equivalence, i.e. a quasi-isomorphism. Since for any $n\geq 1$, $D(n)$ and $S(n-1)$ are connective $R$-complexes of flat modules, and $\mu$ is a weak equivalence i.e. a quasi-isomorphism, Corollary \ref{5.11} implies that $\mu\boxtimes_{R}S(n-1)$ and $\mu\boxtimes_{R}D(n)$ are quasi-isomorphisms. On the other hand, since $\mu$ is a cofibration, $\mu_{k}:Z_{k}\rightarrow W_{k}$ is injective with projective cokernel for every $k\geq 0$, so $\bigoplus_{(f,g)\in \mathcal{S}_{m}''}\mu_{k}:\bigoplus_{(f,g)\in \mathcal{S}_{m}''}Z_{k} \rightarrow \bigoplus_{(f,g)\in \mathcal{S}_{m}''}W_{k}$ is injective with projective cokernel. It follows that $\mu\boxtimes_{R}S(n-1)$ is a cofibration. Thus $\mu\boxtimes_{R}S(n-1)$ is a trivial cofibration, so by \cite[Proposition 7.2.12]{Hi}, we deduce that $\varpi$ is a trivial cofibration, whence $\varpi$ is a quasi-isomorphism. Now $\zeta'\varpi=\mu\boxtimes_{R}D(n)$, so we conclude that $\zeta'$ is a quasi-isomorphism.

Now let $\varrho:K\rightarrow L$ be the morphism $\DK(\lambda):\DK(0)\rightarrow \DK\left(S(0)\right)$. We show that
$$\zeta':\left(W\boxtimes_{R}0\right)\textstyle\bigsqcup_{\left(Z\boxtimes_{R}0\right)}\left(Z\boxtimes_{R}S(0)\right)\rightarrow W\boxtimes_{R}S(0)$$
is a cofibration. We have the following commutative diagram:
\begin{equation*}
\begin{tikzcd}
 \left(W\boxtimes_{R}0\right)\bigsqcup_{\left(Z\boxtimes_{R}0\right)}\left(Z\boxtimes_{R}S(0)\right) \arrow{r}{\zeta'} \arrow{d}[swap]{\cong}
  & W\boxtimes_{R}S(0) \ar[equal]{d}
  \\
   0\bigsqcup_{0}\left(Z\boxtimes_{R}S(0)\right) \arrow{d}[swap]{\cong}
  & W\boxtimes_{R}S(0) \ar[equal]{d}
  \\
  Z\boxtimes_{R}S(0) \arrow{r}{\mu\boxtimes_{R}S(0)}
  & W\boxtimes_{R}S(0)
\end{tikzcd}
\end{equation*}

\noindent
As $\mu$ is a cofibration, we have the short exact sequence $0\rightarrow Z\xrightarrow{\mu} W\rightarrow \Coker(\mu)\rightarrow 0$ of $R$-complexes in which $\Coker(\mu)$ is an $R$-complex of projective modules. In particular, for any $k\geq 0$, we have the short exact sequence $0\rightarrow Z_{k}\xrightarrow{\mu_{k}} W_{k}\rightarrow \Coker(\mu_{k})\rightarrow 0$ of $R$-modules in which $\Coker(\mu_{k})$ is projective. Now consider the following sequence:
$$0\rightarrow Z\boxtimes_{R}S(0)\xrightarrow{\mu\boxtimes_{R}S(0)} W\boxtimes_{R}S(0)\rightarrow \Coker(\mu)\boxtimes_{R}S(0)\rightarrow 0$$
For each $m\geq 0$, we have the following commutative diagram:
\begin{equation*}
\begin{tikzcd}
 0 \arrow{r} & \left(Z\boxtimes_{R}S(0)\right)_{m} \arrow{r}{\left(\mu\boxtimes_{R}S(0)\right)_{m}} \arrow{d}{\cong} & [3.5em]
 \left(W\boxtimes_{R}S(0)\right)_{m} \arrow{r} \arrow{d}{\cong} &
 \left(\Coker(\mu)\boxtimes_{R}S(0)\right)_{m} \arrow{r} \arrow{d}{\cong} & 0
  \\
 0 \arrow{r} & \bigoplus_{(f,g)\in \mathcal{S}_{m}''}Z_{k} \arrow{r}{\bigoplus_{(f,g)\in \mathcal{S}_{m}''}\mu_{k}} &
 \bigoplus_{(f,g)\in \mathcal{S}_{m}''}W_{k} \arrow{r} &
 \bigoplus_{(f,g)\in \mathcal{S}_{m}''}\Coker(\mu_{k}) \arrow{r} & 0
\end{tikzcd}
\end{equation*}

\noindent
Now the second row is exact and $\bigoplus_{(f,g)\in \mathcal{S}_{m}''}\Coker(\mu_{k})$ is projective, so it follows that the first row is exact and $\left(\Coker(\mu)\boxtimes_{R}S(0)\right)_{m}$ is projective. As a result, the previous sequence is exact and $\Coker(\mu)\boxtimes_{R}S(0)$ is an $R$-complex of projective modules. This means that $\mu\boxtimes_{R}S(0)$ is a cofibration, so $\zeta'$ is a cofibration.

If $\rho:M\rightarrow N$ is a trivial fibration in $\mathpzc{s}\mathcal{M}\mathpzc{od}(R)$, then we are done. If $R^{(\iota)}:R^{(U)}\rightarrow R^{(V)}$ is a trivial cofibration in $\mathpzc{s}\mathcal{M}\mathpzc{od}(R)$, i.e. $\mu:Z\rightarrow W$ is a trivial cofibration in $\mathcal{C}_{\geq 0}(R)$, then we need to further show that $\zeta'$ is a weak equivalence, i.e. a quasi-isomorphism. Using the fact $S(0)=\Sigma^{0}R=R$ and Corollary \ref{5.10}, we have the following commutative diagram for every $i\geq 0$:
\begin{equation*}
\begin{tikzcd}[column sep=6em,row sep=2em]
 H_{i}\left(Z\boxtimes_{R}S(0)\right) \arrow{r}{H_{i}\left(\mu\boxtimes_{R}S(0)\right)} \arrow{d}[swap]{\cong}
  & H_{i}\left(W\boxtimes_{R}S(0)\right) \arrow{d}{\cong}
  \\
 H_{i}\left(Z\otimes_{R}S(0)\right) \arrow{r}{H_{i}\left(\mu\otimes_{R}S(0)\right)} \ar[equal]{d}
  & H_{i}\left(W\otimes_{R}S(0)\right) \ar[equal]{d}
  \\
  H_{i}\left(Z\otimes_{R}R\right) \arrow{r}{H_{i}\left(\mu\otimes_{R}R\right)} \arrow{d}[swap]{\cong}
  & H_{i}\left(W\otimes_{R}R\right) \arrow{d}{\cong}
  \\
  H_{i}(Z) \arrow{r}{H_{i}(\mu)}
  & H_{i}(W)
\end{tikzcd}
\end{equation*}

\noindent
As $\mu$ is a weak equivalence, i.e. a quasi-isomorphism, $H_{i}(\mu):H_{i}(Z)\rightarrow H_{i}(W)$ is an isomorphism for every $i\geq 0$, so the above diagram shows that $H_{i}\left(\mu\boxtimes_{R}S(0)\right)$ is an isomorphism for every $i\geq 0$, whence $\mu\boxtimes_{R}S(0)$ is a quasi-isomorphism. Thus $\zeta'$ is a quasi-isomorphism.
\end{proof}

\begin{corollary} \label{6.2}
Let $R$ be a commutative ring, $M$ a simplicial $R$-module, and $\iota:U\rightarrow V$ a morphism of simplicial sets. If $R^{(\iota)}:R^{(U)}\rightarrow R^{(V)}$ is a (trivial) cofibration in $\mathpzc{s}\mathcal{M}\mathpzc{od}(R)$, then $\Map_{\mathpzc{s}\mathcal{S}\mathpzc{et}}(\iota,M):\Map_{\mathpzc{s}\mathcal{S}\mathpzc{et}}(V,M)\rightarrow \Map_{\mathpzc{s}\mathcal{S}\mathpzc{et}}(U,M)$ is a (trivial) fibration in $\mathpzc{s}\mathcal{M}\mathpzc{od}(R)$.
\end{corollary}

\begin{proof}
It is clear from Theorem \ref{4.15} that the unique morphism $f:M\rightarrow 0$ is a fibration. Consider the commutative diagram
\begin{equation*}
\begin{tikzcd}[column sep=5.5em,row sep=3em]
  \Map_{\mathpzc{s}\mathcal{S}\mathpzc{et}}(V,M) \arrow{r}{\Map_{\mathpzc{s}\mathcal{S}\mathpzc{et}}(V,f)} \arrow{d}[swap]{\Map_{\mathpzc{s}\mathcal{S}\mathpzc{et}}(\iota,M)}
  & \Map_{\mathpzc{s}\mathcal{S}\mathpzc{et}}(V,0)=0 \arrow{d}{\Map_{\mathpzc{s}\mathcal{S}\mathpzc{et}}(\iota,0)=0}
  \\
  \Map_{\mathpzc{s}\mathcal{S}\mathpzc{et}}(U,M) \arrow{r}{\Map_{\mathpzc{s}\mathcal{S}\mathpzc{et}}(U,f)}
  & \Map_{\mathpzc{s}\mathcal{S}\mathpzc{et}}(U,0)=0
\end{tikzcd}
\end{equation*}

\noindent
and form the following pullback diagram:
\begin{equation*}
  \begin{tikzcd}
  \Map_{\mathpzc{s}\mathcal{S}\mathpzc{et}}(V,M) \arrow{dr}{\eta} \arrow[bend right, swap]{ddr}{\Map_{\mathpzc{s}\mathcal{S}\mathpzc{et}}(\iota,M)}
  \arrow[bend left]{drr}{\Map_{\mathpzc{s}\mathcal{S}\mathpzc{et}}(V,f)} & &
  \\
  & \Map_{\mathpzc{s}\mathcal{S}\mathpzc{et}}(U,M)\bigsqcap_{0} 0 \arrow{r} \arrow{d}[swap]{\cong} & 0 \arrow{d}
  \\
  & \Map_{\mathpzc{s}\mathcal{S}\mathpzc{et}}(U,M) \arrow{r} & 0
\end{tikzcd}
\end{equation*}

\noindent
If $R^{(\iota)}:R^{(U)}\rightarrow R^{(V)}$ is a (trivial) cofibration in $\mathpzc{s}\mathcal{M}\mathpzc{od}(R)$, then Lemma \ref{6.1} implies that $\eta$ is a (trivial) fibration in $\mathpzc{s}\mathcal{M}\mathpzc{od}(R)$, so by the above diagram, we deduce that $\Map_{\mathpzc{s}\mathcal{S}\mathpzc{et}}(\iota,M)$ is a (trivial) fibration in $\mathpzc{s}\mathcal{M}\mathpzc{od}(R)$.
\end{proof}

The following lemma is dual to Lemma \ref{6.1}:

\begin{lemma} \label{6.3}
Let $R$ be a commutative ring, $\varrho:M\rightarrow N$ a cofibration in $\mathpzc{s}\mathcal{M}\mathpzc{od}(R)$, and $\iota:U\rightarrow V$ a morphism of simplicial sets. Consider the commutative diagram
\begin{equation*}
\begin{tikzcd}
  M^{(U)} \arrow{r}{\varrho^{(U)}} \arrow{d}[swap]{M^{(\iota)}} & N^{(U)} \arrow{d}{N^{(\iota)}}
  \\
  M^{(V)} \arrow{r}{\varrho^{(V)}} & N^{(V)}
\end{tikzcd}
\end{equation*}

\noindent
and form the following pushout diagram:
\begin{equation*}
  \begin{tikzcd}
  M^{(U)} \arrow{r}{\varrho^{(U)}} \arrow{d}[swap]{M^{(\iota)}} & N^{(U)} \arrow{d}{r} \arrow[bend left]{ddr}{N^{(\iota)}} &
  \\
  M^{(V)} \arrow{r}{s} \arrow[bend right, swap]{drr}{\varrho^{(V)}} & M^{(V)}\bigsqcup_{M^{(U)}}N^{(U)} \arrow{dr}{\zeta} &
  \\
  & & N^{(V)}
\end{tikzcd}
\end{equation*}

\noindent
If $R^{(\iota)}:R^{(U)}\rightarrow R^{(V)}$ is a cofibration in $\mathpzc{s}\mathcal{M}\mathpzc{od}(R)$, then $\zeta$ is a cofibration in $\mathpzc{s}\mathcal{M}\mathpzc{od}(R)$ which is trivial if either $\varrho$ or $R^{(\iota)}$ is.
\end{lemma}

\begin{proof}
Let $\rho:K\rightarrow L$ be a morphism of simplicial $R$-modules. Consider a commutative diagram as follows:
\begin{equation*}
\begin{tikzcd}
  M^{(V)}\bigsqcup_{M^{(U)}}N^{(U)} \arrow{r}{p} \arrow{d}[swap]{\zeta} & K \arrow{d}{\rho}
  \\
  N^{(V)} \arrow{r}{q} & L
\end{tikzcd}
\end{equation*}

\noindent
Considering the commutative diagram
\begin{equation*}
\begin{tikzcd}[column sep=5.5em,row sep=3em]
  \Map_{\mathpzc{s}\mathcal{S}\mathpzc{et}}(V,K) \arrow{r}{\Map_{\mathpzc{s}\mathcal{S}\mathpzc{et}}(V,\rho)} \arrow{d}[swap]{\Map_{\mathpzc{s}\mathcal{S}\mathpzc{et}}(\iota,K)}
  & \Map_{\mathpzc{s}\mathcal{S}\mathpzc{et}}(V,L) \arrow{d}{\Map_{\mathpzc{s}\mathcal{S}\mathpzc{et}}(\iota,L)}
  \\
  \Map_{\mathpzc{s}\mathcal{S}\mathpzc{et}}(U,K) \arrow{r}{\Map_{\mathpzc{s}\mathcal{S}\mathpzc{et}}(U,\rho)}
  & \Map_{\mathpzc{s}\mathcal{S}\mathpzc{et}}(U,L)
\end{tikzcd}
\end{equation*}

\noindent
there is a unique morphism $\eta:\Map_{\mathpzc{s}\mathcal{S}\mathpzc{et}}(V,K)\rightarrow \Map_{\mathpzc{s}\mathcal{S}\mathpzc{et}}(U,K)\textstyle\bigsqcap_{\Map_{\mathpzc{s}\mathcal{S}\mathpzc{et}}(U,L)} \Map_{\mathpzc{s}\mathcal{S}\mathpzc{et}}(V,L)$ that makes the following pullback diagram commutative:
\begin{equation*}
  \begin{tikzcd}
  \Map_{\mathpzc{s}\mathcal{S}\mathpzc{et}}(V,K) \arrow{dr}{\eta} \arrow[bend right, swap]{ddr}{\Map_{\mathpzc{s}\mathcal{S}\mathpzc{et}}(\iota,K)}
  \arrow[bend left]{drr}{\Map_{\mathpzc{s}\mathcal{S}\mathpzc{et}}(V,\rho)} & &
  \\
  & \Map_{\mathpzc{s}\mathcal{S}\mathpzc{et}}(U,K)\bigsqcap_{\Map_{\mathpzc{s}\mathcal{S}\mathpzc{et}}(U,L)} \Map_{\mathpzc{s}\mathcal{S}\mathpzc{et}}(V,L) \arrow{r} \arrow{d} & \Map_{\mathpzc{s}\mathcal{S}\mathpzc{et}}(V,L) \arrow{d}{\Map_{\mathpzc{s}\mathcal{S}\mathpzc{et}}(\iota,L)}
  \\ [1em]
  & \Map_{\mathpzc{s}\mathcal{S}\mathpzc{et}}(U,K) \arrow{r}{\Map_{\mathpzc{s}\mathcal{S}\mathpzc{et}}(U,\rho)} & \Map_{\mathpzc{s}\mathcal{S}\mathpzc{et}}(U,L)
\end{tikzcd}
\end{equation*}

\noindent
On the other hand, considering the natural bijections
$$\phi_{NVL}: \Mor_{\mathpzc{s}\mathcal{M}\mathpzc{od}(R)}\left(N,\Map_{\mathpzc{s}\mathcal{S}\mathpzc{et}}(V,L)\right) \rightarrow \Mor_{\mathpzc{s}\mathcal{M}\mathpzc{od}(R)}\left(N^{(V)},L\right)$$
and
$$\phi_{NUK}: \Mor_{\mathpzc{s}\mathcal{M}\mathpzc{od}(R)}\left(N,\Map_{\mathpzc{s}\mathcal{S}\mathpzc{et}}(U,K)\right) \rightarrow \Mor_{\mathpzc{s}\mathcal{M}\mathpzc{od}(R)}\left(N^{(U)},K\right)$$
from Remark \ref{4.4}, we get the following commutative diagram:
\begin{equation*}
\begin{tikzcd}[column sep=5.5em,row sep=3em]
  N \arrow{r}{\phi_{NVL}^{-1}(q)} \arrow{d}[swap]{\phi_{NUK}^{-1}(pr)}
  & \Map_{\mathpzc{s}\mathcal{S}\mathpzc{et}}(V,L) \arrow{d}{\Map_{\mathpzc{s}\mathcal{S}\mathpzc{et}}(\iota,L)}
  \\
  \Map_{\mathpzc{s}\mathcal{S}\mathpzc{et}}(U,K) \arrow{r}{\Map_{\mathpzc{s}\mathcal{S}\mathpzc{et}}(U,\rho)}
  & \Map_{\mathpzc{s}\mathcal{S}\mathpzc{et}}(U,L)
\end{tikzcd}
\end{equation*}

\noindent
Hence there is a unique morphism $\theta:N \rightarrow \Map_{\mathpzc{s}\mathcal{S}\mathpzc{et}}(U,K)\bigsqcap_{\Map_{\mathpzc{s}\mathcal{S}\mathpzc{et}}(U,L)} \Map_{\mathpzc{s}\mathcal{S}\mathpzc{et}}(V,L)$ that makes the following diagram commutative:
\begin{equation*}
  \begin{tikzcd}
  N \arrow{dr}{\theta} \arrow[bend right, swap]{ddr}{\phi_{NUK}^{-1}(pr)} \arrow[bend left]{drr}{\phi_{NVL}^{-1}(q)} & &
  \\
  & \Map_{\mathpzc{s}\mathcal{S}\mathpzc{et}}(U,K)\bigsqcap_{\Map_{\mathpzc{s}\mathcal{S}\mathpzc{et}}(U,L)} \Map_{\mathpzc{s}\mathcal{S}\mathpzc{et}}(V,L) \arrow{r} \arrow{d} & \Map_{\mathpzc{s}\mathcal{S}\mathpzc{et}}(V,L) \arrow{d}{\Map_{\mathpzc{s}\mathcal{S}\mathpzc{et}}(\iota,L)}
  \\ [1em]
  & \Map_{\mathpzc{s}\mathcal{S}\mathpzc{et}}(U,K) \arrow{r}{\Map_{\mathpzc{s}\mathcal{S}\mathpzc{et}}(U,\rho)} & \Map_{\mathpzc{s}\mathcal{S}\mathpzc{et}}(U,L)
\end{tikzcd}
\end{equation*}

\noindent
Then considering the natural bijection
$$\phi_{MVK}: \Mor_{\mathpzc{s}\mathcal{M}\mathpzc{od}(R)}\left(M,\Map_{\mathpzc{s}\mathcal{S}\mathpzc{et}}(V,K)\right) \rightarrow \Mor_{\mathpzc{s}\mathcal{M}\mathpzc{od}(R)}\left(M^{(V)},K\right)$$
we get the following commutative diagram:
\begin{equation*}
\begin{tikzcd}
  M \arrow{r}{\phi_{MVK}^{-1}(ps)} \arrow{d}[swap]{\varrho}
  & \Map_{\mathpzc{s}\mathcal{S}\mathpzc{et}}(V,K) \arrow{d}{\eta}
  \\
  N \arrow{r}{\theta}
  & \Map_{\mathpzc{s}\mathcal{S}\mathpzc{et}}(U,K)\textstyle\bigsqcap_{\Map_{\mathpzc{s}\mathcal{S}\mathpzc{et}}(U,L)} \Map_{\mathpzc{s}\mathcal{S}\mathpzc{et}}(V,L)
\end{tikzcd}
\end{equation*}

\noindent
If $\tau:N\rightarrow \Map_{\mathpzc{s}\mathcal{S}\mathpzc{et}}(V,K)$ is a morphism of simplicial $R$-modules that makes the diagram
\begin{equation*}
\begin{tikzcd}[column sep=3em,row sep=3.5em]
  M \arrow{r}{\phi_{MVK}^{-1}(ps)} \arrow{d}[swap]{\varrho}
  & \Map_{\mathpzc{s}\mathcal{S}\mathpzc{et}}(V,K) \arrow{d}{\eta}
  \\
  N \arrow{r}{\theta} \arrow{ru}{\tau}
  & \Map_{\mathpzc{s}\mathcal{S}\mathpzc{et}}(U,K)\textstyle\bigsqcap_{\Map_{\mathpzc{s}\mathcal{S}\mathpzc{et}}(U,L)} \Map_{\mathpzc{s}\mathcal{S}\mathpzc{et}}(V,L)
\end{tikzcd}
\end{equation*}

\noindent
commutative, then considering the natural bijection
$$\phi_{NVK}:\Mor_{\mathpzc{s}\mathcal{M}\mathpzc{od}(R)}\left(N,\Map_{\mathpzc{s}\mathcal{S}\mathpzc{et}}(V,K)\right)\rightarrow \Mor_{\mathpzc{s}\mathcal{M}\mathpzc{od}(R)}\left(N^{(V)},K\right)$$
we see that the following diagram is commutative:
\begin{equation*}
\begin{tikzcd}[column sep=5em,row sep=3em]
  M^{(V)}\bigsqcup_{M^{(U)}}N^{(U)} \arrow{r}{p} \arrow{d}[swap]{\zeta} & K \arrow{d}{\rho}
  \\
  N^{(V)} \arrow{r}{q} \arrow{ru}{\phi_{NVK}(\tau)} & L
\end{tikzcd}
\end{equation*}

\noindent
Now if $\varrho$ and $R^{(\iota)}$ are cofibrations, and $\rho$ is a trivial fibration, then Lemma \ref{6.1} implies that $\eta$ is a trivial fibration, so $\tau$ exists, whence we conclude from the above diagram that $\zeta$ is a cofibration. If $\varrho$ is a trivial cofibration, $R^{(\iota)}$ is a cofibration, and $\rho$ is a fibration, then Lemma \ref{6.1} implies that $\eta$ is a fibration, so $\tau$ exists, whence we deduce from the above diagram that $\zeta$ is a trivial cofibration. If $\varrho$ is a cofibration, $R^{(\iota)}$ is a trivial cofibration, and $\rho$ is a fibration, then Lemma \ref{6.1} implies that $\eta$ is a trivial fibration, so $\tau$ exists, whence we infer from the above diagram that $\zeta$ is a trivial cofibration.
\end{proof}

\begin{corollary} \label{6.4}
Let $R$ be a commutative ring, $M$ a cofibrant simplicial $R$-module, and $\iota:U\rightarrow V$ a morphism of simplicial sets. If $R^{(\iota)}:R^{(U)}\rightarrow R^{(V)}$ is a (trivial) cofibration in $\mathpzc{s}\mathcal{M}\mathpzc{od}(R)$, then $M^{(\iota)}:M^{(U)}\rightarrow M^{(V)}$ is a (trivial) cofibration in $\mathpzc{s}\mathcal{M}\mathpzc{od}(R)$.
\end{corollary}

\begin{proof}
Since $M$ is cofibrant, the unique morphism $f:0\rightarrow M$ is a cofibration. Consider the commutative diagram
\begin{equation*}
\begin{tikzcd}
  0=0^{(U)} \arrow{r}{f^{(U)}} \arrow{d}[swap]{0=0^{(\iota)}} & M^{(U)} \arrow{d}{M^{(\iota)}}
  \\
  0=0^{(V)} \arrow{r}{f^{(V)}} & M^{(V)}
\end{tikzcd}
\end{equation*}

\noindent
and form the following pushout diagram:
\begin{equation*}
  \begin{tikzcd}
  0 \arrow{r} \arrow{d} & M^{(U)} \arrow{d}{\cong} \arrow[bend left]{ddr}{M^{(\iota)}} &
  \\
  0 \arrow{r} \arrow[bend right, swap]{drr}{f^{(V)}} & 0\bigsqcup_{0}M^{(U)} \arrow{dr}{\zeta} &
  \\
  & & M^{(V)}
\end{tikzcd}
\end{equation*}

\noindent
If $R^{(\iota)}:R^{(U)}\rightarrow R^{(V)}$ is a (trivial) cofibration in $\mathpzc{s}\mathcal{M}\mathpzc{od}(R)$, then Lemma \ref{6.3} implies that $\zeta$ is a (trivial) cofibration in $\mathpzc{s}\mathcal{M}\mathpzc{od}(R)$, so by the above diagram, we deduce that $M^{(\iota)}$ is a (trivial) cofibration in $\mathpzc{s}\mathcal{M}\mathpzc{od}(R)$.
\end{proof}

Now we can construct specific cylinder and path objects in $\mathpzc{s}\mathcal{M}\mathpzc{od}(R)$. These objects are introduced in \cite[Corollary 4.14]{GS} but the proof outlined there relies heavily on the model structure of simplicial sets. Here we provide a proof based on Lemmas \ref{6.1} and \ref{6.3} whose proofs are independent of the model structure of simplicial sets.

\begin{proposition} \label{6.5}
Let $R$ be a commutative ring, and $M$ a simplicial $R$-module. Then the following assertions hold:
\begin{enumerate}
\item[(i)] If $M$ is cofibrant in $\mathpzc{s}\mathcal{M}\mathpzc{od}(R)$, then $M^{(\Delta^{1})}$ is a cylinder object for $M$.
\item[(ii)] $\Map_{\mathpzc{s}\mathcal{S}\mathpzc{et}}\left(\Delta^{1},M\right)$ is a path object for $M$.
\end{enumerate}
\end{proposition}

\begin{proof}
Let $\varepsilon:\Delta^{1}\rightarrow \Delta^{0}$ and $\delta:\partial\Delta^{1}\rightarrow \Delta^{0}$ be the unique morphisms, and $\iota:\partial\Delta^{1}\rightarrow \Delta^{1}$ the inclusion morphism. Then we have the following commutative diagram of simplicial sets:
\begin{equation} \label{eq:6.5.1}
\begin{tikzcd}
  \partial\Delta^{1} \arrow{r}{\delta} \arrow{d}[swap]{\iota} & \Delta^{0}
  \\
  \Delta^{1} \arrow{ru}[swap]{\varepsilon} &
\end{tikzcd}
\end{equation}

\noindent
We note that $\Delta_{n}^{0}=\{(\underbrace{0,...,0}_{n+1 \textrm{ times}})\}$ and $\Delta_{n}^{1}=\{(\alpha_{0},\alpha_{1},\ldots,\alpha_{n})\suchthat 0\leq \alpha_{0}\leq \alpha_{1}\leq \ldots \leq \alpha_{n}\leq 1\}$ for every $n\geq 0$. Define a morphism $\zeta=:\Delta^{0}\rightarrow \Delta^{1}$ of simplicial sets by setting $\zeta_{n}(0,...,0)=(0,...,0)$ for every $n\geq 0$. Then it is clear that $\varepsilon\zeta=1^{\Delta^{0}}$. We show that the induced morphism $R^{(\zeta)}:R^{(\Delta^{0})}\rightarrow R^{(\Delta^{1})}$ is a trivial cofibration in $\mathpzc{s}\mathcal{M}\mathpzc{od}(R)$, i.e. $\Nor\left(R^{(\zeta)}\right):\Nor\left(R^{(\Delta^{0})}\right) \rightarrow \Nor\left(R^{(\Delta^{1})}\right)$ is a trivial cofibration in $\mathcal{C}_{\geq 0}(R)$. To this end, we analyze this morphism more closely. We have observed in the proof of Lemma \ref{4.8} that $\Nor\left(R^{(\Delta^{0})}\right)$ is the $R$-complex
$$0\rightarrow R \rightarrow 0$$
with $R$ in degree $0$. On the other hand, we have $\Nor\left(R^{(\Delta^{1})}\right)_{0}=R^{(\Delta_{0}^{1})}=R^{(\{0,1\})}=R^{2}$ with a basis $\{e_{1},e_{2}\}$ corresponding to $\Delta_{0}^{1}=\{0,1\}$, and $R^{(\Delta_{1}^{1})}=R^{(\{(0,0),(0,1),(1,1)\})}=R^{3}$ with a basis $\{e_{1}',e_{2}',e_{3}'\}$ corresponding to $\Delta_{1}^{1}=\left\{(0,0),(0,1),(1,1)\right\}$. We describe the morphisms
$$R^{3}=R^{(\Delta_{1}^{1})} \xrightarrow{d_{1,0}^{R^{(\Delta^{1})}}=R^{\left(d_{1,0}^{\Delta^{1}}\right)}} R^{(\Delta_{0}^{1})}=R^{2}$$
and
$$R^{3}=R^{(\Delta_{1}^{1})} \xrightarrow{d_{1,1}^{R^{(\Delta^{1})}}=R^{\left(d_{1,1}^{\Delta^{1}}\right)}} R^{(\Delta_{0}^{1})}=R^{2}$$
as follows. We note that $e_{1}'$ corresponds to $(0,0)$, so since $d_{1,0}^{\Delta^{1}}(0,0)=0$ and $0$ corresponds to $e_{1}$, we get $d_{1,0}^{R^{(\Delta^{1})}}(e_{1}')=e_{1}$. Similarly, we have:
\[
d_{1,0}^{\Delta^{1}}:
 \begin{dcases}
  (0,0)\mapsto 0 \\
  (0,1)\mapsto 1 \\
  (1,1)\mapsto 1
 \end{dcases}
\quad \text{and} \quad
d_{1,1}^{\Delta^{1}}:
 \begin{dcases}
  (0,0)\mapsto 0 \\
  (0,1)\mapsto 0 \\
  (1,1)\mapsto 1
 \end{dcases}
\]

\noindent
Hence we get:
\[
d_{1,0}^{R^{(\Delta^{1})}}:
 \begin{dcases}
  e_{1}'\mapsto e_{1} \\
  e_{2}'\mapsto e_{2} \\
  e_{3}'\mapsto e_{2}
 \end{dcases}
\quad \text{and} \quad
d_{1,1}^{R^{(\Delta^{1})}}:
 \begin{dcases}
  e_{1}'\mapsto e_{1} \\
  e_{2}'\mapsto e_{1} \\
  e_{3}'\mapsto e_{2}
 \end{dcases}
\]

\noindent
It follows that $d_{1,0}^{R^{(\Delta^{1})}}(a,b,c)=(a,b+c)$ and $d_{1,1}^{R^{(\Delta^{1})}}(a,b,c)=(a+b,c)$ for every $(a,b,c)\in R^{3}$. Thus we have:
$$\Nor\left(R^{(\Delta^{1})}\right)_{1} = \Ker\left(d_{1,0}^{R^{(\Delta^{1})}}\right)= \left\{(a,b,c)\in R^{3} \suchthat a=0 \textrm{ and } b=-c\right\} = \left\{(0,-a,a)\suchthat a\in R\right\} \cong R$$
Moreover, in light of the above isomorphism, $\partial_{1}=\partial_{1}^{\Nor\left(R^{(\Delta^{1})}\right)}:R\rightarrow R^{2}$ is given by $\partial_{1}(a)=-d_{1,1}^{R^{(\Delta^{1})}}(0,-a,a)=(a,-a)$ for every $a\in R$. As we saw in the proof of Lemma \ref{4.8}, if $n\geq 2$, then $\Nor\left(R^{(\Delta^{1})}\right)_{n}=0$. As a consequence, $\Nor\left(R^{(\Delta^{1})}\right)$ is the $R$-complex
$$0\rightarrow R \xrightarrow{\partial_{1}} R^{2}\rightarrow 0$$
with $R$ in degree $1$ and $R^{2}$ in degree $0$. Now consider $\Nor\left(R^{(\zeta)}\right):\Nor\left(R^{(\Delta^{0})}\right)\rightarrow \Nor\left(R^{(\Delta^{1})}\right)$. We note that $\zeta_{0}:\Delta_{0}^{0}\rightarrow \Delta_{0}^{1}$ is given by $\zeta_{0}(0)=(0)$, so $\Nor\left(R^{(\zeta)}\right)_{0}=R^{(\zeta_{0})}:R\rightarrow R^{2}$ is given by $R^{(\zeta_{0})}(a)=(a,0)$ for every $a\in R$. Indeed, $\Nor\left(R^{(\zeta)}\right)$ is as follows:
\begin{equation*}
\begin{tikzcd}
  0 \arrow{r} & 0 \arrow{r} \arrow{d} & R \arrow{r} \arrow{d}{R^{(\zeta_{0})}} & 0
  \\
  0 \arrow{r} & R \arrow{r}{\partial_{1}} & R^{2} \arrow{r} & 0
\end{tikzcd}
\end{equation*}

\noindent
It is now clear that $\Nor\left(R^{(\zeta)}\right)$ is injective whose cokernel is an $R$-complex of projective modules, so $\Nor\left(R^{(\zeta)}\right)$ is a cofibration. Furthermore, by Lemma \ref{4.8}, $\Nor\left(R^{(\zeta)}\right)$ is a homotopy equivalence, hence a quasi-isomorphism, i.e. a weak equivalence. Therefore, $\Nor\left(R^{(\zeta)}\right)$ is a trivial cofibration in $\mathcal{C}_{\geq 0}(R)$.

We next show that $R^{(\iota)}:R^{(\partial\Delta^{1})}\rightarrow R^{(\Delta^{1})}$ is a cofibration in $\mathpzc{s}\mathcal{M}\mathpzc{od}(R)$, i.e. $\Nor\left(R^{(\iota)}\right):\Nor\left(R^{(\partial\Delta^{1})}\right)\rightarrow \Nor\left(R^{(\Delta^{1})}\right)$ is a cofibration in $\mathcal{C}_{\geq 0}(R)$. To this end, we analyze this morphism more closely. We have $\Nor\left(R^{(\partial\Delta^{1})}\right)_{0}=R^{(\partial\Delta_{0}^{1})}=R^{(\{0,1\})}=R^{2}$. If $n\geq 1$, then every element of $\partial\Delta_{n}^{1}$ is degenerate, so $\Nor\left(R^{(\partial\Delta^{1})}\right)_{n}=0$. As a consequence, $\Nor\left(R^{(\partial\Delta^{1})}\right)$ is the $R$-complex
$$0\rightarrow R^{2}\rightarrow 0$$
with $R^{2}$ in degree $0$. Now consider $\Nor\left(R^{(\iota)}\right):\Nor\left(R^{(\partial\Delta^{1})}\right)\rightarrow \Nor\left(R^{(\Delta^{1})}\right)$. We note that $\partial\Delta_{0}^{1}=\Delta_{0}^{1}=\{0,1\}$, so $\iota_{0}:\partial\Delta_{0}^{1}\rightarrow \Delta_{0}^{1}$ is the identity map, thereby $R^{(\iota_{0})}=1^{R^{2}}:R^{2}\rightarrow R^{2}$. Indeed, $\Nor\left(R^{(\iota)}\right)$ is as follows:
\begin{equation*}
\begin{tikzcd}
  0 \arrow{r} & 0 \arrow{r} \arrow{d} & R^{2} \arrow{r} \arrow{d}{R^{(\iota_{0})}} & 0
  \\
  0 \arrow{r} & R \arrow{r}{\partial_{1}} & R^{2} \arrow{r} & 0
\end{tikzcd}
\end{equation*}

\noindent
It is now clear that $\Nor\left(R^{(\iota)}\right)$ is injective whose cokernel is an $R$-complex of projective modules, so it is a cofibration. Now we have:

(i): The diagram \eqref{eq:6.5.1} yields the following commutative diagram of simplicial $R$-modules:
\begin{equation*}
\begin{tikzcd}
  M^{(\partial\Delta^{1})} \arrow{r}{M^{(\delta)}} \arrow{d}[swap]{M^{(\iota)}} & M^{(\Delta^{0})}
  \\
  M^{(\Delta^{1})} \arrow{ru}[swap]{M^{(\varepsilon)}} &
\end{tikzcd}
\end{equation*}

\noindent
From the relation $\varepsilon\zeta=1^{\Delta^{0}}$, we get $M^{(\varepsilon)}M^{(\zeta)}=1^{M^{(\Delta^{0})}}$. Since $R^{(\zeta)}:R^{(\Delta^{0})}\rightarrow R^{(\Delta^{1})}$ is a trivial cofibration in $\mathpzc{s}\mathcal{M}\mathpzc{od}(R)$, Corollary \ref{6.4} shows that $M^{(\zeta)}:M^{(\Delta^{0})}\rightarrow M^{(\Delta^{1})}$ is a trivial cofibration in $\mathpzc{s}\mathcal{M}\mathpzc{od}(R)$, hence in particular a weak equivalence. Thus $M^{(\varepsilon)}$ is a weak equivalence in $\mathpzc{s}\mathcal{M}\mathpzc{od}(R)$. On the other hand, $R^{(\iota)}:R^{(\partial\Delta^{1})}\rightarrow R^{(\Delta^{1})}$ is a cofibration in $\mathpzc{s}\mathcal{M}\mathpzc{od}(R)$, so Corollary \ref{6.4} implies that $M^{(\iota)}:M^{(\partial\Delta^{1})}\rightarrow M^{(\Delta^{1})}$ is a cofibration in $\mathpzc{s}\mathcal{M}\mathpzc{od}(R)$. We finally note that $M^{(\Delta^{0})}\cong M$, and also we have $M^{(\partial\Delta^{1})}\cong M^{(\Delta^{0}\sqcup \Delta^{0})}\cong M^{(\Delta^{0})}\oplus M^{(\Delta^{0})}\cong M\oplus M$ by Example \ref{4.3}. Taking these isomorphisms into account, the above diagram yields a commutative diagram
\begin{equation*}
\begin{tikzcd}
  M\oplus M \arrow{r}{\nabla^{M}} \arrow{d}[swap]{\kappa^{M}} & M
  \\
  M^{\left(\Delta^{1}\right)} \arrow{ru}[swap]{\xi^{M}} &
\end{tikzcd}
\end{equation*}

\noindent
in which $\kappa^{M}$ is a cofibration and $\xi^{M}$ is a weak equivalence. This means that $M^{(\Delta^{1})}$ is a cylinder object for $M$.

(ii): The diagram \eqref{eq:6.5.1} yields the following commutative diagram of simplicial $R$-modules:
\begin{equation*}
\begin{tikzcd}[column sep=6em,row sep=2em]
  \Map_{\mathpzc{s}\mathcal{S}\mathpzc{et}}\left(\Delta^{0},M\right) \arrow{r}{\Map_{\mathpzc{s}\mathcal{S}\mathpzc{et}}(\delta,M)} \arrow{d}[swap]{\Map_{\mathpzc{s}\mathcal{S}\mathpzc{et}}(\varepsilon,M)} & \Map_{\mathpzc{s}\mathcal{S}\mathpzc{et}}\left(\partial\Delta^{1},M\right)
  \\
  \Map_{\mathpzc{s}\mathcal{S}\mathpzc{et}}\left(\Delta^{1},M\right) \arrow{ru}[swap]{\Map_{\mathpzc{s}\mathcal{S}\mathpzc{et}}(\iota,M)} &
\end{tikzcd}
\end{equation*}

\noindent
From the relation $\varepsilon\zeta=1^{\Delta^{0}}$, we get $\Map_{\mathpzc{s}\mathcal{S}\mathpzc{et}}(\zeta,M)\Map_{\mathpzc{s}\mathcal{S}\mathpzc{et}}(\varepsilon,M)= 1^{\Map_{\mathpzc{s}\mathcal{S}\mathpzc{et}}\left(\Delta^{0},M\right)}$. Since $R^{(\zeta)}:R^{(\Delta^{0})}\rightarrow R^{(\Delta^{1})}$ is a trivial cofibration in $\mathpzc{s}\mathcal{M}\mathpzc{od}(R)$, the Corollary \ref{6.2} shows that $\Map_{\mathpzc{s}\mathcal{S}\mathpzc{et}}(\zeta,M):\Map_{\mathpzc{s}\mathcal{S}\mathpzc{et}}\left(\Delta^{1},M\right)\rightarrow \Map_{\mathpzc{s}\mathcal{S}\mathpzc{et}}\left(\Delta^{0},M\right)$ is a trivial fibration in $\mathpzc{s}\mathcal{M}\mathpzc{od}(R)$, hence in particular a weak equivalence. Thus $\Map_{\mathpzc{s}\mathcal{S}\mathpzc{et}}(\varepsilon,M)$ is a weak equivalence in $\mathpzc{s}\mathcal{M}\mathpzc{od}(R)$. On the other hand, $R^{(\iota)}:R^{(\partial\Delta^{1})}\rightarrow R^{(\Delta^{1})}$ is a cofibration in $\mathpzc{s}\mathcal{M}\mathpzc{od}(R)$, so Corollary \ref{6.2} implies that $\Map_{\mathpzc{s}\mathcal{S}\mathpzc{et}}(\iota,M):\Map_{\mathpzc{s}\mathcal{S}\mathpzc{et}}\left(\Delta^{1},M\right)\rightarrow \Map_{\mathpzc{s}\mathcal{S}\mathpzc{et}}\left(\partial\Delta^{1},M\right)$ is a fibration in $\mathpzc{s}\mathcal{M}\mathpzc{od}(R)$. We finally note that $\Map_{\mathpzc{s}\mathcal{S}\mathpzc{et}}\left(\Delta^{0},M\right)\cong M$, and also we have $\Map_{\mathpzc{s}\mathcal{S}\mathpzc{et}}\left(\partial\Delta^{1},M\right)\cong \Map_{\mathpzc{s}\mathcal{S}\mathpzc{et}}\left(\Delta^{0}\sqcup \Delta^{0},M\right)\cong \Map_{\mathpzc{s}\mathcal{S}\mathpzc{et}}\left(\Delta^{0},M\right)\times \Map_{\mathpzc{s}\mathcal{S}\mathpzc{et}}\left(\Delta^{0},M\right) \cong M\times M$. Taking these isomorphisms into account, the above diagram yields a commutative diagram
\begin{equation*}
\begin{tikzcd}
  M \arrow{r}{\Delta^{M}} \arrow{d}[swap]{\zeta^{M}} & M\times M
  \\
  \Map_{\mathpzc{s}\mathcal{S}\mathpzc{et}}\left(\Delta^{1},M\right) \arrow{ru}[swap]{\nu^{M}} &
\end{tikzcd}
\end{equation*}

\noindent
in which $\nu^{M}$ is a fibration and $\zeta^{M}$ is a weak equivalence. This means that $\Map_{\mathpzc{s}\mathcal{S}\mathpzc{et}}\left(\Delta^{1},M\right)$ is a path object for $M$.
\end{proof}

We are now ready to prove the main theorem.

\begin{theorem} \label{6.6}
Let $R$ be a commutative ring, and $\sharp:\mathpzc{s}\mathcal{C}\mathpzc{om}\mathcal{A}\mathpzc{lg}(R)\rightarrow \mathpzc{s}\mathcal{M}\mathpzc{od}(R)$ the forgetful functor. Then $\mathpzc{s}\mathcal{C}\mathpzc{om}\mathcal{A}\mathpzc{lg}(R)$ is a cofibrantly generated model category with the model structure given as follows:
\begin{enumerate}
\item[(i)] Fibrations are morphisms $f:A\rightarrow B$ in $\mathpzc{s}\mathcal{C}\mathpzc{om}\mathcal{A}\mathpzc{lg}(R)$ such that $\Nor\left(f^{\sharp}\right)_{i}:\Nor\left(A^{\sharp}\right)_{i}\rightarrow \Nor\left(B^{\sharp}\right)_{i}$ is surjective for every $i\geq 1$.
\item[(ii)] Weak equivalences are morphisms $f:A\rightarrow B$ in $\mathpzc{s}\mathcal{C}\mathpzc{om}\mathcal{A}\mathpzc{lg}(R)$ such that $\Nor\left(f^{\sharp}\right):\Nor\left(A^{\sharp}\right)\rightarrow \Nor\left(B^{\sharp}\right)$ is a quasi-isomorphism in $\mathcal{C}_{\geq 0}(R)$.
\item[(iii)] Cofibration are morphisms in $\mathpzc{s}\mathcal{C}\mathpzc{om}\mathcal{A}\mathpzc{lg}(R)$ that have the left lifting property against trivial fibrations defined in (i) and (ii).
\end{enumerate}
\end{theorem}

\begin{proof}
By Theorems \ref{4.14} and \ref{4.15}, $(\DK,\Nor):\mathcal{C}_{\geq 0}(R) \leftrightarrows \mathpzc{s}\mathcal{M}\mathpzc{od}(R)$ is an adjoint equivalence of categories through which $\mathpzc{s}\mathcal{M}\mathpzc{od}(R)$ becomes a cofibrantly generated model category. If $\mathcal{X}$ and $\mathcal{Y}$ are the sets of morphisms in $\mathcal{C}_{\geq 0}(R)$ corresponding to its cofibrant generation, then we have observed that the domains of morphisms in $\mathcal{X}$ and $\mathcal{Y}$ are small with respect to $\Mor\left(\mathcal{C}_{\geq 0}(R)\right)$. It follows that the domains of morphisms in $\DK(\mathcal{X})$ and $\DK(\mathcal{Y})$ are small with respect to $\Mor\left(\mathpzc{s}\mathcal{M}\mathpzc{od}(R)\right)$. On the other hand, we have the adjoint pair of functors $\left(\Sym_{R}(-),\sharp\right):\mathcal{M}\mathpzc{od}(R)\leftrightarrows \mathcal{C}\mathpzc{om}\mathcal{A}\mathpzc{lg}(R)$ that extends degreewise to an adjoint pair of functors $\left(\Sym_{R}(-),\sharp\right):\mathpzc{s}\mathcal{M}\mathpzc{od}(R)\leftrightarrows \mathpzc{s}\mathcal{C}\mathpzc{om}\mathcal{A}\mathpzc{lg}(R)$. Also, $\mathcal{C}\mathpzc{om}\mathcal{A}\mathpzc{lg}(R)$ is locally small and bicomplete, so $\mathpzc{s}\mathcal{C}\mathpzc{om}\mathcal{A}\mathpzc{lg}(R)$ is locally small and bicomplete. By the construction of filtered direct limits, we notice that the forgetful functor $\sharp:\mathcal{C}\mathpzc{om}\mathcal{A}\mathpzc{lg}(R)\rightarrow \mathcal{M}\mathpzc{od}(R)$ preserves filtered direct limits, so its extension $\sharp:\mathpzc{s}\mathcal{C}\mathpzc{om}\mathcal{A}\mathpzc{lg}(R)\rightarrow \mathpzc{s}\mathcal{M}\mathpzc{od}(R)$ also preserves filtered direct limits, hence sequential direct limits in particular. Therefore, Corollary \ref{2.11} implies that the domains of morphisms in $\Sym_{R}\left(\DK(\mathcal{X})\right)$ and $\Sym_{R}\left(\DK(\mathcal{Y})\right)$ are small with respect to $\Mor\left(\mathpzc{s}\mathcal{C}\mathpzc{om}\mathcal{A}\mathpzc{lg}(R)\right)$. In particular, $\Sym_{R}\left(\DK(\mathcal{X})\right)$ and $\Sym_{R}\left(\DK(\mathcal{Y})\right)$ permit the small object argument. Now we let a morphism $f:A\rightarrow B$ in $\mathpzc{s}\mathcal{C}\mathpzc{om}\mathcal{A}\mathpzc{lg}(R)$ be a fibration or a weak equivalence if $f^{\sharp}:A^{\sharp}\rightarrow B^{\sharp}$ is a fibration or a weak equivalence in $\mathpzc{s}\mathcal{M}\mathpzc{od}(R)$, respectively. As a consequence, $f:A\rightarrow B$ is a fibration in $\mathpzc{s}\mathcal{C}\mathpzc{om}\mathcal{A}\mathpzc{lg}(R)$ if and only if $\Nor\left(f^{\sharp}\right)_{i}:\Nor\left(A^{\sharp}\right)_{i}\rightarrow \Nor\left(B^{\sharp}\right)_{i}$ is surjective for every $i\geq 1$. Also, $f:A\rightarrow B$ is a weak equivalence in $\mathpzc{s}\mathcal{C}\mathpzc{om}\mathcal{A}\mathpzc{lg}(R)$ if and only if $\Nor\left(f^{\sharp}\right):\Nor\left(A^{\sharp}\right)\rightarrow \Nor\left(B^{\sharp}\right)$ is a quasi-isomorphism. Finally, we let a morphism $f:A\rightarrow B$ in $\mathpzc{s}\mathcal{C}\mathpzc{om}\mathcal{A}\mathpzc{lg}(R)$ be a cofibration if it has the left lifting property against trivial fibrations.

In order to apply Theorem \ref{2.9}, we need to establish the acyclicity condition, i.e. we need to show that any cofibration in $\mathpzc{s}\mathcal{C}\mathpzc{om}\mathcal{A}\mathpzc{lg}(R)$ that has the left lifting property against fibrations is also a weak equivalence. Let $g:A\rightarrow B$ be a cofibration in $\mathpzc{s}\mathcal{C}\mathpzc{om}\mathcal{A}\mathpzc{lg}(R)$ that has the left lifting property against fibrations. If $f:A\rightarrow 0$ is the unique morphism in $\mathpzc{s}\mathcal{C}\mathpzc{om}\mathcal{A}\mathpzc{lg}(R)$, then $f^{\sharp}:A^{\sharp}\rightarrow 0^{\sharp}$ is a fibration in $\mathpzc{s}\mathcal{M}\mathpzc{od}(R)$, so $f$ is a fibration in $\mathpzc{s}\mathcal{C}\mathpzc{om}\mathcal{A}\mathpzc{lg}(R)$. Therefore, the commutative diagram on the left can be completed to the commutative diagram on the right:
\[
 \begin{tikzcd}
  A \arrow{r}{1^{A}} \arrow{d}[swap]{g} & A \arrow{d}{f}
  \\
  B \arrow{r} & 0
\end{tikzcd}
\quad \Longrightarrow \quad
 \begin{tikzcd}
 A \arrow{r}{1^{A}} \arrow{d}[swap]{g} & A \arrow{d}{f}
  \\
  B \arrow{r} \arrow{ru}{h} & 0
\end{tikzcd}
\]

\noindent
In particular, $hg=1^{A}$. On the other hand, by the universal property of product, there is a unique morphism $l:B\rightarrow B\sqcap B$ that makes the following diagrams commutative:
\begin{equation*}
\begin{tikzcd}[column sep=4em,row sep=2em]
  & B \arrow{ld}[swap]{1^{B}} \arrow{rd}{gh} \arrow{d}{l} &
  \\
  B & B\sqcap B \arrow{l}[swap]{\pi_{1}^{B}} \arrow{r}{\pi_{2}^{B}} & B
\end{tikzcd}
\end{equation*}

\noindent
Consider the path object $\Path\left(B^{\sharp}\right)=\Map_{\mathpzc{s}\mathcal{S}\mathpzc{et}}\left(\Delta^{1},B^{\sharp}\right)$ in $\mathpzc{s}\mathcal{M}\mathpzc{od}(R)$. Then there is a commutative diagram
\begin{equation*}
  \begin{tikzcd}
  B^{\sharp} \arrow{r}{\Delta_{B^{\sharp}}}\arrow{d}[swap]{\zeta_{B^{\sharp}}}
  & B^{\sharp}\sqcap B^{\sharp}
  \\
  \Path\left(B^{\sharp}\right) \arrow{ru}[swap]{\nu_{B^{\sharp}}}
\end{tikzcd}
\end{equation*}

\noindent
in which $\nu_{B^{\sharp}}$ is a fibration and $\zeta_{B^{\sharp}}$ is a weak equivalence in $\mathpzc{s}\mathcal{M}\mathpzc{od}(R)$. But we observed in the proof of Proposition \ref{6.5} that this diagram is isomorphic to a diagram in which morphisms could be either morphisms of simplicial modules or simplicial commutative algebras depending on what structure we consider; See also Remark \ref{4.4}. Therefore, we can also have the commutative diagram
\begin{equation}\label{eq:6.6.1}
  \begin{tikzcd}
  B \arrow{r}{\Delta_{B}} \arrow{d}[swap]{\zeta_{B}}
  & B\sqcap B
  \\
  \Path(B) \arrow{ru}[swap]{\nu_{B}}
\end{tikzcd}
\end{equation}

\noindent
in which $\nu_{B}$ is a fibration and $\zeta_{B}$ is a weak equivalence in $\mathpzc{s}\mathcal{C}\mathpzc{om}\mathcal{A}\mathpzc{lg}(R)$. Then the following diagram is commutative:
\begin{equation*}
  \begin{tikzcd}
  A \arrow{r}{\zeta_{B}g} \arrow{d}[swap]{g} & \Path(B) \arrow{d}{\nu_{B}}
  \\
  B \arrow{r}{l} & B\sqcap B
\end{tikzcd}
\end{equation*}

\noindent
Indeed, we have
$$\pi_{1}^{B}lg=1^{B}g=\pi_{1}^{B}\Delta_{B}g=\pi_{1}^{B}\nu_{B}\zeta_{B}g$$
and
$$\pi_{2}^{B}lg=ghg=g1^{A}=g=1^{B}g=\pi_{2}^{B}\Delta_{B}g=\pi_{2}^{B}\nu_{B}\zeta_{B}g,$$
so it follows that $lg=\nu_{B}\zeta_{B}g$. As $\nu_{B}$ is a fibration, the above diagram can be completed to the following commutative diagram:
\begin{equation*}
  \begin{tikzcd}
  A \arrow{r}{\zeta_{B}g} \arrow{d}[swap]{g} & \Path(B) \arrow{d}{\nu_{B}}
  \\
  B \arrow{r}{l} \arrow{ru}{\psi} & B\sqcap B
\end{tikzcd}
\end{equation*}

\noindent
Now the following diagram is commutative:
\begin{equation*}
  \begin{tikzcd}
  B & B \arrow{l}[swap]{1^{B}} \arrow{r}{gh} \arrow{d}{\psi} & B
  \\
  & \Path(B) \arrow{lu}{\pi_{1}^{B}\nu_{B}} \arrow{ru}[swap]{\pi_{2}^{B}\nu_{B}}
\end{tikzcd}
\end{equation*}

\noindent
Indeed, we have $\pi_{1}^{B}\nu_{B}\psi=\pi_{1}^{B}l=1^{B}$ and $\pi_{2}^{B}\nu_{B}\psi=\pi_{2}^{B}l=gh$. Let $\mathcal{G}=\Nor\left(-^{\sharp}\right):\mathpzc{s}\mathcal{C}\mathpzc{om}\mathcal{A}\mathpzc{lg}(R)\rightarrow \mathcal{C}_{\geq 0}(R)$. Applying the functor $\mathcal{G}$ to the diagram \eqref{eq:6.6.1} yields the commutative diagram
\begin{equation*}
  \begin{tikzcd}
  \mathcal{G}(B) \arrow{r}{\eta\mathcal{G}(\Delta_{B})=\Delta_{\mathcal{G}(B)}} \arrow{d}[swap]{\mathcal{G}(\zeta_{B})}
  & [2.5em] \mathcal{G}(B)\sqcap \mathcal{G}(B)
  \\
  \mathcal{G}\left(\Path(B)\right) \arrow{ru}[swap]{\eta\mathcal{G}(\nu_{B})}
\end{tikzcd}
\end{equation*}

\noindent
in $\mathcal{C}_{\geq 0}(R)$ in which $\eta:\mathcal{G}(B\sqcap B)\rightarrow \mathcal{G}(B)\sqcap \mathcal{G}(B)$ is the natural isomorphism. Since $\nu_{B}$ is a fibration and $\zeta_{B}$ is a weak equivalence in $\mathpzc{s}\mathcal{C}\mathpzc{om}\mathcal{A}\mathpzc{lg}(R)$, we see by definition that $\eta\mathcal{G}(\Delta_{B})$ is a fibration and $\mathcal{G}(\zeta_{B})$ is a weak equivalence in $\mathcal{C}_{\geq 0}(R)$. This means that $\mathcal{G}\left(\Path(B)\right)$ is a path object for $\mathcal{G}(B)$. Moreover, the following diagram is commutative:
\begin{equation*}
  \begin{tikzcd}
  \mathcal{G}(B) & \mathcal{G}(B) \arrow{l}[swap]{1^{\mathcal{G}(B)}} \arrow{r}{\mathcal{G}(g)\mathcal{G}(h)} \arrow{d}{\mathcal{G}(\psi)} & \mathcal{G}(B)
  \\
  & \mathcal{G}\left(\Path(B)\right) \arrow{lu}{\pi_{1}^{\mathcal{G}(B)}\eta\mathcal{G}(\nu_{B})} \arrow{ru}[swap]{\pi_{2}^{\mathcal{G}(B)}\eta\mathcal{G}(\nu_{B})}
\end{tikzcd}
\end{equation*}

\noindent
Since $\mathcal{C}_{\geq 0}(R)$ is a model category by Theorem \ref{3.1}, the above diagram shows that $\mathcal{G}(g)\mathcal{G}(h)\sim_{r}1^{\mathcal{G}(B)}$. Applying the localization functor $\mathcal{L}_{R}:=\mathcal{L}_{\mathcal{C}_{\geq 0}(R)}:\mathcal{C}_{\geq 0}(R)\rightarrow \Ho\left(\mathcal{C}_{\geq 0}(R)\right)$, noting that $\mathcal{L}_{R}$ maps weak equivalences to isomorphisms, and invoking \cite[Lemma 8.3.4]{Hi}, we get $\mathcal{L}_{R}\left(\mathcal{G}(g)\right)\mathcal{L}_{R}\left(\mathcal{G}(h)\right)=1^{\mathcal{G}(B)}$. On the other hand, from the relation $hg=1^{A}$, we get $\mathcal{L}_{R}\left(\mathcal{G}(h)\right)\mathcal{L}_{R}\left(\mathcal{G}(g)\right)=1^{\mathcal{G}(A)}$. Therefore, $\mathcal{L}_{R}\left(\mathcal{G}(g)\right)$ is an isomorphism, so by Proposition \ref{2.7}, $\mathcal{G}(g)$ is a weak equivalence in $\mathcal{C}_{\geq 0}(R)$, so by definition, $g$ is a weak equivalence in $\mathpzc{s}\mathcal{C}\mathpzc{om}\mathcal{A}\mathpzc{lg}(R)$.

It now follows from Theorem \ref{2.9} that $\mathpzc{s}\mathcal{C}\mathpzc{om}\mathcal{A}\mathpzc{lg}(R)$ is a cofibrantly generated model category.
\end{proof}


\end{document}